\documentclass[a4paper]{article}

\usepackage[dvips]{graphics,graphicx}
\usepackage{amsfonts,amssymb,amsmath,color,mathrsfs, amstext,bm}
\usepackage{amsbsy, amsopn, amscd, amsxtra, amsthm,authblk, enumerate}
\usepackage{upref}
\usepackage[colorlinks,
            linkcolor=red,
            anchorcolor=red,
            citecolor=red
            ]{hyperref}

\usepackage{geometry}
\usepackage[displaymath]{lineno}
\usepackage{float}
\allowdisplaybreaks

\usepackage{yhmath}

\usepackage[normalem]{ulem}

\usepackage[ruled,linesnumbered]{algorithm2e}

\usepackage{authblk}

\usepackage{amscd}

\usepackage{subfigure}

\numberwithin{equation}{section}

\def\e{\mathrm{e}}

\def\E{\mathbb{E}}
\def\R{\mathbb{R}}

\newtheorem{theorem}{Theorem}
\newtheorem{lemma}{Lemma}

\newtheorem{proposition}{Proposition}
\newtheorem{remark}{Remark}

\newtheorem{definition}{Definition}

\begin{document}
\title{Asymptotic uniform estimate of random batch method with replacement for the Cucker-Smale model}
\author[a,b]{Shi Jin\thanks{shijin-m@sjtu.edu.cn}}
\author[a]{Yuelin Wang \thanks{sjtu$\_$wyl@sjtu.edu.cn}}
\author[c]{Yuliang Wang \thanks{yuliang.wang2@duke.edu}}
\affil[a]{School of Mathematical Sciences, Shanghai Jiao Tong University, Shanghai, 200240, China. }
\affil[b]{Institute of Natural Sciences and MOE-LSC, Shanghai Jiao Tong University, Shanghai 200240, China.}
\affil[c]{Department of Mathematics, Duke University, Durham, NC 27708, USA.}

\maketitle

\begin{abstract}
    The Random Batch Method (RBM) [S. Jin, L. Li and J.-G. Liu, Random Batch Methods (RBM) for interacting particle systems, J. Comput. Phys. 400 (2020) 108877] is not only an efficient algorithm for simulating interacting particle systems, but also a randomly switching networked model for interacting particle system. This work investigates two RBM variants (RBM-r and RBM-1) applied to the Cucker-Smale flocking model. We establish the asymptotic emergence of global flocking and derive corresponding error estimates. By introducing a crucial auxiliary system and leveraging the intrinsic characteristics of the Cucker-Smale model, and under suitable conditions on the force, our estimates are uniform in both time and particle numbers. In the case of RBM-1, our estimates are sharper than those in Ha et al. (2021). Additionally, we provide numerical simulations to validate our analytical results.\\
    

\end{abstract}

\section{Introduction}

        Collective behaviors in many-body systems are prevalent in the natural world, such as the flocking of birds \cite{ha2009simple, Ha2008particle, cucker2007emergent, motsch2011new}, swarming of fish \cite{toner1998flocks}, synchronicity of fireflies \cite{kuramoto1975self, buck1966biology}, and the behavior of pacemaker cells \cite{peskin1975mathematical}. We use the term ``flocking" to describe the process by which self-propelled particles organize into coordinated motion, based solely on limited environmental information and simple rules \cite{toner1998flocks}. For a non-exhaustive list of   literature on collective behaviors and related models, we  refer to  \cite{degond2008continuum, justh2002simple, albi2019vehicular, bellomo2017quest, bellomo2012mathematical, vicsek1995novel, ha2016collective, dorfler2014synchronization, topaz2004swarming} and references therein. 

    The Cucker-Smale model, introduced by Cucker and Smale \cite{cucker2007emergent}, is a well-known model of collective behavior that phenomenologically describes flockings. It is formulated as an $N$-body second-order system of ordinary differential equations that govern the positions and velocities of particles, resembling Newton's laws of motion. Let $X^i$ and $V^i$ be the position and velocity of the $i$-th particle with unit mass, and $\psi(| X^j -X^i|)$ be the communication weight between the $j$-th and $i$-th particles. The Cucker-Smale model reads as the following
    \begin{equation}\label{eq: cs}
        \left\{\begin{aligned}
            \frac{d}{dt} X^i(t) =& V^i(t),\quad i=1,\cdots, N,\\
            \frac{d}{dt} V^i(t) =& \frac{1}{N-1}\sum\limits_{j=1}^N 
            \psi(|X^j(t)-X^i(t)|)(V^j(t)-V^i(t)),\\
            X^i(0) =& X^{in,i},\quad V^i(0)=V^{in,i},
        \end{aligned}\right.
    \end{equation}
    where the nonnegative coupling strength $\psi$ satisfies positivity, boundedness, Lipschitz continuity and mononticity conditions, i.e., there exist positive constants $\psi_0, \psi_M  > 0 $  such that 
    \begin{equation}\label{ass:psi}
        0< \psi_0 \le \psi(r) \le \psi_M,\, \forall r\ge0; \quad \|\psi\|_\text{{Lip}}<\infty;\quad (\psi(r_1)-\psi(r_2))(r_1-r_2)\le 0,\, r_1,r_2 \in \mathbb{R}_+.
    \end{equation}
    

    In this model, each particle interacts with $N-1$ other particles, resulting in a computational cost of $\mathcal{O}(N^2)$ per time step. To address this complexity, the Random Batch Method (RBM) is proposed by Jin et al. in 2020 \cite{jin2020random}, providing an efficient algorithm that reduces the computational cost to $\mathcal{O}(N)$. The RBM constructs a randomly decoupled system comprised of subsystems that interact among $p$ particles, where $p\ll N$. At each time step, interactions occur only within small batches of $p$ particles. Therefore, the RBM-approximation system for \eqref{eq: cs} becomes a randomly switching networked system \cite{dong2020stochastic} along time. The random selection of batches allows for a time-averaged effect, making this approach a good approximation of the original system \cite{jin2020random, ha2021uniform, jin2022random}.

    In the original work \cite{jin2020random}, the authors proposed two random methods, named RBM-1 (the RBM without replacement) and RBM-r (the RBM with replacement), and provided an error estimate for the former. These methods approximate the time evolution of large particle systems by decomposing the full system into smaller, randomly coupled subsystems, thereby reducing computational complexity. 
    The main distinction between the RBM-r and the RBM-1 lies in the use of various random methods to select particles. The algorithmic details are elaborated in Section \ref{subsec:rbmrrbm1}. In RBM-1, the entire system is randomly divided into $\lceil N/p \rceil$ batches of size $p$, and each subsystem evolves simultaneously and independently. Subsequent studies have primarily focused on analyzing RBM-1, see the survey article \cite{jin2021randomreview}. In particular, several works have addressed error estimates for the Cucker-Smale model and generalized consensus models \cite{ha2021uniform, ko2021uniform, ha2021convergence}. 
    
    Differently, the RBM-r randomly selects a batch of size $p$ during each time step, allowing only the particles within this batch to interact briefly. This dynamic approach in RBM-r is reminiscent of the selection-interaction philosophy found in the kinetic Monte Carlo (KMC) method \cite{bortz1975new,voter2007introduction} for first-order pairwise interacting particle systems. It can be viewed as a generalization of the Bird algorithm in \cite{albi2013binary} for the mean-field equation. From a numerical simulation perspective, RBM-r is simpler to implement and has applications in various fields, including quantum simulation \cite{jin2020randomQMC}, molecular dynamics \cite{liang2021random,gao2024rbmd}, and enhanced Monte Carlo sampling method \cite{li2020random}.

    However, to our knowledge, only one recent study has analyzed RBM-r \cite{cai2024convergence}. 
    In \cite{cai2024convergence}, the authors provide a convergence analysis for the RBM-r approximation of the first-order system in the Wasserstein-2 distance. For the deterministic interaction particle system, they obtained  a rate of $\mathcal{O}(\sqrt{1+T}(N/p)^{\frac{3}{4}}\tau^{\frac{1}{2}})$, which depends on the particle number $N$, the time step $\tau$ and time $T$. Inspired by this work, we discuss the asymptotic flocking dynamics and the convergence of the Cucker-Smale model based on the RBM-r approximation.

    In this paper, our main contribution lies on two aspects.
    \begin{itemize}
        \item We present the stochastic flocking analysis of the RBM systems.
        \item We present uniform-in-time error estimates of $\mathcal{O}(\tau)$ with exponential decay.
    \end{itemize}


    The primary difficulty stems from the fundamental differences between the Cucker-Smale model and the Langevin dynamics, the latter studied in \cite{cai2024convergence}. Unlike the Langevin dynamics, which relies on contraction properties (as discussed in \cite{jin2020random,cai2024convergence}), the analysis of the Cucker-Smale model depends on proving asymptotic flocking behavior—a property that requires separate verification.

    Moreover, the analysis of RBM-r is quite different from prior works on RBM-1. The RBM-r allows for replacement, meaning it employs independent selection of random batches, which is a significant departure from RBM-1. Consequently, while both methods maintain the same computational cost during their effective periods, the RBM-r possesses a different effective time of $\frac{N}{p}t$ compared to $t$ of the RBM-1. Thus, the techniques used in estimating the RBM-1 for the Cucker-Smale model \cite{cai2024convergence} cannot be directly applied to RBM-r. 
    
    To address this issue, we first present a stochastic flocking analysis of the RBM-r system. By the crucial observation of the momentum conservation, we prove that the RBM-r (Theorem \ref{thm:flocking}) shares the similar flocking dynamics with the original Cucker-Smale system.:
    $$\E \frac{1}{N^2}\sum\limits_{i,j} |\tilde{V}^i(\frac{N}{p}t)-\tilde{V}^j(\frac{N}{p}t)|^2 \le  \frac{1}{N^2}\sum\limits_{i,j} |\tilde{V}^i(0)-\tilde{V}^j(0)|^2 \exp\left(-\frac{2N}{N-1} \frac{\psi_0}{1+\frac{2p}{p-1}\psi_0\tau}t\right).$$
    
    Building upon the flocking analysis, we establish our second main result: the convergence analysis.  Inspired by \cite{cai2024convergence}, we employ an auxiliary system as an intermediate bridge to connect the original system \eqref{eq: cs} with the RBM-r approximation. (See Section \ref{subsec:aux sys} in the main text.) Utilizing the exchangeability of the particles and combinatorial tools, we derive a uniform-in-time error estimate with exponential decay (Theorem \ref{thm:main}): 
    \begin{equation*}
        \E \frac{1}{N}\sum\limits_{i=1}^N\left|\tilde{V}^i(\frac{N}{p}t) - V^i(t)\right|^2 \le C\tau \left(1-\frac{p}{N}+\frac{1}{p-1}-\frac{1}{N-1}\right)\e^{-C_3 t} + C\tau^2 \e^{-\frac{C_3}{2} t}.
    \end{equation*}
    Here, $C_1$ and $C_3$ are constants defined in \eqref{ntt:C1C2}, and $C$ is a constant depending on $\psi,$ $D(X^{in})$ and $D(V^{in})$.
    
    In addition, our method can similarly be applied to the analysis of the RBM-1, yielding the flocking emergence (Theorem \ref{thm:flockingrbm1}) and error estimate of the RBM-1:
    \begin{equation*}
        \E \frac{1}{N}\sum\limits_{i=1}^N\left|\tilde{V}^i_{(1)}(t) - V^i(t)\right|^2 \le  C\tau \left(\frac{1}{p-1}-\frac{1}{N-1}\right)\e^{-C_1 t} +  C\tau^2\e^{-C_1 t},
    \end{equation*}
    which improves the estimate of $\mathcal{O}(e^{-Ct} + \tau \left(\frac{1}{p-1}-\frac{1}{N-1}\right) + \tau^2)$ in the prior work \cite{ha2021uniform}.

    The rest of this paper is organized as follows. In Section \ref{sec:preliminaries}, we provide preliminary information, including the properties of the Cucker-Smale model, details of the RBMs, and an introduction to a key auxiliary dynamical system. In Section \ref{sec:mainresult}, we present our main theorems, where the proofs will by given in Section \ref{sec: pf of thms}. For better organization, we leave the auxiliary lemmas in Section \ref{sec:auxiliary lemmas}. Section \ref{sec:numerical} offers numerical simulations to evaluate the error over time evolution. Finally, Section \ref{sec:conclusion} is devoted to a summary of our main results and some remaining issues to be explored in the future.
    
    \paragraph{Notation:} For readers' convenience, we give a list of notations here.
    We set
    \begin{align*}
        &X^i := (X^{i_1},\cdots,X^{i_d})\in\R^{d}, \,i=1\cdots,N ,\quad X:=(X^1,\cdots,X^N)\in\R^{Nd},\\
        &V^i := (V^{i_1},\;\cdots,V^{i_d})\in\R^{d}, \,i=1\cdots,N ,\quad V:= (\,V^1,\:\cdots,V^N)\in\R^{Nd}.
    \end{align*}
    For the sake of notation simplicity, we also denote
    $$Z^i = (X^i,V^i)\in\mathbb{R}^{2d}, \,i=1\cdots,N ,\quad Z:= (X,V)\in \mathbb{R}^{2Nd}.$$    
    We use the following handy notation:
    \begin{equation*}
        \sum\limits_i := \sum\limits_{i=1}^N, \quad \sum\limits_{i,j} = \sum\limits_{i=1}^N\sum\limits_{j=1}^N,\quad \max\limits_i: = \max\limits_{1\le i\le N}.
    \end{equation*}
    Set the constants
    \begin{equation}\label{ntt:C1C2}
        C_1:=\frac{2\psi_0}{1+\frac{2p}{p-1}\psi_0\tau}, \quad C_2:= \psi_0(1-\psi_0 \tau),\quad C_3:=\min\{C_1,\, 2C_2\}.
    \end{equation}
    
    Denote the diameters of compact support in spatial and velocity variables at time $t$ by $\mathcal{D}_X(t)$ and $\mathcal{D}_V(t)$ respectively, i.e.,
    \begin{equation}\label{eq:3.13}
        \mathcal{D}_X(t) := \max\limits_{i,j}\| X^i-X^j\|_2, \quad \mathcal{D}_V(t) := \max\limits_{i,j}\| V^i-V^j\|_2.
    \end{equation}
    
    Also, we use several auxiliary discrete time dependent random quantities. More details see Section \ref{subsec:aux sys}. We denote the stopping time $\zeta_n^i$ the $n$-th time the particle with index $i$ is chosen into the batch:
    \begin{equation*}
        \zeta_n^i :=\inf \left\{K : \sum\limits_{k=0}^K \mathbb{I}_{\{i\in\mathcal{C}_k\}}\tau \ge t \right\}, \,\forall n>0; \quad \zeta_0^i :=0.
    \end{equation*}
    Based on the definition of $\zeta_n^i$, we denote the total number of times that the index $i$ is selected before time $t\in[t_k,t_{k+1})$:
    \begin{equation*}
        \eta_t^i:= \left\{\begin{aligned}
            &\sup \{ n; \zeta_n^i \tau < t, \, n\in \mathbb{N}\},\quad &i\notin \mathcal{C}_k,\\
            &\sup \{ n; \zeta_n^i \tau < t, \, n\in \mathbb{N}\}+1,\quad &i\in \mathcal{C}_k.
        \end{aligned}\right.
    \end{equation*}
    The time period during which the particle $i$ is chosen is denoted to be
    \begin{equation*}
        t^{(i)} := \left\{\begin{aligned}
            &\eta_t^{i}\tau + t-t_{k+1} , \quad i \in \mathcal{C}_k,\\
            & \eta_t^{i} \tau, \quad i \notin \mathcal{C}_k.
        \end{aligned}\right.
    \end{equation*}

\section{Preliminaries}\label{sec:preliminaries}

In this section, we outline the properties of the Cucker-Smale model and provide a brief introduction to the Random Batch Method.

\subsection{Properties.}
As mentioned in the introduction, the Cucker-Smale model is inspired by the collective behaviors of birds and fishes in the natural world. It possesses several properties that align with physical intuition. Below, we present several key attributes of the Cucker-Smale model.

\paragraph{Conservation law.}
It is easy to see that Equation \eqref{eq: cs} conserves the first-order momentum, i.e.,
\begin{equation*}
    \sum\limits_{i=1}^N V^i(t)= \sum\limits_{i=1}^N V^{in,i}.
\end{equation*}
Additionally, the total energy does not increase as time progresses.

\begin{proposition}
Let the $\{(X^i,V^i)\},$ $1\le i\le N,$ be the solution of  system \eqref{eq: cs}. Then for any $t>0,$ the total momentum is conserved as a constant and the total energy is nonincreasing.
\end{proposition}
\begin{proof}
    The proof is straightforward since
    $$\frac{d}{dt}\sum\limits_{i=1}^N V^i (t) = 0,$$ and
    $$\frac{d}{dt}\sum\limits_{i=1}^N |V^i(t)|^2 = -\frac{1}{N-1}\sum\limits_{i,j}\psi(|X^j(t) - X^i(t)|) |V^j (t) - V^i (t) |^2.$$
\end{proof}

\paragraph{Translation-invariance.}
 By straightforward calculus, it's easy to find that  system \eqref{eq: cs} is invariant under the translation $V^{i,\prime}=V^i + c$ for any constant vector $c.$ Therefore, we set $\sum\limits_{i=1}^N V^i(0)=0$ in this paper without loss of generality.

\paragraph{Flocking.}
The most notable characteristic of the Cucker-Smale model is asymptotic flocking, which illustrates the emergence of fundamental collective behavior. This concept is defined  below.
    \begin{definition}\label{def:flocking}
        Let $(X,V)$ be a solution to \eqref{eq: cs}. Then $(X,V)$ exhibits asymptotic flocking if the following relations hold,
        \begin{equation*}
            \sup\limits_{0<t<\infty} |X^i(t)-X^j(t)|<\infty, \quad \lim\limits_{t\to\infty} |V^i(t)-V^j(t)|=0, \quad 1<i,j<N.
        \end{equation*}
    \end{definition}
Previously the following asymptotic flocking estimate eas established. 
\begin{proposition}[\cite{ha2021uniform}, Lemma 2.2, Proposition 2.1]\label{lem:V2diff}
    Let $(X,V)$ be a solution to \eqref{eq: cs} with a zero-sum condition:
    $$\sum\limits_{i=1}^N V^i(t) = 0, \quad t\ge 0.$$ Suppose that the coupling strength and the initial data $(X^{in},V^{in})$ satisfy
    \begin{equation*}
        D(V^{in})< \frac{1}{2}\int_{D(X^{in})}^\infty \psi(s)ds.
    \end{equation*}
    Then there exists a positive constant $x_\infty$ such that 
    \begin{equation*}
        \sup\limits_{t}D_X(t) \le x_\infty, \quad D_V(t) \le D(V^{in})\e^{-\psi(x_\infty)t}, \quad t\ge 0.
    \end{equation*}
\end{proposition}
In particular, under the assumption \eqref{ass:psi} on $\psi,$ for any $t_2\le t_1,$ by direct computation, one obtains 
    \begin{equation*}
         |V^i(t_2)-V^i(t_1)|^2 \le\psi_M^2 (t_2-t_1)^2 D(V^{in})\e^{-2\psi_0 t_1}.
    \end{equation*}

\subsection{The Random batch method with replacement.}\label{subsec:rbmrrbm1}

    Now we present the details of the RBM. In each sub-time interval $[t_k, t_{k+1}),$ particle $i$ only interacting with those within a specific, randomly selected or partitioned small batch that includes particle $i$. The varying random division approaches lead to the RBM-1 and the RBM-r. The RBM-1 evenly divides all particles into batches of size $p$ at each step and computes the dynamics for each batch subsystem. In contrast, RBM-r randomly selects a  batch of size $p$ and only compute the dynamics of the chosen subsystem. A notable feature is that, in a single time step, the computational effort of RBM-r is $p/N$ times that of RBM-1. It can be conjectured that the results of RBM-r after $\lceil N/p \rceil$ steps will converge to those of RBM-1 in expectation. Additionally, in the $N/p$ steps, RBM-r may select certain same particle $i$ multiple times, whereas in RBM-1, particle $i$ can only be selected once per time step. This is why RBM-r is referred to as RBM with replacement.
    
    We set the initial data $(\tilde{X}^i(0),\tilde{V}^i(0)) = (X^{in,i},V^{in,i}),$ for $ i = 1,\cdots,N$, uniformly. We illustrate RBM-1 and RBM-r as examples derived from the Cucker-Smale model. For details, see Algorithm \ref{algo: the RBM1} and Algorithm \ref{algo: the RBMr}, respectively. 
    This comparison highlights the differences between the two methods, using the initial data from system \eqref{eq: cs}.

\begin{algorithm}\label{algo: the RBM1}
\caption{ The RBM-1 for \eqref{eq: cs}}
\For{$k=1$ to $T/\tau$}
    {Divide $\{1,\cdots, N\}$ into $n = N/p$ batches randomly. Denote all the batches $\xi_k=(\xi_k(1),\cdots, \xi_k(\frac{N}{p}))$\;
    
    \For{each batch belonging to $\xi_k$ do}
        {
        Update $(\tilde{X}^i,\tilde{V}^i)$ in $\mathcal{C}_q$ by solving 
        \begin{equation}\label{eq: the RBM-1cs}
            \left\{\begin{aligned}
                \partial_t \tilde{X}^i(t) =& \tilde{V}^i(t),\\
                \partial_t \tilde{V}^i(t) =& \frac{1}{p-1}\sum\limits_{j\in\mathcal{C}_q}
                \psi(|\tilde{X}^j(t)-\tilde{X}^i(t)|)(\tilde{V}^j(t)-\tilde{V}^i(t)),
            \end{aligned}\right.
        \end{equation}
        for $t\in [t_k,t_{k+1}).$ 
        }
    }
\end{algorithm}

\begin{algorithm}\label{algo: the RBMr}
\caption{ The RBM-r for \eqref{eq: cs}}
\For{$k=1$ to $\frac{NT}{p\tau}$}
    {Pick a batch $\mathcal{C}_k$ of size $p$ randomly.
    Update $(\tilde{X}^i,\tilde{V}^i)$ in $\mathcal{C}_k$ by solving 
    \begin{equation}\label{eq: the RBMcs}
        \left\{\begin{aligned}
            \partial_t \tilde{X}^i(t) =& \tilde{V}^i(t),\\
            \partial_t \tilde{V}^i(t) =& \frac{1}{p-1}\sum\limits_{j\in\mathcal{C}_k}
            \psi(|\tilde{X}^j(t)-\tilde{X}^i(t)|)(\tilde{V}^j(t)-\tilde{V}^i(t)),
        \end{aligned}\right.
    \end{equation}
    for $t\in [t_k,t_{k+1}).$ 
    }
\end{algorithm}


    For an RBM-r system described by \eqref{eq: the RBMcs} with a given random batch division, the flocking property cannot be directly derived. However, the boundedness of velocity can still be established, as stated in the following proposition.  
    \begin{proposition}\label{ineq: coarseD}
        Let $\tilde{X},\tilde{V}$ be a solution of the RBM-r model (Algorithm \ref{algo: the RBMr}) with initial data $f_0$ satisfying 
        $$
        \mathcal{D}(X^{in}) + \mathcal{D}(V^{in})<\infty.
        $$
        Then one has $\mathcal{D}_{\tilde{V}} (t) \le \mathcal{D}(V^{in}).$
    \end{proposition}
    Proposition \ref{ineq: coarseD} was first introduced as Lemma 2.4 in \cite{ha2021uniform} for the RBM-1 system.
    We prove it for the RBM-r in Appendix \ref{app:A}. Note that this estimate works on any fixed sequence of batches.

\subsection{Auxiliary dynamics.}\label{subsec:aux sys}
    For better organization, we present three dynamics that will be utilized in the proof of Theorem \ref{thm:main}. 
    
    Due to the differences between RBM-r and RBM-1 mentioned above, analyzing RBM-r directly can be challenging. To address this issue, we introduce an intermediate dynamics, consisting of $N$ copies of the original system after a random time change. This insight is crucial. We refer to this intermediate system as IPS', in contrast to the original interacting particle system as IPS.

    In details, for the Cuker-Smale model \eqref{eq: cs}, we define the following dynamic triple $(\tilde{Z},\hat{Z},Z)$ where $Z:=(X,V)$ denotes a pair of position and velocity, with the initial state
    \begin{equation*}
        \tilde{Z}^\ell(0) = \hat{Z}^{\ell i}(0) = Z^\ell(0), \quad \forall 1\le \ell,i\le N.
    \end{equation*}
    We list the three dynamical systems below.
    \paragraph{RBM-r:} For $1\le i \le N,$ $t\in [t_k,t_{k+1}),$
     \begin{equation}\label{eq:rbmr}
        \left\{\begin{aligned}
            \partial_t \tilde{X}^i(t) =& \tilde{V}^i(t),\\
            \partial_t \tilde{V}^i(t) =& \frac{1}{p-1}\sum\limits_{j\in\mathcal{C}_k}
            \psi(|\tilde{X}^j(t)-\tilde{X}^i(t)|)(\tilde{V}^j(t)-\tilde{V}^i(t)),
        \end{aligned}\right. \quad \text{if } i\in \mathcal{C}_k,
    \end{equation}
    and
    \begin{equation}\label{eq:rbmr2}
        \partial_t \tilde{X}^i(t) = 0,\quad
            \partial_t \tilde{V}^i(t) = 0, \quad \text{if } i\notin \mathcal{C}_k.
    \end{equation}
    \paragraph{IPS (\eqref{eq: cs}):} For $1\le \ell \le N,$ $t\in [t_k,t_{k+1}),$
    \begin{equation*} 
    \left\{\begin{aligned}
        \frac{d}{dt} X^\ell(t) =& V^\ell(t),\\
        \frac{d}{dt} V^\ell(t) =& \frac{1}{N-1}\sum\limits_{j=1}^N 
        \psi(|X^j(t)-X^\ell(t)|)(V^j(t)-V^\ell(t)).
    \end{aligned}\right.
    \end{equation*}
    \paragraph{IPS':} For $1\le \ell,i \le N,$ $t\in [t_k,t_{k+1}),$
    \begin{equation}\label{eq:IPS'}
    \left\{\begin{aligned}
        \frac{d}{dt} \hat{X}^{\ell i}(t) =& \hat{V}^{\ell i}(t),\\
        \frac{d}{dt} \hat{V}^{\ell i}(t) =& \frac{1}{N-1}\sum\limits_{j=1}^N 
        \psi(|\hat{X}^{ji}(t)-\hat{X}^{\ell i}(t)|)(\hat{V}^{ji}(t)-\hat{V}^{\ell i}(t)),
    \end{aligned}\right. \text{if } i\in \mathcal{C}_k,
    \end{equation}
    and
    \begin{equation*}
        \partial_t \hat{X}^{\ell i}(t) = 0,\quad
            \partial_t \hat{V}^{\ell i}(t) = 0,\quad \text{if } i\notin \mathcal{C}_k.
    \end{equation*}
    
    Under our assumption of the zero mean of initial data $\{V^{in,i}\}_{i=1}^N$, all the above systems of the RBM-r \eqref{eq:rbmr}, IPS \eqref{eq: cs} and IPS' \eqref{eq:IPS'} preserve the first velocity momentum.

    \begin{proposition}\label{lem:M}
        The first velocity momentum of RBM-r \eqref{eq:rbmr}, IPS \eqref{eq: cs}  and 
        IPS' \eqref{eq:IPS'} are conserved, i.e.
        \begin{equation*}
            \sum\limits_{i=1}^N \tilde{V}^i(t) = \sum\limits_{i=1}^N V^i(t) = \sum\limits_{i=1}^N V^{in,i}, \text{ and }\, \sum\limits_{\ell =1}^N \hat{V}^{\ell i} = \sum\limits_{\ell=1}^N V^{in,\ell}, \text{ for any } 1\le i\le N.
        \end{equation*}
    \end{proposition}
    \begin{proof}
        First, for $t\in[t_k, t_{k+1}),$ 
        \begin{equation}\label{eq:LemM1}
            \frac{d}{dt}\sum\limits_{i\in \mathcal{C}_k} \tilde{V}^i(t) = \sum\limits_{i\in\mathcal{C}_k}\frac{1}{p-1} \sum\limits_{j\in\mathcal{C}_k}\psi (|\tilde{X}^j(t)-\tilde{X}^i(t)|)(\tilde{V}^j(t)-\tilde{V}^i(t)) = 0,
        \end{equation}
        by the exchangability of the particles, Then, one has
        \begin{equation*}
            \sum\limits_{i=1}^N \tilde{V}^i(t) = \sum\limits_{i\in \mathcal{C}_k}\tilde{V}^i(t) + \sum\limits_{i\notin \mathcal{C}_k}\tilde{V}^i(t) = \sum\limits_{i\in \mathcal{C}_k}\tilde{V}^i(t) + \sum\limits_{i\notin \mathcal{C}_k}\tilde{V}^i(t),
        \end{equation*}
        because of \eqref{eq:LemM1} and \eqref{eq:rbmr2}. The proofs of IPS \eqref{eq: cs} and IPS' \eqref{eq:IPS'} are similar.
    \end{proof}
    Note that in IPS', we create $N$ copies of the original system IPS, adjusting the ``run/static" time based on whether a certain particle is chosen or not. To represent this random time change, we introduce the stopping time $\zeta_n^i$, the $n$-th time the particle with index $i$ is chosen into the batch:
    \begin{equation}\label{ntt:zeta}
        \zeta_n^i :=\inf \{K : \sum\limits_{k=0}^K \mathbb{I}_{\{i\in\mathcal{C}_k\}}\tau \ge t \}, \,\forall n>0; \quad \zeta_0^i :=0.
    \end{equation}
    Based on the definition of $\zeta_n^i$, we denote the total number of times that the index $i$ is selected before time $t\in[t_k,t_{k+1})$:
    \begin{equation}\label{ntt:eta}
        \eta_t^i:= \left\{\begin{aligned}
            &\sup \{ n; \zeta_n^i \tau < t, \, n\in \mathbb{N}\},\quad &i\notin \mathcal{C}_k,\\
            &\sup \{ n; \zeta_n^i \tau < t, \, n\in \mathbb{N}\}+1,\quad &i\in \mathcal{C}_k.
        \end{aligned}\right.
    \end{equation}
    The time period during which the particle $i$ is chosen is denoted to be
    \begin{equation}\label{ntt:ti}
        t^{(i)} := \left\{\begin{aligned}
            &\eta_t^{i}\tau + t-t_{k+1} , \quad i \in \mathcal{C}_k,\\
            & \eta_t^{i} \tau, \quad i \notin \mathcal{C}_k.
        \end{aligned}\right.
    \end{equation}

\section{The main theorem}\label{sec:mainresult}

In this section, we present the main results of this paper. For better organization, we leave the proofs to Section \ref{sec: pf of thms}.

Recall the notations $X:=(X^1,\cdots,X^N)\in\R^{Nd}$ and $ V:= (V^1,\cdots,V^N)\in\R^{Nd}$. In addition, we take $\Omega$ as the sample space equipped with the uniform probability measure $\mathbb{P}$, and define the filtration $\{\mathcal{F}_n\}_{n\ge 0}$, where $\mathcal{F}_n$ is the $\sigma$-algebra generated by $\{\xi_j,\,j\le n\}$. In the following, we will use the symbol $\E$ to indicate expectation over this probability space.

\subsection{Flocking dynamics of the RBM Cucker-Smale model}
First, according to Definition \ref{def:flocking}, we formalize the concept of stochastic flocking for the RBM systems.
\begin{definition}\label{def:stochastic flocking}
        Let $(\tilde{X},\tilde{V})$ be the solution to \eqref{eq:rbmr} or \eqref{eq: the RBMcs}. Then $(X,V)$ exhibits asymptotic flocking if the following relations hold,
        \begin{equation*}
            \sup\limits_{0<t<\infty} \E|\tilde{X}^i(t)-\tilde{X}^j(t)|^2<\infty, \quad \lim\limits_{t\to\infty} \E|\tilde{V}^i(t)-\tilde{V}^j(t)|^2=0, \quad 1<i,j<N.
        \end{equation*}
    \end{definition}
Then, we provide the emergence of flocking dynamics to both the RBM-r and RBM-1 Cucker-Smale models. 
\begin{theorem}[Stochastic flocking dynamics of the RBM-r]\label{thm:flocking}
Suppose that the communication weight satisfies the assumption \eqref{ass:psi}, and let $(\tilde{X},\tilde{V})$ be the solutions to \eqref{eq:rbmr}. Then, there exists a positive constant $\tilde{x}_\infty= \tilde{x}_\infty(\psi, D(X^{in}), D(V^{in}))$ such that
\begin{enumerate}
    \item $\sup\limits_{t>0} \E \left( \frac{1}{N^2}\sum\limits_{i,j}|\tilde{X}^i(\frac{N}{p}t) - \tilde{X}^j(\frac{N}{p}t)|^2\right) < \tilde{x}_\infty$,\\
    \item $\E \frac{1}{N^2}\sum\limits_{i,j} |\tilde{V}^i(\frac{N}{p}t)-\tilde{V}^j(\frac{N}{p}t)|^2 \le  \frac{1}{N^2}\sum\limits_{i,j} |\tilde{V}^i(0)-\tilde{V}^j(0)|^2 \exp\left(-\frac{2N}{N-1} \frac{\psi_0}{1+\frac{2p}{p-1}\psi_0\tau}t\right).$
\end{enumerate}
\end{theorem}

In our theorem, the decay rate is {\it independent} of $N$ and $p$. Similar with the proof technique in Theorem \ref{thm:flocking}, we can derive a flocking estimate of the RBM-1. This is a notable improvement from \cite[Theorem 3.1]{ha2021uniform}, where the decay rate of the RBM-1 is $\mathcal{O}(p/N)$ under the same assumption.

\begin{theorem}[Stochastic flocking dynamics of the RBM-1]\label{thm:flockingrbm1}
Suppose that the communication weight satisfies the assumption \eqref{ass:psi}, and let $(\tilde{X}_{(1)},\tilde{V}_{(1)})$ be the solutions to \eqref{eq: the RBM-1cs}. Then, there exists a positive constant $\tilde{x}^{(1)}_\infty= \tilde{x}^{(1)}_\infty(\psi, D(X^{in}), D(V^{in}))$ such that
\begin{enumerate}
    \item $\sup\limits_{t>0} \E \Big( \frac{1}{N^2}\sum\limits_{i,j}|\tilde{X}^i_{(1)}(t) - \tilde{X}^j_{(1)}(t)|^2\Big) < \tilde{x}^{(1)}_\infty$,\\
    \item $\E \frac{1}{N^2}\sum\limits_{i,j} |\tilde{V}_{(1)}^i(t)-\tilde{V}_{(1)}^j(t)|^2 \le  \frac{1}{N^2}\sum\limits_{i,j} |\tilde{V}_{(1)}^i(0)-\tilde{V}_{(1)}^j(0)|^2 \exp\Big(-\frac{2N}{N-1} \frac{\psi_0}{1+\frac{2p}{p-1}\psi_0\tau}t\Big).$
\end{enumerate}
\end{theorem}

\begin{remark}\label{rmk:N=N}
    Recall $(X,V)$ is the solutions to \eqref{eq: cs}. By direct calculus similar with Lemma \ref{lem:basic flocking}, one can find that 
    \begin{equation*}
        \E \frac{1}{N^2}\sum\limits_{i,j} |V^i(t)-V^j(t)|^2 \le \frac{1}{N^2}\sum\limits_{i,j}|V^i(0)-V^j(0)|^2 \exp\left( -2\frac{N}{N-1}\psi_0 t\right).
    \end{equation*}
    For the decay coefficient of the exponential functions in Theorem \ref{thm:flocking} and Theorem \ref{thm:flockingrbm1}, we observe that
    \begin{equation*}
        \exp\left(-\frac{2N}{N-1} \frac{\psi_0}{1+\frac{2p}{p-1}\psi_0\tau}t\right) \to \exp\left( -2\frac{N}{N-1}\psi_0t\right), \text{ as } \tau \to 0.
    \end{equation*}
    In this sense, both the two RBM methods share the similar flocking dynamics with the original system \eqref{eq: cs}. 
\end{remark}
\begin{remark}
    The assumption of the lower bound $\psi_0$ is for the convenience of the proof. By the Lyapunov approach in \cite{ha2009simple}, this assumption can be relaxed if the coupling strength and the initial data $(X^{in},V^{in})$ satisfy
    \begin{equation*}
        D(V^{in}) < \frac{1}{2}\int^{\infty}_{D(X^{in})}\psi(s)ds.
    \end{equation*}
\end{remark}

\subsection{Uniform error estimate of the RBM-r}
We now present our second main result, which establishes an asymptotic error estimate between the RBM-r system \eqref{eq:rbmr} and the original Cucker-Smale system \eqref{eq: cs}.

\begin{theorem}[Uniform error estimate]\label{thm:main}
Suppose that the communication weight satisfies  assumption \eqref{ass:psi}, and let $(X,V),$ $(\tilde{X},\tilde{V})$ be the solutions to \eqref{eq: cs} and \eqref{eq:rbmr}, respectively. Then it holds 
    \begin{equation}\label{eq:thmerr}
        \E \frac{1}{N}\sum\limits_{i=1}^N\left|\tilde{V}^i(\frac{N}{p}t) - V^i(t)\right|^2 \le C\tau \left(1-\frac{p}{N}+\frac{1}{p-1}-\frac{1}{N-1}\right)\e^{-C_3 t} + C\tau^2 \e^{-\frac{C_3}{2} t},
    \end{equation}
    where $C$ is a constant depending on $\psi,$ $D(X^{in})$ and $D(V^{in})$, and $C_3$ is defined in \eqref{ntt:C1C2}.
\end{theorem}

\begin{remark}\label{rmk:C1 N/N-1}
    To achieve a decay rate independent of $N$, we define $C_1:= \frac{2\psi_0}{1+\frac{2p}{p-1}\psi_0 \tau}$ in \eqref{ntt:C1C2}. However, $C_1$ is actually derived from Theorem \ref{thm:flocking} and Theorem \ref{thm:flockingrbm1} (since $(N-1)/N \ge 1$) and thus, in the main text, it can be replaced by $\frac{N}{N-1} \frac{2\psi_0}{1+\frac{2p}{p-1}\psi_0 \tau}$.
\end{remark}

\begin{remark}\label{thm:errrbm1}
    Let $(\tilde{X}_{(1)}, \tilde{V}_{(1)})$ be the solutions to the RBM-1 (Algorithm \ref{algo: the RBM1}). Then, under the assumption of Theorem \ref{thm:main}, one can derive
    \begin{equation}\label{eq:thmerrrbm1}
        \E \frac{1}{N}\sum\limits_{i=1}^N\left|\tilde{V}^i_{(1)}(t) - V^i(t)\right|^2 \le C\tau \left(\frac{1}{p-1}-\frac{1}{N-1}\right)\e^{-C_1 t}+C\tau^2 \e^{-C_1 t},
    \end{equation} 
    where $C=C(\psi,D(X^{in}), D(V^{in}))$.
    Since the proof follows similarly (and simpler) with using Step 1 of Theorem \ref{thm:main}, we omit it in this paper. This represents a mild refinement of the earlier bound 
    \begin{equation*}
        \E \frac{1}{N}\sum\limits_{i=1}^N\left|\tilde{V}^i_{(1)}(t) - V^i(t)\right|^2 \le C \e^{-\psi_0 t} +C\tau^2 + C\tau \left(\frac{1}{p-1}-\frac{1}{N-1}\right),
    \end{equation*}
    in \cite[Theorem 3.2]{ha2021uniform} by Ha et al. The key improvement stems from a more precise decay analysis leveraging flocking estimates in Theorem \ref{thm:flockingrbm1}, rather than relying only on the coarse bound of $D(\tilde{V}_{(1)})$ similar as Proposition \ref{ineq: coarseD}.
\end{remark}

\begin{remark}
    Although the flocking dynamics of the RBM-1 and RBM-r are similar, the RBM-r is expected to perform comparably, though not as well as the RBM-1, in practice, since replacement reduces interaction uniformity. In other words, the RBM-r’s interactions are less averaged than those of the RBM-1, leading to accumulated errors over intermediate time periods. However, the RBM-r is easier to implement. A simple numerical demonstration of this is provided in Section \ref{sec:numerical} (Figure \ref{fig:rbm1}).
\end{remark}

\section{Proofs of the main results}\label{sec: pf of thms}
In the following, we derive the proofs of Theorem \ref{thm:flocking}, Theorem \ref{thm:flockingrbm1} and Theorem \ref{thm:main}. For better organization, we leave the auxillary lemmas in Section \ref{sec:auxiliary lemmas} and Section \ref{sec:subsecSR}.

\subsection{Proof of Theorem \ref{thm:flocking}}
We split the proof into two parts of velocity alignment and spatial cohension.
\paragraph{Emergence of velocity alignment}\label{subsec: pf of flockingrbmr}
Since $\frac{d}{dt}\sum\limits_i \tilde{V}^i (t) = 0$ and $\sum\limits_i \tilde{V}^i (0) = 0,$ one has that
\begin{align}
    \E \frac{1}{N^2}\sum\limits_{i,j} |\tilde{V}^i(t)-\tilde{V}^j(t)|^2 =& \frac{2}{N}\E \sum\limits_i |\tilde{V}^i (t)|^2 \notag\\
    =& \E\frac{2}{N}\sum\limits_{i\in \mathcal{C}_k} |\tilde{V}^i (t)|^2 + \E\frac{2}{N}\sum\limits_{i\notin \mathcal{C}_k}|\tilde{V}^i (t_k)|^2\notag\\
    =&\E\frac{2}{N}\sum\limits_{i\in \mathcal{C}_k} |\tilde{V}^i (t)|^2 + \E\frac{2}{N}\frac{N-p}{N}\sum\limits_i|\tilde{V}^i (t_k)|^2. \label{eq:thmrbmr1}
\end{align}
Note that
\begin{equation*}
    \E \sum\limits_{i\in\mathcal{C}_k}|\tilde{V}^i (t)|^2 = \E \frac{1}{p}|\sum\limits_{i\in\mathcal{C}_k}\tilde{V}^i (t)|^2 + \frac{1}{2p} \E\sum\limits_{i,j\in\mathcal{C}_k} |\tilde{V}^i(t) -\tilde{V}^j(t)|^2.
\end{equation*}
The key observation is that momentum of each subsystem is conserved by the RBM. By \eqref{eq:LemM1}, it holds that
\begin{equation*}
    \sum\limits_{i\in\mathcal{C}_k}\tilde{V}^i(t) = \sum\limits_{i\in\mathcal{C}_k}\tilde{V}^i(t_k).
\end{equation*}
Note that at $t=t_k$, the randomness of $\mathcal{C}_k$ is independent of $\tilde{V}(t_k)$. Then, 
\begin{align}
    \notag\E |\sum\limits_{i\in\mathcal{C}_k}\tilde{V}^i(t)|^2 &= \E |\sum\limits_{i\in\mathcal{C}_k}\tilde{V}^i(t_k)|^2\\
    \notag& = \E\sum\limits_{i,j}\tilde{V}^i(t_k)  \cdot \tilde{V}^j(t_k)\mathbb{I}_{i\in\mathcal{C}_k}\mathbb{I}_{j\in\mathcal{C}_k}\\
    \notag& = \frac{p}{N}\frac{p-1}{N-1}\E\sum\limits_{i,j;i\ne j}\tilde{V}^i(t_k) \cdot \tilde{V}^j(t_k) + \frac{p}{N}\E \sum\limits_i|\tilde{V}^i(t_k)|^2\\
    \notag& = \frac{p}{N}\frac{p-1}{N-1}\E\sum\limits_{i,j}\tilde{V}^i(t_k) \cdot  \tilde{V}^j(t_k) + \frac{p}{N}\frac{N-p}{N-1}\E \sum\limits_i|\tilde{V}^i(t_k)|^2\\
    & = \frac{p}{N}\frac{N-p}{N-1}\E \sum\limits_i|\tilde{V}^i(t_k)|^2.\label{eq:EsumVr}
\end{align}
Hence one obtains
\begin{equation*}
    \E \sum\limits_{i\in\mathcal{C}_k}|\tilde{V}^i (t)|^2 = \frac{1}{N}\frac{N-p}{N-1}\E \sum\limits_i|\tilde{V}^i(t_k)|^2 + \frac{1}{2p}\E \sum\limits_{i,j\in\mathcal{C}_k} |\tilde{V}^i(t) - \tilde{V}^j(t)|^2.
\end{equation*}
Then, \eqref{eq:thmrbmr1} can be written as 
\begin{align}
    \notag\E &\frac{1}{N^2}\sum\limits_{i,j} |\tilde{V}^i(t)-\tilde{V}^j(t)|^2 \\=& \frac{2}{N}\frac{1}{N}\frac{N-p}{N-1}\E \sum\limits_{i}|\tilde{V}^i(t_k)|^2 + \frac{2}{N}\frac{1}{2p} \E\sum\limits_{i,j\in\mathcal{C}_k} |\tilde{V}^i (t)-\tilde{V}^j(t)|^2 + \frac{2}{N}\frac{N-p}{N}\E \sum\limits_i |\tilde{V}^i(t_k)|^2\\
    \notag=&  \frac{2}{N}\frac{N-p}{N-1}\E \sum\limits_i |\tilde{V}^i(t_k)|^2 + \frac{2}{N}\frac{1}{2p}\E\sum\limits_{i,j\in\mathcal{C}_k}|\tilde{V}^i(t) - \tilde{V}^j(t)|^2\\
    \le& \left(\frac{N-p}{N-1} + \frac{p-1}{N-1}\exp\left(-\frac{2p}{p-1}\psi_0(t-t_k)\right)\right) \frac{1}{N^2}\E \sum\limits_{i,j}|\tilde{V}^i(t_k)-\tilde{V}^j(t_k)|^2,\label{eq:thmrbmr2}
\end{align}
where the last inequality is derived by Lemma \ref{lem:basic flocking}.
Now we use Lemma \ref{lem:4.1} to simplify the decay rate. 
Taking $$a= \frac{N-p}{N-1}, \quad b = \frac{2p}{p-1}\psi_0 \tau, \quad x = \frac{2p}{p-1}\psi_0(t-t_k) ,$$
in Lemma \ref{lem:4.1}, then it holds that
\begin{equation*}
     \E \frac{1}{N^2}\sum\limits_{i,j} |\tilde{V}^i(t)-\tilde{V}^j(t)|^2 \le  \E \frac{1}{N^2}\sum\limits_{i,j} |\tilde{V}^i(t_k)-\tilde{V}^j(t_k)|^2 \exp\left( -\frac{2p}{N-1} \frac{\psi_0}{1+\frac{2p}{p-1}\psi_0\tau}(t-t_k)\right).
\end{equation*}
By induction on $k$, we get
\begin{equation*}
    \E \frac{1}{N^2}\sum\limits_{i,j} |\tilde{V}^i(t)-\tilde{V}^j(t)|^2 \le C\exp\left(-\frac{2p}{N-1} \frac{\psi_0}{1+\frac{2p}{p-1}\psi_0\tau}t\right).
\end{equation*}
Therefore, we have
\begin{equation*}
    \E \frac{1}{N^2}\sum\limits_{i,j} |\tilde{V}^i(\frac{N}{p}t)-\tilde{V}^j(\frac{N}{p}t)|^2 \le C\exp\left(-\frac{2N}{N-1} \frac{\psi_0}{1+\frac{2p}{p-1}\psi_0\tau}t\right).
\end{equation*}

\paragraph{Uniform spatial cohesion}
From the Cauchy-Schwarz inequality,
\begin{equation*}
    \frac{d}{dt}\left(\E\frac{1}{N^2}\sum\limits_{i,j}|\tilde{X}^i -\tilde{X}^j|^2 \right)\le 2 \sqrt{\E\frac{1}{N^2}\sum\limits_{i,j}|\tilde{X}^i -\tilde{X}^j|^2 }\sqrt{\E\frac{1}{N^2}\sum\limits_{i,j}|\tilde{V}^i -\tilde{V}^j|^2 },
\end{equation*}
so that
\begin{equation*}
    \frac{d}{dt}\sqrt{\E\frac{1}{N^2}\sum\limits_{i,j}|\tilde{X}^i -\tilde{X}^j|^2} \le \sqrt{\E\frac{1}{N^2}\sum\limits_{i,j}|\tilde{V}^i -\tilde{V}^j|^2}.
\end{equation*}
Hence the derivative of the position norm is bounded by the velocity norm. Then, by integrating both sides with respect to time and take an expectation, one gets
\begin{multline*}
    \sqrt{\E\frac{1}{N^2}\sum\limits_{i,j}|\tilde{X}^i(t) -\tilde{X}^j(t)|^2}  \le \\\sqrt{\frac{1}{N^2}\sum\limits_{i,j}|\tilde{X}^i(0) -\tilde{X}^j(0)|^2}  + C\int_0^\infty \exp\left(-\frac{2N}{N-1} \frac{\psi_0}{1+\frac{2p}{p-1}\psi_0\tau}t\right) dt \le C,
\end{multline*}
where $C$ is a positive constant depending on $\psi,$ $D(X^{in})$ and $D(V^{in})$. This completes the second part of Theorem \ref{thm:flocking}.

\subsection{Proof of Theorem \ref{thm:flockingrbm1}}
The proof is similar with the approach presented in the above subsection. We briefly demonstrate the emergence of velocity alignment of the RBM-1 here. The uniform spital cohesion proof is totally same.

Recall Algorithm \ref{algo: the RBM1} and the batch division $\xi_k=(\xi_k(1),\cdots,\xi_k(\frac{N}{p}))$ at $t\in[t_k,t_{k+1})$. Similar with Proposition \ref{lem:M}, for any batch $\xi_k(\ell)$, we have
\begin{equation*}
    \sum\limits_i \tilde{V}^i_{(1)}(t) = \sum\limits_i \tilde{V}^{in,i}(t) = 0,\quad \sum\limits_{i\in\xi_k(\ell)}\tilde{V}^i_{(1)}(t)=\sum\limits_{i\in\xi_k(\ell)}\tilde{V}^i_{(1)}(t_k),
\end{equation*}
as well. Set $[i]_k \in\xi_k$ to be the batch containing $i$.
By the independence of $\xi_k$ of $\tilde{V}_{(1)}(t_k)$, it holds that
\begin{align*}
    \E \sum\limits_{\ell =1}^{N/p}|\sum\limits_{i\in\xi_k(\ell)}\tilde{V}^i_{(1)}(t)|^2 &= \E \sum\limits_{\ell =1}^{N/p}|\sum\limits_{i\in\xi_k(\ell)}\tilde{V}^i_{(1)}(t_k)|^2\\
    & = \E\sum\limits_{i,j}\tilde{V}^i_{(1)}(t_k) \cdot\tilde{V}^j_{(1)}(t_k)\mathbb{I}_{\{[i]_k = [j]_k\}}\\
    & = \frac{N-p}{N-1}\E \sum\limits_i|\tilde{V}^i_{(1)}(t_k)|^2.
\end{align*}
Note that 
\begin{equation*}
    \sum\limits_{i\in\xi_k(\ell)}|\tilde{V}^i_{(1)}(t)|^2 = \frac{1}{p}|\sum\limits_{i\in\xi_k(\ell)}\tilde{V}^i_{(1)}(t)|^2 + \frac{1}{2p}\sum\limits_{i,j\in\xi_k(\ell)}|\tilde{V}^i_{(1)}(t)-\tilde{V}^j_{(1)}(t)|^2.
\end{equation*}
Hence one obtains
\begin{align}
    \notag\E &\frac{1}{N^2}\sum\limits_{i,j} |\tilde{V}^i_{(1)}(t)-\tilde{V}^j_{(1)}(t)|^2 
    =\frac{2}{N}\E \sum\limits_i |\tilde{V}^i_{(1)}(t)|^2 = \frac{2}{N}\sum\limits_{\ell=1}^{N/p}\E\sum\limits_{i\in\xi_k(\ell)}|\tilde{V}^i_{(1)}(t)|^2
    \\=& \frac{2}{N}\sum\limits_{\ell=1}^{N/p}\E\frac{1}{N}\frac{N-p}{N-1}\sum\limits_i|\tilde{V}^i_{(1)}(t_k)|^2 + 
    \frac{2}{N}\sum\limits_{\ell=1}^{N/p}\E\frac{1}{2p}\sum\limits_{i,j\in\xi_k(\ell)}|\tilde{V}^i_{(1)}(t)-\tilde{V}^j_{(1)}(t)|^2\\
    \le& \left(\frac{1}{p}\frac{N-p}{N-1} + \frac{p-1}{p}\frac{N}{N-1}\exp\left(-\frac{2p}{p-1}\psi_0(t-t_k)\right)\right) \frac{1}{N^2}\E \sum\limits_{i,j}|\tilde{V}^i(t_k)-\tilde{V}^j(t_k)|^2,\notag
\end{align}
where the last inequality is derived by the similar method mentioned in Theorem \ref{thm:flocking}. Lemma \ref{lem:4.1} and the induction of $k$ leads to the second item of Theorem \ref{thm:flockingrbm1}.

\subsection{Proof of Theorem \ref{thm:main}.}\label{subsec: pf of thmerr}
 
In this section, we give the proof of Theorem \ref{thm:main}. Consider for all $t\in [0,\infty).$ Without loss of generality, we assume that $\tau$ divides $t$ and $p$ divides $N$ for simplicity. To estimate $\E\left|\tilde{V}^i(\frac{N}{p}t) - V^i(t)\right|^2,$ we introduce the auxiliary system IPS' \eqref{eq:IPS'} and divide our analysis into three steps.
\begin{itemize}
    \item (Step 1.) Compare RBM-r and IPS': estimate $\E |\hat{V}^{ii}(\frac{N}{p}t)-\tilde{V}^i (\frac{N}{p}t)|^2$.
    \item (Step 2.) Compare IPS and IPS': estimate $\E |\hat{V}^{ii}(\frac{N}{p}t)-V^i (t)|^2$.
    \item (Step 3.) Combine the results of the above steps. 
\end{itemize}
For better organization, the related auxiliary lemmas are presented in Section \ref{subsec:aux lmm}. 
\paragraph{Step 1.}
For estimating $\E |\hat{V}^{ii}(\frac{N}{p}t)-\tilde{V}^i (\frac{N}{p}t)|^2$, we first introduce the discrepancies between the IPS' and the approximate system RBM-r
\begin{equation*}
    w_X^i (t): = \tilde{X}^i(t) - \hat{X}^{ii}(t), \quad w_V^i (t): = \tilde{V}^i(t) - \hat{V}^{ii}(t),
\end{equation*}
and consider the dynamics of $w_V^i(t).$ Clearly if $i \notin \mathcal{C}_k,$ then $dw_V^i(t)=0.$ We only need to consider the case of $i\in\mathcal{C}_k.$ Define the random variable
\begin{equation}\label{ntt:chi}
    \chi_{k,i}(Z) : = \frac{1}{p-1} \sum\limits_{j\in\mathcal{C}_k} \psi (|X^j - X^i|)(V^j-V^i) - \frac{1}{N-1}\sum\limits_{j=1}^N \psi (|X^j - X^i|)(V^j-V^i).
\end{equation}
Then one has that, if $i \in \mathcal{C}_k,$ 
\begin{align*}
    dw_V^i(t) =& \chi_{k,i}(\hat{Z}^{\cdot,i})dt \\
    &+ \Big[\frac{1}{p-1}\sum\limits_{j\in\mathcal{C}_k} \psi(|\tilde{X}^j - \tilde{X}^i|)(\tilde{V}^j - \tilde{V}^i)
    -\frac{1}{p-1}\sum\limits_{j\in\mathcal{C}_k} \psi(|\hat{X}^{ji} - \hat{X}^{ii}|)(\hat{V}^{ji} - \hat{V}^{ii})\Big]dt,
\end{align*}
where
\begin{equation*}
        \chi_{k,i}(\hat{Z}^{\cdot,i}):=\frac{1}{p-1} \sum\limits_{j\in\mathcal{C}_k} \psi (|\hat{X}^{ji} - \hat{X}^{ii}|)(\hat{V}^{ji} - \hat{V}^{ii}) - \frac{1}{N-1}\sum\limits_{j=1}^N \psi (|\hat{X}^{ji} - \hat{X}^{ii}|)(\hat{V}^{ji} - \hat{V}^{ii}).
\end{equation*}
Now we consider the time derivative of the averaged square error
\begin{align}
    \notag\frac{d}{dt}&\frac{1}{N} \E \sum\limits_i |w_V^i (t) |^2
    =\frac{d}{dt}\frac{1}{N} \E \sum\limits_{i\in\mathcal{C}_k} |w_V^i (t) |^2\\
    \notag=&\frac{1}{N}\E \sum\limits_{i,j\in \mathcal{C}_k} \frac{2}{p-1}w_V^i(t)\cdot  \Big[ \psi(|\tilde{X}^j - \tilde{X}^i|)(\tilde{V}^j - \tilde{V}^i)
    - \psi(|\hat{X}^{ji} - \hat{X}^{ii}|)(\hat{V}^{ji} - \hat{V}^{ii})\Big]dt\\
    \notag&+ \frac{1}{N}\E \sum\limits_{i\in \mathcal{C}_k} 2w_V^i(t)\cdot \chi_{k,i}(\tilde{Z})dt\\
    =&:S(t)+R(t).\label{eqntt:SR cS}
\end{align}

To estimate \eqref{eqntt:SR cS}, we employ a three-stage bootstrapping argument:
\begin{itemize}
    \item Step 1(a): Establish a preliminary exponential decay estimate.
    \item Step 1(b): Refine this to obtain a bound of order $\big(\frac{1}{p-1} - \frac{1}{N-1}\big)\tau$.
    \item Step 1(c): Derive an explicit exponential decay rate.
\end{itemize}

\paragraph{Step 1(a).} Using the flocking estimate established in Theorem \ref{thm:flocking}, we prove Proposition \ref{thm:oldmain}, which demonstrates the exponential decay of the error.
\begin{proposition}\label{thm:oldmain}
Under the assumption of Theorem \ref{thm:main}, one has
    \begin{equation*}
        \frac{1}{N}\E\sum\limits_{i=1}^N|w_V^i(t)|^2 
    \le  C \e^{-C_3 \frac{p}{N}t}+  C\tau\left(\frac{1}{p-1}-\frac{1}{N-1}\right) \e^{-\frac{p}{N}C_3 t}.
    \end{equation*}
\end{proposition}

\begin{proof}
We first coarsely estimate $S(t)$ and $R(t)$ in Lemma \ref{lem:S} and Lemma \ref{lem:R} respectively. 
Combining Young's inequality, one has that for $t\in[t_k,t_{k+1})$ and $i\in \mathcal{C}_k$,
    \begin{align}
        \label{eq:pf1step1.1}\frac{d}{dt}\frac{1}{N}\E\sum\limits_{i\in \mathcal{C}_k}|w_V^i(t)|^2 =&S(t) +  R(t)\\
        \le & -C\frac{1}{N}\E\sum\limits_{i\in\mathcal{C}_k}|w_V^i(t)|^2 \\
        &+ C \frac{p}{N}\e^{-C_3 \frac{p}{N}t} + C\tau \frac{p}{N} \left(\frac{1}{p-1}-\frac{1}{N-1}\right) \e^{-\frac{p}{N}C_3 t}.
    \end{align}
By Gr\"onwall's inequality, one has
\begin{equation}\label{eq:pf1step1.3}
    \frac{1}{N}\E\sum\limits_{i\in \mathcal{C}_k}|w_V^i(t)|^2 \le C \frac{p}{N}\e^{-C_3 \frac{p}{N}t} + C\tau \frac{p}{N} \left(\frac{1}{p-1}-\frac{1}{N-1}\right) \e^{-\frac{p}{N}C_3 t}.
\end{equation}
Take $t=t_k$ to get 
\begin{align}
    \label{eq:pf1step1.4}\frac{1}{N}\E\sum\limits_{i=1}^N|w_V^i(t_k)|^2 =& \frac{N}{p}\frac{1}{N} \E\sum\limits_{i\in\mathcal{C}_k}|w_V^i(t_k)|^2\\
    \notag\le & C \e^{-C_3 \frac{p}{N}t} + C\tau\left(\frac{1}{p-1}-\frac{1}{N-1}\right) \e^{-\frac{p}{N}C_3 t}.
\end{align}
Combining \eqref{eq:pf1step1.3} and \eqref{eq:pf1step1.4}, it holds that
\begin{align}
    \notag\frac{1}{N}&\E\sum\limits_{i=1}^N|w_V^i(t)|^2 =  \frac{1}{N}\E\sum\limits_{i\notin \mathcal{C}_k}|w_V^i(t_k)|^2 + \frac{1}{N}\E\sum\limits_{i\in \mathcal{C}_k}|w_V^i(t)|^2\\
    \notag\le & \frac{N-p}{N^2}\E\sum\limits_{i=1}^N|w_V^i(t_k)|^2 +   C \frac{p}{N}\e^{-C_3 \frac{p}{N}t} + C\tau \frac{p}{N} \left(\frac{1}{p-1}-\frac{1}{N-1}\right) \e^{-\frac{p}{N}C_3 t}\\
    \le &  C \e^{-C_3 \frac{p}{N}t} + C\tau\left(\frac{1}{p-1}-\frac{1}{N-1}\right) \e^{-\frac{p}{N}C_3 t}.\label{eq:pf1step1.5}
\end{align}
\end{proof}

\paragraph{Step 1(b).}
In this step, we perform a double bootstrapping argument to refine the estimate  on the order form $\mathcal{O}(1)$ to $\mathcal{O}\left(\tau \left(\frac{1}{p-1}-\frac{1}{N-1}\right)\right)$. 
\begin{itemize}
    \item Building on Proposition \ref{thm:oldmain}, we leverage the exponential decay of the velocity square difference to relax the estimate for $S(t)$, now accounting for the spatial square difference. Applying the Gr\"onwall-type inequality (Lemma \ref{lem:gronwalltype}) then yields an intermediate estimate of order $\mathcal{O}\left(\tau \left(\frac{1}{p-1}-\frac{1}{N-1}\right)^{\frac{1}{2}}\right)$.
    \item We repeat the bootstrapping procedure to further sharpen the estimate, ultimately achieving the desired order $\mathcal{O}\left(\tau \left(\frac{1}{p-1}-\frac{1}{N-1}\right)\right)$.
\end{itemize}

In detail, by Proposition \ref{thm:oldmain}, we can refine the result of Lemma \ref{lem:S} to \eqref{eq:lem:S'} in Lemma \ref{lem:S'}. Then, combining \eqref{eq:lem:S'} and Lemma \ref{lem:R}, one has that
\begin{align}
    \notag S(t)+R(t)\le & \frac{1}{(p-1)N}(-2p + \frac{N-p}{N-1})\E\sum\limits_{i\in\mathcal{C}_k}|w_V^i(t)|^2 \\
    \notag &+ C\frac{p^{\frac{3}{2}}}{(p-1)N}(\tau^2 \e^{-C_2 \frac{p}{N}t})^{\frac{1}{2}}\left(\E\sum\limits_{i\in\mathcal{C}_k}|w_V^i(t)|^2\right)^{\frac{1}{2}} \\
    \notag &+  \frac{C}{N}\e^{-\frac{C_3}{2}\frac{p}{N}t}\left(\E\sum\limits_{i\in\mathcal{C}_k}|w_X^i(t)|^2\right)^{\frac{1}{2}} \left(\E\sum\limits_{i\in\mathcal{C}_k}|w_V^i(t)|^2\right)^{\frac{1}{2}} \\
        \notag &+ C\frac{p}{N}\tau \left(\frac{1}{p-1} - \frac{1}{N-1}\right)^{\frac{1}{2}}\e^{-C_3 \frac{p}{N}t} + C\frac{p}{N}\tau^2 \e^{-C_3 \frac{p}{N}t}\\
        \le & -\frac{C}{N}\E\sum\limits_{i\in\mathcal{C}_k}|w_V^i(t)|^2 +\frac{C}{N}\E\sum\limits_{i\in\mathcal{C}_k}|w_X^i(t)|^2\e^{-C_3 \frac{p}{N}t}\\& + C \frac{p}{N}\tau^2 \e^{-C_2 \frac{p}{N}t} +  C \frac{p}{N}\left(\frac{1}{p-1} - \frac{1}{N-1}\right)^{\frac{1}{2}}\tau\e^{-C_3 \frac{p}{N}t},\label{eq:WW1.0}
\end{align}
where the last inequality is derived by Young's inequality.
Then by Lemma \ref{lem:gronwalltype}, it holds
\begin{equation}\label{eq:WW1}
\begin{aligned}
    &\frac{1}{N}\E\sum\limits_{i\in\mathcal{C}_k}|w_V^i(t)|^2\le C \frac{p}{N}\tau^2 \e^{-C \frac{p}{N}t} +  C \frac{p}{N}\left(\frac{1}{p-1} - \frac{1}{N-1}\right)^{\frac{1}{2}}\tau\e^{-C \frac{p}{N}t},
    \\&\frac{1}{N}\E\sum\limits_{i\in\mathcal{C}_k}|w_X^i(t)|^2\le C \frac{p}{N}\left(\frac{1}{p-1} - \frac{1}{N-1}\right)^{\frac{1}{2}}\tau + C \frac{p}{N}\tau^2 ,
\end{aligned}
\end{equation}
for some constant $C =C(\psi,D(X^{in}),D(V^{in})).$

Again, using \eqref{eq:WW1}, we refine Lemma \ref{lem:R} to Lemma \ref{lem:R'}, where the order of $\left(\frac{1}{p-1} - \frac{1}{N-1}\right)^{\frac{1}{2}}$ turns to $\left(\frac{1}{p-1} - \frac{1}{N-1}\right)$.
Combining Lemma \ref{lem:S'} and Lemma \ref{lem:R'}, 
and using Lemma \ref{lem:gronwalltype} again, one gets that
\begin{equation}\label{eq:WW2}
\begin{aligned}
    &\frac{1}{N}\E\sum\limits_{i\in\mathcal{C}_k}|w_V^i(t)|^2\le C \frac{p}{N}\tau^2 \e^{-C \frac{p}{N}t} +  C \frac{p}{N}\left(\frac{1}{p-1} - \frac{1}{N-1}\right)\tau\e^{-C \frac{p}{N}t},\\
    &\frac{1}{N}\E\sum\limits_{i\in\mathcal{C}_k}|w_X^i(t)|^2\le  C \frac{p}{N}\left(\frac{1}{p-1} - \frac{1}{N-1}\right)\tau +C \frac{p}{N}\tau^2 .
    \end{aligned}
\end{equation}

\paragraph{Step 1(c).}
Since our analysis relies on a combination of Young's inequality and the Gr\"onwall-type inequality (Lemma \ref{lem:gronwalltype}), it's not easy to determine the precise exponential decay rate explicitly. Nevertheless, we are able to establish the $\mathcal{O}(\tau)$ estimate for $\mathbb{E}\sum\limits_{i\in\mathcal{C}_k}|w_X^i(t)|^2$, which provides a crucial foundation for further bootstrapping.

In this step, we refine the exponential decay rate in \eqref{eq:WW2}. Then, in Section \ref{subsubsec:refined2}, by applying the modified \eqref{eq:WW2} to Lemma \ref{lem:S'} and Lemma \ref{lem:R'}, we derive the improved estimates in Lemma \ref{lem:S''} and Lemma \ref{lem:R''}. Combining these results with Young's inequality, it holds
\begin{align*}
    S(t)+R(t)\le & -\frac{C}{N}\E\sum\limits_{i\in\mathcal{C}_k}|w_V^i(t)|^2 \\
     &+ C\frac{p}{N}\tau \left(\frac{1}{p-1} - \frac{1}{N-1}\right)\e^{-C_3 \frac{p}{N}t} + C\frac{p}{N}\tau^2 \e^{-\frac{C_3}{2} \frac{p}{N}t}.
\end{align*}
Therefore by Gr\"onwall's inequality,
\begin{equation*}
    \frac{1}{N}\E\sum\limits_{i\in\mathcal{C}_k}|w_V^i(t)|^2 \le C\frac{p}{N}\tau \left(\frac{1}{p-1} - \frac{1}{N-1}\right)\e^{-C_3 \frac{p}{N}t} + C\frac{p}{N}\tau^2 \e^{-\frac{C_3}{2} \frac{p}{N}t}.
\end{equation*}

By the similar procedure between \eqref{eq:pf1step1.4} and \eqref{eq:pf1step1.5}, one obtains
\begin{equation*}
    \frac{1}{N}\E\sum\limits_i |\hat{V}^{ii}(t)-\tilde{V}^i (t)|^2 \le C \tau^2 \e^{-C_3 \frac{p}{N}t} +  C\left(\frac{1}{p-1} - \frac{1}{N-1}\right)\tau\e^{-C_3 \frac{p}{N}t},
\end{equation*}
and
\begin{equation*}
    \E \frac{1}{N}\sum\limits_i|\hat{V}^{ii}(\frac{N}{p}t)-\tilde{V}^i (\frac{N}{p}t)|^2 \le C \tau^2 \e^{-C_3 t} +  C \left(\frac{1}{p-1} - \frac{1}{N-1}\right)\tau\e^{-C_3 t}.
\end{equation*}
Now we finish the Step 1.

\paragraph{Step 2.} Then we estimate $\E \left|\hat{V}^{ii}(\frac{N}{p}t)-V^i(t)\right|^2$. 
By the random time change relation discussed in Section \ref{subsec:aux sys} and the notation \eqref{ntt:ti}, one has
    \begin{equation*}
        \E \left|\hat{V}^{ii}(\frac{N}{p}t)-V^i(t)\right|^2 = \E  \left|V^i(\frac{N}{p}t)^{(i)}-V^i(t)\right|^2.
    \end{equation*}
    Recall the notation $\eta_t^{i}$ defined in \eqref{ntt:eta}.
    By the flocking property stated in Lemma \ref{lem:V2diff},
    \begin{equation*}
        \E  \left|V^i(\frac{N}{p}t)^{(i)}-V^i(t)\right|^2 \le C \E \exp\left(-2\psi_0 (\eta_{\frac{N}{p}t}^i \tau \wedge t)\right) \left|\eta_{\frac{N}{p}t}^i \tau - t\right|^2.
    \end{equation*}
    For $n = \frac{N}{p}\frac{t}{\tau}$, one has
    \begin{equation*}
        \begin{aligned}
            &\E \,\exp\left(-2\psi_0 (\eta_{\frac{N}{p}t}^i \tau \wedge t)\right) \left|\eta_{\frac{N}{p}t}^i \tau - t\right|^2\\
            &= \sum\limits_{r=0}^n C_n^r \e^{-2\psi_0 (r\tau\wedge t)}|r\tau -t|^2 \left(\frac{p}{N}\right)^r \left(1-\frac{p}{N}\right)^{n-r}\\
            &\le C \sum\limits_{r=0}^{n}C_n^r\e^{-2\psi_0 r\tau}|r\tau -t|^2 \left(\frac{p}{N}\right)^r \left(1-\frac{p}{N}\right)^{n-r}\\
            &+ C\sum\limits_{r=0}^{n}C_n^r\e^{-2\psi_0t}|r\tau -t|^2 \left(\frac{p}{N}\right)^r \left(1-\frac{p}{N}\right)^{n-r}\\
            &=: I_3 +I_4.
        \end{aligned}
    \end{equation*}
    For $I_3,$  by Lemma \ref{lem:AGB}, taking $A= \e^{-2\psi_0 \tau}$ and $G=\frac{p}{N}A+1-\frac{p}{N},$ it holds
    \begin{equation*}
    \begin{aligned}
        I_3 =& G^{n-2}(A-1)^2 t^2 \left(1-\frac{p}{N}\right)^2 + AG^{n-2}t\tau \left(1-\frac{p}{N}\right)\\
        \le& C \e^{-2\psi_0 t (1-\psi_0 \tau)}\psi_0^2 \tau^2 (1-\psi_0\tau)^2 t^2 \left(1-\frac{p}{N}\right)^2 +  C \e^{-2\psi_0 t (1-\psi_0 \tau)}t\tau \left(1-\frac{p}{N}\right),
    \end{aligned}
    \end{equation*}
    for $N$ being large enough. 

    For $I_4,$ similarly, take $A=G=1$ in Lemma \ref{lem:AGB}. Then one has
    \begin{equation*}
        I_4 \le \sum\limits_{r=0}^{n}C_n^r\e^{-2\psi_0t}|r\tau -t|^2 \left(\frac{p}{N}\right)^r \left(1-\frac{p}{N}\right)^{n-r}
        \le C \e^{-2\psi_0 t}t\tau\left(1 - \frac{p}{N}\right).
    \end{equation*}
    Thus, it holds
    \begin{equation*}
        \E \left|\hat{V}^{ii}(\frac{N}{p}t)-V^i(t)\right|^2 \le C \left(1 - \frac{p}{N}\right)\e^{-2\psi_0 t}\tau .
    \end{equation*}

\paragraph{Step 3.} Combining Step 1 and Step 2, one has that
\begin{align}
        \notag\E &\frac{1}{N}\sum\limits_{i=1}^N\left|\tilde{V}^i(\frac{N}{p}t) - V^i(t)\right|^2 
        \le \E\frac{1}{N}\sum\limits_{i=1}^N|w_V^i (\frac{N}{p}t)|^2 + \E \frac{1}{N}\sum\limits_{i=1}^N\left|\hat{V}^{ii}(\frac{N}{p}t)-V^i(t)\right|^2\\
        \le &C\tau^2\ \e^{-C_3t} + C\tau \left[\left(\frac{1}{p-1}-\frac{1}{N-1}\right)\e^{-C_3 t} + \left(1 - \frac{p}{N}\right)\e^{-2\psi_0 t}\right]\label{eq:CSerr result}\\
        \notag \le &C\tau^2\e^{-C_3t}+ C\tau \left(1-\frac{p}{N}+\frac{1}{p-1}-\frac{1}{N-1}\right)\e^{-C_3 t}.
\end{align}
Now we finish the proof.

\section{Auxiliary lemmas of main results}\label{sec:auxiliary lemmas}
In this section, we present and prove auxiliary lemmas in the order of their appearance in the main proof of Section \ref{sec: pf of thms}.

\subsection{Auxiliary lemmas of flocking dynamics}
First we show a simple preliminary flocking estimate for some subsystem in the RBM-r system.
\begin{lemma}\label{lem:basic flocking}
    Let $(\tilde{X},\tilde{V})$ be the solutions to \eqref{eq: cs} with batch size $p$. Suppose that the communication weight satisfies the assumption \eqref{ass:psi}. Then it holds that 
    \begin{equation*}
         \sum\limits_{i,j\in\mathcal{C}_k} |\tilde{V}^i(t)-\tilde{V}^j(t)|^2 \le \sum\limits_{i,j\in\mathcal{C}_k} |\tilde{V}^i(t_k)-\tilde{V}^j(t_k)|^2 \exp\left( -2\frac{p}{p-1}\psi_0 (t-t_k)\right).
        \end{equation*}
\end{lemma}
\begin{proof}
    The result is derived by direct calculus:
    \begin{align*}
        \frac{d}{dt}&\sum\limits_{i,j\in\mathcal{C}_k} |\tilde{V}^i -\tilde{V}^j|^2 = 2p\frac{d}{dt}\sum\limits_{i\in\mathcal{C}_k}|\tilde{V}^i|^2\\
        =&-\frac{p}{p-1}\sum\limits_{i,j\in\mathcal{C}_k} \psi(|\tilde{X}^j - \tilde{X}^i|)|\tilde{V}^j - \tilde{V}^i|^2\\
        \le& -\frac{p}{p-1}\psi_0 \sum\limits_{i,j\in\mathcal{C}_k} |\tilde{V}^j - \tilde{V}^i|^2.
    \end{align*}
    The Gr\"onwall's inequality leads to the estimate.
\end{proof}
Then, we present an elementary estimate to be used to simplify the decay rate of the relative velocities.
\begin{lemma}[\cite{ha2021uniform}, Lemma 4.1]\label{lem:4.1}
    Let $0\le a \le 1$, $b>0$ be given. Then,
    \begin{equation*}
        a + (1-a)\e^{-x} \le \exp \left( -\frac{1-a}{1+b}x\right), \quad \forall x \in [0,b].
    \end{equation*}
\end{lemma}
\begin{proof}
    We omit the proof but refer to \cite[Lemma 4.1]{ha2021uniform}.
\end{proof}

\subsection{Auxiliary lemmas of Theorem \ref{thm:main}.}\label{subsec:aux lmm}

In this subsection, we establish several auxiliary lemmas that will be essential for proving Theorem \ref{thm:main}. For organizational clarity, we defer the estimates of $S$ and $R$ to  Section \ref{sec:subsecSR}.

Lemma \ref{lem:AGB} provides several combinatorial formulas obtained through direct calculation. Based on Lemma \ref{lem:AGB}, we then derive in Lemma \ref{lem:ui=t} an error estimate arising from the random time change relation. Next, we analyze the expectation and variance induced by the random batch selection process.
 
\begin{lemma}\label{lem:AGB}
         For any constant $A,$ define $G:=\frac{p}{N}A + 1 - \frac{p}{N}.$ By the property of combinatorial number, one has
    \begin{equation*}
        \sum\limits_{r=0}^n C_n^r A^r r^2 \left(\frac{p}{N}\right)^r \left(1-\frac{p}{N}\right)^{n-r} = A\frac{p}{N}n\bigg(G^{n-2}(n-1)A\frac{p}{N} + G^{n-1}\bigg),
    \end{equation*}
    \begin{equation*}
        \sum\limits_{r=0}^n C_n^r A^r r \left(\frac{p}{N}\right)^r \left(1-\frac{p}{N}\right)^{n-r} = A\frac{p}{N}n G^{n-1},
    \end{equation*}
    and
    \begin{equation*}
        \sum\limits_{r=0}^n C_n^r A^r\left(\frac{p}{N}\right)^r \left(1-\frac{p}{N}\right)^{n-r} =  G^n.
    \end{equation*}
    \end{lemma}
   \begin{proof}
       The proof is straightforward by calculus.
   \end{proof}

    In particular, if we set $n= \frac{\frac{N}{p}}{t}\tau$, (suppose $\frac{\frac{N}{p}}{t}\tau \in\mathbb{N}$ without loss of generality), then by Lemma \ref{lem:AGB}, one has 
    \begin{equation}\label{lem:rtau-t}
        \sum\limits_{r=0}^n C_n^r A^r |r\tau-t|^2 \left(\frac{p}{N}\right)^r \left( 1- \frac{p}{N}\right)^{n-r} = G^{n-2}(A-1)^2 t^2 \left( 1- \frac{p}{N}\right)^2 + A G^{n-2}t\tau \left( 1- \frac{p}{N}\right).
    \end{equation}

Next, for the RBM-r, we need to estimate the error term caused by the random time change.
\begin{lemma}\label{lem:ui=t}
Recall the definition of $\hat{V}$ in \eqref{eq:IPS'} and the notation $t^{(i)}$ defined in \eqref{ntt:ti}, it holds that
     \begin{equation}\label{eq:lemui1}
          \E[\sum\limits_{i,j\in\mathcal{C}_k}|\hat{V}^{jj}(t) - \hat{V}^{ji}(t)|^2] \le Cp^2 \tau^2 \left(1-\frac{p}{N}\right)\e^{-\psi_0(1-\psi_0 \tau)\frac{p}{N}t},
     \end{equation}
     for any $j=1,\cdots ,N$, with the positive constant $C$ depending on $\psi$ and $D(V^{in})$. In addition, we have
     \begin{equation}\label{eq:lemui2}
         \E[\E[\e^{-2\psi_0t^{(i)}}\mid i\in\mathcal{C}_k]] \le \exp(-\frac{p}{N}2\psi_0(1-\psi_0 \tau)t).
     \end{equation}
\end{lemma}
\begin{proof}

For the first inequality, without loss of generality, we suppose that $\eta_t^j\ge \eta_t^i$. Then, it holds that
\begin{align}
    \notag\E&[\sum\limits_{i,j\in\mathcal{C}_k}|\hat{V}^{jj}(t) - \hat{V}^{ji}(t)|^2] \\
    \notag\le&
    C\E\sum\limits_{i,j\in\mathcal{C}_k}\left| \int_{t^{(i)}}^{t^{(j)}}\frac{1}{N-1}\sum\limits_{\ell}\psi(|\hat{X}^{\ell j}(s)-\hat{X}^{jj}(s)|) (\hat{V}^{\ell j}(s)-\hat{V}^{jj}(s))ds\right|^2\\
    \label{eq:lemuit1}\le & C\tau^2 \E\sum\limits_{i,j\in\mathcal{C}_k}|\eta_t^j-\eta_t^i|^2 \e^{-2\psi_0 t^{(i)}}\\
    \label{eq:lemuit2}\le& C\tau^2 p^2 \sum\limits_{r_1=0}^n\sum\limits_{r_2=0}^n \e^{-2\psi_0 r_2 \tau} |r_1-r_2|^2 C_n^{r_1} C_n^{r_2} \left(\frac{p}{N}\right)^{r_1+r_2}\left(1-\frac{p}{N}\right)^{2n-r_1-r_2}\\
    \label{eq:lemuit3}=&C\tau^2 p^2 G^{n-2}(G-A)n\frac{p}{N}\left[\frac{p}{N}n(G-A) - \frac{p}{N}(G-A) + G\right].
\end{align}
Here, \eqref{eq:lemuit1} is derived from the combination of Proposition \ref{lem:V2diff} and the fact $\tau\ll 1$. We set $n:= \lceil\frac{t}{\tau}\rceil$ in \eqref{eq:lemuit2}. Taking $A=\e^{-2\psi_0 \tau}$ and $G=1+\frac{p}{N}(A-1)$ in Lemma \ref{lem:AGB}, one obtains \eqref{eq:lemuit3}. Note that 
$    G-A = (1-A)\left(1-\frac{p}{N}\right)$, and $(1-A)\sim 2\psi_0 \tau$, and
\begin{equation*}
    G^n \le C\e^{-\frac{p}{N}2\psi_0(1-\psi_0 \tau)t}.
\end{equation*}
It holds that
\begin{equation*}
    \E[\sum\limits_{i,j\in\mathcal{C}_k}|\hat{V}^{jj}(t) - \hat{V}^{ji}(t)|^2]\le
    \eqref{eq:lemuit3} \le  C\tau^2 p^2 \left(1-\frac{p}{N}\right)\e^{-\frac{p}{N}\psi_0(1-\psi_0 \tau)t},
\end{equation*}
since there exists a constant $C$ such that
\begin{equation*}
    (t^2 + t) \e^{-\frac{p}{N}\psi_0(1-\psi_0 \tau)t} \le C.
\end{equation*}

For \eqref{eq:lemui2}, 
it holds that
\begin{align}
    \notag  \E&[\E[  \e^{-2\psi_0 t^{(i)}}\mid i\in\mathcal{C}_k]] =   \E[\E[ \e^{-2\psi_0 t^{(i)}}\mid i\in\mathcal{C}_k]]\\
    \label{eq:em:ui=t 1}\le & C\sum\limits_{r=0}^{n-1}\left(\frac{p}{N}\right)^r\left(1-\frac{p}{N}\right)^{n-1-r}  C_{n-1}^r \e^{-2\psi_0 r \tau} \e^{-2\psi_0 (t-t_k)\tau}.
\end{align}
Then again by Lemma \ref{lem:AGB}, one obtains \eqref{eq:lemui2}.  
\end{proof}

Then, we complete the lemma on the expectation and variance induced by the random batch division.
\begin{lemma}\label{lem:lem2}
Recall the definitions of $Z$ and $\tilde{Z}$ in Section \ref{subsec:aux sys} and $\chi_{k,i}$ in \eqref{ntt:chi}. Then for $i\in\mathcal{C}_k$, $t\in[t_k,t_{k+1}),$ it holds
    \begin{equation}\label{eq:lem2.1}
        \E [\chi_{k,i}(\hat{Z}^{\cdot,i}(t_k))\mid i\in\mathcal{C}_k] = 0,
    \end{equation}
    \begin{equation}\label{eq:lem2.2}
        \E [w_V^i(t_k)\chi_{k,i}(\hat{Z}^{\cdot,i}(t_k))\mid i\in\mathcal{C}_k] = 0,
    \end{equation}
    \begin{equation}\label{eq:lem2.3}
        \E\sum\limits_{i\in\mathcal{C}_k}\mathrm{Var}(\chi_{k,i}(\hat{Z}^{\cdot,i}(t_k))) \le Cp \left(\frac{1}{p-1}-\frac{1}{N-1}\right)\exp\left(-\frac{p}{N}2C_2 t_k\right),
    \end{equation}
    \begin{equation}\label{eq:lem2.4}
        \E\sum\limits_{i\in\mathcal{C}_k}\left| \chi_{k,i} (\hat{Z}^{\cdot,i}(t)) - \chi_{k,i} (\hat{Z}^{\cdot,i}(t_k))\right|^2 \le Cp \tau^2\exp\left(-\frac{p}{N}2C_2 t\right),
    \end{equation}
    where $C$ denote positive constant that depend on $\psi$, $D(X^{in})$ and $D(V^{in})$, and $C_2$ is defined in \eqref{ntt:C1C2}.
\end{lemma}
\begin{proof}
For notational convenience, we use the symbol $\lesssim$ to denote inequality up to a constant.

\textbf{$\bullet$ Proof of \eqref{eq:lem2.1} and \eqref{eq:lem2.2}.}

The proof follows from the same arguments in [Lemma 3.1, \cite{jin2020random}] or [Lemma 5.2, \cite{ha2021uniform}]. The equalities \eqref{eq:lem2.1} and \eqref{eq:lem2.2} hold due to the independence between the randomness of the batch division and the variable $\hat{Z}^{\cdot,i}(t_k).$ 

\textbf{$\bullet$ Proof of \eqref{eq:lem2.3}.}

Moreover, one has
\begin{equation*}
    \mathrm{Var}(\chi_{k,i}(\hat{Z}^{\cdot,i}(t_k))) = \left(\frac{1}{p-1}-\frac{1}{N-1}\right) \Lambda_i (\hat{Z}^{\cdot,i}(t_k)),
\end{equation*}
where
\begin{equation*}
    \Lambda_i(Z):= \frac{1}{N-2}\sum\limits_j \E \Bigg| \psi(|X^j - X^i|)(V^j - V^i) - \frac{1}{N-1}\sum\limits_\ell \psi(|X^\ell - X^i|(V^\ell - V^i) \Bigg|^2.
\end{equation*}
By Proposition \ref{lem:V2diff}, it holds
\begin{align*}
    &\E\sum\limits_{i\in\mathcal{C}_k}\mathrm{Var}(\chi_{k,i}(\hat{Z}^{\cdot,i}(t_k))) \le
    \E\sum\limits_{i\in\mathcal{C}_k}\frac{1}{N-2}\sum\limits_j  \Bigg| \psi(|\hat{X}^{ji} - \hat{X}^{ii}|)(\hat{V}^{ji} - \hat{V}^{ii}) \\
    &\quad- \frac{1}{N-1}\sum\limits_\ell \psi(|\hat{X}^{\ell i} - \hat{X}^{ii}|)(\hat{V}^{\ell i} - \hat{V}^{ii}) \Bigg|^2  \left(\frac{1}{p-1}-\frac{1}{N-1}\right)\\
    & \le 2 \frac{1}{N-2}\frac{p}{N}\sum\limits_{i,j} \E \Bigg| \psi(|\hat{X}^{ji} - \hat{X}^{ii}|)(\hat{V}^{ji} - \hat{V}^{ii})\Bigg|^2 \left(\frac{1}{p-1}-\frac{1}{N-1}\right)\\
    &\quad + 2\frac{1}{N-2}\frac{p}{N}\frac{1}{N-1}\sum\limits_{i,j,\ell} \E \Bigg|\psi(|\hat{X}^{\ell i} - \hat{X}^{ii}|)(\hat{V}^{\ell i} - \hat{V}^{ii}) \Bigg|^2\left(\frac{1}{p-1}-\frac{1}{N-1}\right)\\
    &\le C p \e^{-\frac{p}{N}2C_2 t_k}\left(\frac{1}{p-1}-\frac{1}{N-1}\right).
\end{align*}
Now one gets \eqref{eq:lem2.3}. 

\textbf{$\bullet$ Proof of \eqref{eq:lem2.4}.}
Note that
\begin{small}
\begin{align*}
    &\left|\chi_{k,i}(\hat{Z}^{\cdot,i}(t)) - \chi_{k,i}(\hat{Z}^{\cdot,i}(t_k)) \right|^2 \\
    & \lesssim \frac{1}{p-1} \sum\limits_{j\in\mathcal{C}_k} \Big| \psi(|\hat{X}^{ji}(t) - \hat{X}^{ii}(t)|)(\hat{V}^{ji}(t) - \hat{V}^{ii}(t)) -  \psi(|\tilde{X}^j(t_k) - \tilde{X}^i(t_k)|)(\hat{V}^{ji}(t) - \hat{V}^{ii}(t))\Big|^2\\
    &\quad + \frac{1}{N-1}\sum\limits_j \Big|\psi(|\hat{X}^{ji}(t) - \hat{X}^{ii}(t)|)(\hat{V}^{ji}(t) - \hat{V}^{ii}(t)) -  \psi(|\hat{X}^{ji}(t_k) - \hat{X}^{ii}(t_k)|)(\hat{V}^{ji}(t) - \hat{V}^{ii}(t))\Big|^2\\
    &=: II_1^i + II_2^i.
\end{align*}
\end{small}
We first consider the expectation of $\E \sum\limits_{i\in\mathcal{C}_k } II_1^i$. It holds that
\begin{align*}
    \E \sum\limits_{i\in\mathcal{C}_k }II_1^i \lesssim & \frac{1}{p-1} \sum\limits_{i,j\in\mathcal{C}_k} \Big| \big[\psi(|\hat{X}^{ji}(t) - \hat{X}^{ii}(t)|) - \psi(|\hat{X}^{ji}(t_k) - \hat{X}^{ii}(t_k)|)\big](\hat{V}^{ji}(t) - \hat{V}^{ii}(t))\Big|^2\\
    &+   \frac{1}{p-1} \sum\limits_{i,j\in\mathcal{C}_k}\psi(|\hat{X}^{ji}(t_k) - \hat{X}^{ii}(t_k)|)\Big|\hat{V}^{ji}(t) - \hat{V}^{ii}(t) - \hat{V}^{ji}(t_k) + \hat{V}^{ii}(t_k)\Big|^2\\
    =:II_{11} + II_{12}.
\end{align*}
By the Lipschitz continuity of $\psi$, one has that
\begin{align}
    \notag II_{11} \le & \E\frac{1}{p-1}\sum\limits_{i,j\in\mathcal{C}_k} \|\psi\|_{Lip}^2 \bigg| \int_{t_k}^t \hat{V}^{ji}(s) - \hat{V}^{ii}(s) ds \,\bigg|^2 \cdot |\hat{V}^{ji}(t) - \hat{V}^{ii}(t)|^2\\
    \le & C \frac{1}{p-1} \E \sum\limits_{i,j\in\mathcal{C}_k} \Big| \hat{V}^{ji}(u)-\hat{V}^{ii}(u) \Big|^2 \tau^2,\label{eq:II11}
\end{align}
with $u\in[t_k,t_{k+1})$, where the second inequality is derived by the mean value theorem and Proposition \ref{lem:V2diff}. Then, by Proposition \ref{lem:V2diff} and Lemma \ref{lem:ui=t}, \eqref{eq:II11} derives
\begin{align*}
    \notag II_{11} \le & Cp\tau^2 \e^{-4C_2\frac{p}{N}t},
\end{align*}
since $\tau \ll 1$.

For $II_{12},$ it holds that
\begin{align*}
    II_{12} \le & \E \frac{2}{p-1} \sum\limits_{i,j\in\mathcal{C}_k} \psi_M^2 \Big| \hat{V}^{ii}(t) - \hat{V}^{ii}(t_k)\Big|^2\\
    \le & 2\E  \sum\limits_{i\in\mathcal{C}_k} \psi_M^2 \bigg|\int_{t_k}^t\frac{1}{p-1}\sum\limits_{j\in\mathcal{C}_k}\psi(|\hat{X}^{ji}(s) - \hat{X}^{ii}(s)|)(\hat{V}^{ji}(s) - \hat{V}^{ii}(s))ds\bigg|^2\\
    \le&C\frac{1}{p-1}\tau^2\E\sum\limits_{i,j\in\mathcal{C}_k} \Big| \hat{V}^{ji}(u)-\hat{V}^{ii}(u) \Big|^2\\
    \le & Cp\tau^2 \e^{-\frac{p}{N}2C_2 t},
\end{align*}
Here, the third equality is derived by the mean value theorem and the fourth inequality is derived by Lemma \ref{lem:basic flocking}. Then one has
\begin{equation*}
    II_1 := II_{11} + II_{12} \le C p \tau^2 \e^{-\frac{p}{N}2C_2 t}.
\end{equation*}

Next we turn to 
\begin{align*}
    \E\sum\limits_{i\in\mathcal{C}_k} II_2^i \lesssim & \frac{1}{N-1} \sum\limits_{i\in\mathcal{C}_k}\sum\limits_j \Big| \big[\psi(|\hat{X}^{ji}(t) - \hat{X}^{ii}(t)|) - \psi(|\hat{X}^{ji}(t_k) - \hat{X}^{ii}(t_k)|)\big](\hat{V}^{ji}(t) - \hat{V}^{ii}(t))\Big|^2\\
    &+   \frac{1}{N-1} \sum\limits_{i\in\mathcal{C}_k}\sum\limits_j\psi(|\hat{X}^{ji}(t_k) - \hat{X}^{ii}(t_k)|)\Big|\hat{V}^{ji}(t) - \hat{V}^{ii}(t) - \hat{V}^{ji}(t_k) + \hat{V}^{ii}(t_k)\Big|^2\\
    =&:II_{21} + II_{22}.
\end{align*}
By combining the Lipschitz continuity of $\psi$, Proposition \ref{lem:V2diff} and Lemma \ref{lem:ui=t}, one has that
\begin{align}
    \notag II_{21} \le & \E\frac{1}{N-1}\sum\limits_{i\in\mathcal{C}_k}\sum\limits_j \|\psi\|_{Lip}^2 \bigg| \int_{t_k}^t \hat{V}^{ji}(s) - \hat{V}^{ii}(s) ds \,\bigg|^2 \cdot |\hat{V}^{ji}(t) - \hat{V}^{ii}(t)|^2\\
    \notag\le & Cp\tau^2 \e^{-\frac{p}{N}2C_2 t},
\end{align}
for some $u\in[t_k,t_{k+1})$, where the estimate here is same as $II_{11}$.


For $II_{22}$, similar with the analysis in $II_{12}$, it holds that
\begin{align*}
    II_{22}
    \le & Cp\tau^2 \e^{-\frac{p}{N}C_2 t}.
\end{align*}
Therefore,
\begin{equation*}
    II_2 := II_{21} + II_{22} \le C p\tau^2\e^{-\frac{p}{N}2C_2 t}.
\end{equation*}
Combining the estimates of $II_1$ and $II_2$, one obtains \eqref{eq:lem2.4}.
\end{proof}

\begin{remark}
    Similarly, corresponding results for $\chi(\tilde{Z})$ can be derived from Theorem \ref{thm:flocking}. However, due to the randomness inherent in RBM-r, the proof, while following a similar structure, becomes more involved. For brevity, we omit the details here. It holds that for $i\in\mathcal{C}_k$, $t\in[t_k,t_{k+1}),$ it holds
    \begin{equation*}
        \E [\chi_{k,i}(\tilde{Z}(t_k))\mid i\in\mathcal{C}_k] = 0,
    \end{equation*}
    \begin{equation*}
        \E [w_V^i(t_k)\chi_{k,i}(\tilde{Z}(t_k))\mid i\in\mathcal{C}_k] = 0,
    \end{equation*}
    \begin{equation*}
        \E\sum\limits_{i\in\mathcal{C}_k}\mathrm{Var}(\chi_{k,i}(\tilde{Z}(t_k))) \le Cp \left(\frac{1}{p-1}-\frac{1}{N-1}\right)\exp\left(-\frac{p}{N}C_1 t_k\right),
    \end{equation*}
    \begin{equation*}
        \E\sum\limits_{i\in\mathcal{C}_k}\left| \chi_{k,i} (\tilde{Z}(t)) - \chi_{k,i} (\tilde{Z}(t_k))\right|^2 \le Cp \tau^2\exp\left(-\frac{p}{N}C_1 t\right),
    \end{equation*}
    where $C$ denote positive constants with respect to $\psi$, $D(X^{in})$ and $D(V^{in})$, and $C_1$ is defined in \eqref{ntt:C1C2}.
\end{remark}

Now we turn to a Gr\"onwall-type inequality. It is a variant of Lemma 3.9 of \cite{ha2018local}.
\begin{lemma}\label{lem:gronwalltype}
    Suppose that two nonnegative functions $\mathcal{X}$ and $\mathcal{V}$ satisfy the coupled differential inequalities:
    \begin{equation*}
    \left\{\begin{aligned}
        &\frac{d}{dt}\sqrt{\mathcal{X}} \le \sqrt{\mathcal{V}},\\
        &\frac{d}{dt}\mathcal{V} \le - \alpha \mathcal{V} + \gamma \e^{-\beta t}\mathcal{X} + f,\quad a.e.\, t>0,
    \end{aligned}\right.
    \end{equation*}
    where $\alpha,$ $\beta$, and $\gamma$ are positive constants, and $f:\mathbb{R}_+\cup \{0\} \to \mathbb{R}$ is a differential nonnegative, nonincreasing function decaying to zero as its argument goes to infinity and it is integrable. Set $\mathcal{X}(0)=\mathcal{V}(0)=0$ for convenience. Then, there exists a positive constant $C= C(\alpha,\beta,\gamma)$ such that 
    \begin{equation*}
        \mathcal{X}\le C \|f\|_{L^1},\quad \mathcal{V}\le C \e^{-(\frac{\alpha\wedge\beta}{2})t}\|f\|_{L^1} + \frac{C}{\alpha}f(\frac{t}{2}),
    \end{equation*}
    for any $t\ge 0$.
\end{lemma}
\begin{proof}
    We leave the proof in Appendix \ref{app:gronwall}.
\end{proof}

\section{Auxiliary lemmas of two functionals}\label{sec:subsecSR}
We now show the estimate of the functionals $S(t)$ and $R(t)$ used in Step 1. Due to our bootstrapping scheme, these estimates can be divided into three parts: 
(i) simple but coarse estimates (Section \ref{subsubsec:coarse}), employed in Step 1(a); (ii) refined estimate (Section \ref{subsubsec:refined1}), used in Steps 1(b); and (iii) refined estimate (Section \ref{subsubsec:refined2}), used in Steps 1(c), which yield an improved decay rate.

\subsection{Coarse estimate.}\label{subsubsec:coarse}
Here, we first show the basic estimate used in Step 1(a).
\begin{lemma}\label{lem:S}
Under the assumption of Theorem \ref{thm:main}, one has
\begin{equation*}
    S(t) \le -C \frac{1}{N} \E \sum\limits_{i\in\mathcal{C}_k} |w_V^i|^2 \\+ C \frac{p}{N} \e^{-C_3 \frac{p}{N}t},
\end{equation*}
where $C$ is a constant depending on $\psi,$ $D(X^{in})$ and $D(V^{in})$.
\end{lemma}
\begin{proof}
Recall the definition \eqref{eqntt:SR cS} of $S(t)$ and split the communication weight term into two pieces to get
    \begin{equation*}
    \begin{aligned}
         S(t) =& \frac{2}{(p-1)N}\E \sum\limits_{i,j\in\mathcal{C}_k} \psi(|\tilde{X}^j - \tilde{X}^i|)\Big[(\tilde{V}^j - \tilde{V}^i) - (\hat{V}^{ji}-\hat{V}^{ii})\Big]\cdot w_V^i\\
         &+ \frac{2}{(p-1)N}\E \sum\limits_{i,j\in\mathcal{C}_k} \Big[\psi(|\tilde{X}^j - \tilde{X}^i|)- \psi(|\hat{X}^{ji}-\hat{X}^{ii}|)\Big](\hat{V}^{ji}-\hat{V}^{ii})\cdot w_V^i \\
         =:& I_1 + I_2.
    \end{aligned}
    \end{equation*}
    
    We first estimate $I_1$. By direct calculus, it holds that
        \begin{align}
            \notag I_1 =& \frac{1}{(p-1)N} \E \sum\limits_{i,j\in\mathcal{C}_k}\psi (|\tilde{X}^j - \tilde{X}^i|)\Big[\big((\tilde{V}^j - \hat{V}^{ji}) - w_V^i\big)\cdot w_V^i + \big((\tilde{V}^i - \hat{V}^{ij})-w_V^j\big)\cdot w_V^j\Big]\\
            \notag =&  \frac{1}{(p-1)N} \E \sum\limits_{i,j\in\mathcal{C}_k}\psi (|\tilde{X}^j - \tilde{X}^i|)\Big[\big(w_V^j - w_V^i + \hat{V}^{jj} - \hat{V}^{ji} \big)\cdot w_V^i \Big]
            \\ \notag &+\frac{1}{(p-1)N} \E \sum\limits_{i,j\in\mathcal{C}_k}\psi (|\tilde{X}^j - \tilde{X}^i|)\Big[ \big(w_V^i-w_V^j + \hat{V}^{ii}-\hat{V}^{ij}\big)\cdot w_V^j\Big]\\
            \notag =& -\frac{1}{(p-1)N} \E \sum\limits_{i,j\in\mathcal{C}_k}\psi (|\tilde{X}^j - \tilde{X}^i|)|w_V^j - w_V^i|^2 \\&
            \notag + \frac{1}{(p-1)N} \E \sum\limits_{i,j\in\mathcal{C}_k}\psi (|\tilde{X}^j - \tilde{X}^i|)\Big[\big(\hat{V}^{jj} - \hat{V}^{ji} \big)\cdot w_V^i + \big(\hat{V}^{ii}-\hat{V}^{ij}\big)\cdot w_V^j\Big]\\
            \notag \le&-\frac{1}{(p-1)N} \E \sum\limits_{i,j\in\mathcal{C}_k}\psi_0 |w_V^j - w_V^i|^2 \\&
            + \frac{1}{(p-1)N} \E \sum\limits_{i,j\in\mathcal{C}_k}\psi (|\tilde{X}^j - \tilde{X}^i|)\Big[\big(\hat{V}^{jj} - \hat{V}^{ji} \big)\cdot w_V^i + \big(\hat{V}^{ii}-\hat{V}^{ij}\big)\cdot w_V^j\Big].\label{eq:lemS 0}
        \end{align}
    Note that
    \begin{equation*}
        \sum\limits_{i\in\mathcal{C}_k}\tilde{V}^i(t) = \sum\limits_{i\in\mathcal{C}_k}\tilde{V}^i(t_k).
    \end{equation*}
    Then we turn to the first term of \eqref{eq:lemS 0}. It holds that
    \begin{equation}\label{eq:lemS 1}
         \E \sum\limits_{i,j\in\mathcal{C}_k}|w_V^j - w_V^i|^2 = 2p \E\sum\limits_{i\in\mathcal{C}_k} |w_V^i|^2 - 2\E\sum\limits_{i,j\in\mathcal{C}_k}w_V^i \cdot w_V^j,
    \end{equation}
    where
    \begin{align}
        \notag\E\sum\limits_{i,j\in\mathcal{C}_k}w_V^i(t) \cdot w_V^j(t) =& \E|\sum\limits_{i\in\mathcal{C}_k}w_V^i|^2 = \E |\sum\limits_{i\in\mathcal{C}_k} \tilde{V}^i(t) - \sum\limits_{i\in\mathcal{C}_k}\hat{V}^{ii}(t)|^2\\
        \le& 2 \E |\sum\limits_{i\in\mathcal{C}_k}\tilde{V}^i(t_k)|^2 + 2 \E|\sum\limits_{i\in\mathcal{C}_k}\hat{V}^{ii}(t)|^2.\label{eq:lemS 2}
    \end{align}
    By Theorem \ref{thm:flocking}, it holds that
    \begin{equation}\label{eq:lemS 3}
        \E |\sum\limits_{i\in\mathcal{C}_k}\tilde{V}^i (t_k)|^2 =\frac{p}{N}\frac{N-p}{N-1}\E\sum\limits_i |\tilde{V}^i (t_k)|^2 \le C p\exp\left(-\frac{2p}{N-1} \frac{\psi_0}{1+\frac{2p}{p-1}\psi_0\tau}t\right).
    \end{equation}
    Recall the notation $t^{(i)}$ defined in \eqref{ntt:ti}. By Proposition \ref{lem:V2diff} and Lemma \ref{lem:ui=t}, it holds that
    \begin{equation*}
        \E |\sum\limits_{i\in\mathcal{C}_k}\hat{V}^{ii}(t)|^2 \le Cp \E \sum\limits_{i\in\mathcal{C}_k} \e^{-2\psi_0 t^{(i)}}
        \le Cp^2 \exp(-\frac{p}{N}2\psi_0(1-\psi_0 \tau)t).
    \end{equation*}
    Combining \eqref{eq:lemS 1} to \eqref{eq:lemS 3}, one has that
    \begin{equation}\label{eq:lemS 11}
         \E \sum\limits_{i,j\in\mathcal{C}_k}|w_V^j - w_V^i|^2 = 2p \E\sum\limits_{i\in\mathcal{C}_k} |w_V^i|^2 +  Cp^2 \exp(-\frac{p}{N}C_3t).
    \end{equation}
    In addition, by Lemma \ref{lem:ui=t}, one has
    \begin{equation}\label{eq:lemS 4}
        \left(\E\sum\limits_{i,j\in\mathcal{C}_k} |\hat{V}^{jj}-\hat{V}^{ji}|^2\right)^{\frac{1}{2}} +\left(\E\sum\limits_{i,j\in\mathcal{C}_k} |\hat{V}^{ii}-\hat{V}^{ij}|^2\right)^{\frac{1}{2}} \le Cp \exp(-\frac{p}{N}\psi_0(1-\psi_0 \tau)t),
    \end{equation}
    and
    \begin{equation}\label{eq:lemS 5}
        \left(\E\sum\limits_{i\in\mathcal{C}_k} p|w_V^i|^2\right)^{\frac{1}{2}} \le \sqrt{p}  \left(\E\sum\limits_{i\in\mathcal{C}_k} |w_V^i|^2\right)^{\frac{1}{2}}.
    \end{equation}
    Then, combining \eqref{eq:lemS 0} and \eqref{eq:lemS 11}-\eqref{eq:lemS 5}, it holds that
    \begin{multline*}
        I_1 \le -\frac{2\psi_0 p}{(p-1)N} \E \sum\limits_{i\in\mathcal{C}_k}|w_V^i|^2 \\+ \frac{C p^2}{(p-1)N} \exp(-\frac{p}{N}C_3t) + \frac{C p\sqrt{p}}{(p-1)N}C \sqrt{\E \sum\limits_{i\in\mathcal{C}_k}|w_V^i|^2} \exp(-\frac{p}{N}\psi_0(1-\psi_0 \tau)t).
    \end{multline*}
    Since $\hat{V}^{ji}(t) = V^j(t^{(i)}),$ by Remark \ref{rmk:N=N} and Lemma \ref{lem:ui=t}, one has that
    \begin{align*}
            I_2 =& \frac{1}{(p-1)N}\E \sum\limits_{i,j\in\mathcal{C}_k} \Big[\psi(|\tilde{X}^j - \tilde{X}^i|)- \psi(|\hat{X}^{ji}-\hat{X}^{ii}|)\Big](\hat{V}^{ji}-\hat{V}^{ii})\cdot w_V^i\\
            \le & \frac{2\psi_M}{(p-1)N}\left(\E \sum\limits_{i,j\in\mathcal{C}_k}| (\hat{V}^{ji}-\hat{V}^{ii})|^2\right)^{\frac{1}{2}}\left(p\E \sum\limits_{i\in\mathcal{C}_k}| w_V^i|^2 \right)^{\frac{1}{2}} \\   
            \le &C\frac{\sqrt{p}}{N} \left(\E \sum\limits_{i\in\mathcal{C}_k} |w_V^i|^2 \right)^{\frac{1}{2}}\exp(-\frac{p}{N}\psi_0(1-\psi_0 \tau)t) .
    \end{align*}
    Therefore, combining the above estimates and Young's inequality, the proof is completed. 
\end{proof}

Next we estimate $R(t)$.
\begin{lemma}\label{lem:R}
Under the assumption of Theorem \ref{thm:main}, one has that
\begin{equation*}
    R(t) \le C\frac{2}{N}\tau \left(\E\sum\limits_{i\in\mathcal{C}_k}|w_V^i(t)|^2 \right)^{\frac{1}{2}}\left(p \e^{-\frac{p}{N}C_3 t}\right)^{\frac{1}{2}} + C \frac{p}{N}\tau \left(\frac{1}{p-1} - \frac{1}{N-1}\right)^\frac{1}{2} \e^{-C_3 \frac{p}{N}t},
\end{equation*}
    where the positive constant $C$ depends on $\psi,$ $D(X^{in})$ and $D(V^{in})$.
\end{lemma}

\begin{proof}
    Recall the definition of $R(t)$ in \eqref{eqntt:SR cS}. 
    Since by Lemma \ref{lem:lem2},
    \begin{equation*}
        \E[w_V^i(t_k)\cdot\chi_{k,i}(\hat{Z}^{\cdot,i}(t_k))\mid i\in\mathcal{C}_k]=0.
    \end{equation*}
    Then, one has
    \begin{align}
        \notag R(t) =& \frac{2}{N}\E\sum\limits_{i\in\mathcal{C}_k}[w_V^i \cdot \chi_{k,i}(\hat{Z}^{\cdot,i})]\\
        \notag =&  \frac{2}{N}\E[\sum\limits_{i\in\mathcal{C}_k}(w_V^i(t)-w_V^i(t_k))\cdot\chi_{k,i}(\hat{Z}^{\cdot,i}(t_k))]\\
            \notag&+  \frac{2}{N}\E [\sum\limits_{i\in\mathcal{C}_k}w_V^i(t)\cdot(\chi_{k,i}(\hat{Z}^{\cdot,i}(t))-\chi_{k,i}(\hat{Z}^{\cdot,i}(t_k)))]\\
        \label{eq:lemR 1.1} \le& \frac{2}{N} \left(\E\sum\limits_{i\in\mathcal{C}_k}|w_V^i(t)|^2 \right)^{\frac{1}{2}}\left(\E\sum\limits_{i\in\mathcal{C}_k}\left| \chi_{k,i} (\hat{Z}^{\cdot,i}(t)) - \chi_{k,i} (\hat{Z}^{\cdot,i}(t_k))\right|^2\right)^{\frac{1}{2}} \\
        \label{eq:lemR 1.2} &+\frac{2}{N}\left(\E\sum\limits_{i\in\mathcal{C}_k}|w_V^i(t)-w_V^i(t_k)|^2\right)^{\frac{1}{2}}\left(\E\sum\limits_{i\in\mathcal{C}_k}\left| \chi_{k,i} (\hat{Z}^{\cdot,i}(t_k))\right|^2\right)^{\frac{1}{2}}.
    \end{align}
    Consider 
    \begin{align}
        \notag&\E\sum\limits_{i\in\mathcal{C}_k}|w_V^i(t)-w_V^i(t_k)|^2\\
        \notag&=\E\sum\limits_{i\in\mathcal{C}_k}\bigg|\int_{t_k}^t \frac{1}{p-1} \sum\limits_{j\in\mathcal{C}_k}\psi(|\tilde{X}^j(s) - \tilde{X}^i(s)|)(\tilde{V}^j(s) - \tilde{V}^i(s))
        \\
        \notag&\quad-\frac{1}{N-1} \sum\limits_{j}\psi(|\hat{X}^{ji}(s) - \hat{X}^{ii}(s)|)(\hat{V}^{ji}(s) - \hat{V}^{ii}(s))ds\bigg|^2\\
        \label{eq:lem deltaw1}&=\E\sum\limits_{i\in\mathcal{C}_k}\tau^2 \bigg| \frac{1}{p-1} \sum\limits_{j\in\mathcal{C}_k}\psi(|\tilde{X}^j(u) - \tilde{X}^i(u)|)(\tilde{V}^j(u) - \tilde{V}^i(u))
        \\
        \notag&\quad-\frac{1}{N-1} \sum\limits_{j}\psi(|\hat{X}^{ji}(u) - \hat{X}^{ii}(u)|)(\hat{V}^{ji}(u) - \hat{V}^{ii}(u))\bigg|^2\\
        \label{eq:lem deltaw2} &\le C\tau^2\frac{1}{p-1}\E\sum\limits_{i,j\in\mathcal{C}_k}|\tilde{V}^j(t_k) - \tilde{V}^i(t_k)|^2
        +C\tau^2\frac{1}{N-1}\sum\limits_{j}\sum\limits_{i\in\mathcal{C}_k}(|\hat{V}^{ji}(u)|^2 + |\hat{V}^{ii}(u)|^2)\\
        \label{eq:lem deltaw3} &\le C \tau^2 p \e^{-\frac{p}{N}C_3 t},
    \end{align}
    where $u\in[t_k,t_{k+1})$.
    Here \eqref{eq:lem deltaw1} is derived by the mean value theorem. The inequality \eqref{eq:lem deltaw2} is by Lemma \ref{lem:basic flocking}. Recall $C_3 = \min\{C_1,\,2C_2\}$
    and thus \eqref{eq:lem deltaw3} is derived by the combination of Proposition \ref{lem:V2diff} and Lemma \ref{lem:ui=t}.  
    
    Combining \eqref{eq:lem2.3}-\eqref{eq:lem2.4} in Lemma \ref{lem:lem2} and \eqref{eq:lemR 1.1}-\eqref{eq:lemR 1.2}, \eqref{eq:lem deltaw3}, it holds that
    \begin{equation*}
        R(t) \le \frac{2}{N} \left(\E\sum\limits_{i\in\mathcal{C}_k}|w_V^i(t)|^2 \right)^{\frac{1}{2}}C\left(p \tau^2\e^{-\frac{p}{N}C_3 t}\right)^{\frac{1}{2}} + C \frac{p}{N}\tau \left(\frac{1}{p-1} - \frac{1}{N-1}\right)^{\frac{1}{2}} \e^{-C_3 \frac{p}{N}t}.
    \end{equation*}
\end{proof}

\subsection{Refined estimate 1.}\label{subsubsec:refined1}

Now, we employ the exponential decay estimate of Proposition \ref{thm:oldmain} to optimize the bounds for the functionals $S(t)$ and $R(t)$.

\begin{lemma}\label{lem:S'}
Under the assumption of Theorem \ref{thm:main}, one has
\begin{align}
    \notag S(t) \le &- \frac{1}{(p-1)N}\left(-2p + \frac{N-p}{N-1}\right) \E \sum\limits_{i\in\mathcal{C}_k} |w_V^i|^2 \\&+ C \frac{p\sqrt{p}}{(p-1)N}\left(\tau^2 \e^{-C_3\frac{p}{N}t}\right)^{\frac{1}{2}}\left(\E \sum\limits_{i\in\mathcal{C}_k} |w_V^i|^2\right)^{\frac{1}{2}}\\
    \notag &+ \frac{C}{N}\left(\E\sum\limits_{i\in\mathcal{C}_k}|w_X^i|^2\right)^{\frac{1}{2}}\left(\E \sum\limits_{i\in\mathcal{C}_k} |w_V^i|^2\right)^{\frac{1}{2}} \e^{-\frac{C_3}{2}\frac{p}{N}t}\\
    &+ C \frac{p}{N} \left(\frac{1}{p-1}-\frac{1}{N-1}\right)\tau \e^{-C_3 \frac{p}{N}t} + C\frac{p}{N}\tau^2 \e^{-C_2\frac{p}{N}t},\label{eq:lem:S'}
\end{align}
where $C$ is a constant depending on $\psi,$ $D(X^{in})$ and $D(V^{in})$.
\end{lemma}
\begin{proof}
Recall the definition of $I_1$ and $I_2$ in Lemma \ref{lem:S}. It holds that
    \begin{align}
        \label{eq:lemS' 1} I_1 \le&-\frac{1}{(p-1)N} \E \sum\limits_{i,j\in\mathcal{C}_k}\psi_0 |w_V^j - w_V^i|^2 \\&
        + \frac{1}{(p-1)N} \E \sum\limits_{i,j\in\mathcal{C}_k}\psi (|\tilde{X}^j - \tilde{X}^i|)\Big[\big(\hat{V}^{jj} - \hat{V}^{ji} \big)\cdot w_V^i + \big(\hat{V}^{ii}-\hat{V}^{ij}\big)\cdot w_V^j\Big].\label{eq:lemS' 2}
    \end{align}
We consider the cross terms in \eqref{eq:lemS' 1}. It holds that
    \begin{align}
        \notag\E& \sum\limits_{i,j\in\mathcal{C}_k}w_V^i(t)\cdot w_V^j(t)\\
        \label{eq:lemS'1.0.1}\le&\E \sum\limits_{i,j\in\mathcal{C}_k}w_V^i(t_k)\cdot w_V^j(t_k)\\
         \label{eq:lemS'1.0.2}& + \E \sum\limits_{i,j\in\mathcal{C}_k}w_V^i(t)\cdot (w_V^j(t) - w_V^j(t_k)) + w_V^j(t)\cdot (w_V^i(t) - w_V^i(t_k))\\
         \label{eq:lemS'1.0.3}&- \E \sum\limits_{i,j\in\mathcal{C}_k}(w_V^i(t) - w_V^i(t_k)) \cdot (w_V^j(t) - w_V^j(t_k)).
    \end{align}
    Similar with the discussion in the derivation of \eqref{eq:EsumVr}, it holds
    \begin{align}
        \notag\E &\sum\limits_{i,j\in\mathcal{C}_k}w_V^i(t_k)\cdot w_V^j(t_k)\\ 
        =&\frac{p}{N}\frac{p-1}{N-1} \E\sum\limits_{i,j}w_V^i(t_k)\cdot w_V^j(t_k)+ \frac{N-p}{N-1}\E\sum\limits_{i\in\mathcal{C}_k}|w_V^i(t_k)|^2.\label{eq:lemS'1.1}
    \end{align}
    For the first term of \eqref{eq:lemS'1.1}, it holds
    \begin{align*}
         \E&\sum\limits_{i,j}w_V^i(t_k)\cdot w_V^j(t_k)\\
         =& \E\sum\limits_{i,j} (\tilde{V}^i(t_k) - \hat{V}^{ii}(t_k))\cdot (\tilde{V}^j(t_k) - \hat{V}^{jj}(t_k))\\
         =&\E|\sum\limits_i \hat{V}^{ii}(t_k)|^2,
    \end{align*}
    since $\sum\limits_i \tilde{V}^i = 0$. Also, note that
    \begin{equation*}
        \sum\limits_j \hat{V}^{ji} (t) = 0, \quad\text{for any }i.
    \end{equation*}
    Then it holds
    \begin{equation*}
        \sum\limits_i \hat{V}^{ii} = \frac{1}{N}\sum\limits_{i,j}\hat{V}^{ii} -  \frac{1}{N}\sum\limits_{i,j}\hat{V}^{ji} =\frac{1}{N}\sum\limits_{i,j}\hat{V}^{ii} -  \frac{1}{N}\sum\limits_{i,j}\hat{V}^{ij}.
    \end{equation*}
    As a simple variant of Lemma \ref{lem:ui=t}, one can prove that
    \begin{equation*}
        \E\left|\frac{1}{N}\sum\limits_{i,j}(\hat{V}^{ii}(t_k) - \hat{V}^{ij}(t_k))\right|^2 \le CN^2\tau^2 \e^{-C_2\frac{p}{N}t}.
    \end{equation*}
    Therefore,
    \begin{equation}\label{eq:lemS'1.1.1}
        \E\sum\limits_{i,j}w_V^i(t_k)\cdot w_V^j(t_k)
         =\E|\sum\limits_i \hat{V}^{ii}(t_k)|^2 \le  CN^2\tau^2 \e^{-C_2\frac{p}{N}t}.
    \end{equation}
    For the second term of \eqref{eq:lemS'1.1}, it holds that
    \begin{align}
        \notag&\frac{N-p}{N-1}\E\sum\limits_{i\in\mathcal{C}_k}|w_V^i(t_k)|^2\\
        \notag\le & \frac{N-p}{N-1}\E\sum\limits_{i\in\mathcal{C}_k}|w_V^i(t)|^2 + \frac{N-p}{N-1}\E\sum\limits_{i\in\mathcal{C}_k}|w_V^i(t_k)|^2 -|w_V^i(t)|^2\\
        \le & \frac{N-p}{N-1}\E\sum\limits_{i\in\mathcal{C}_k}|w_V^i(t)|^2 + \frac{N-p}{N-1}Cp \tau \e^{-C_3 \frac{p}{N}t},\label{eq:lemS'1.1.2}
    \end{align}
    where the last inequality can be derived by Proposition \ref{thm:oldmain}.
    Combining  \eqref{eq:lemS'1.1.1} with \eqref{eq:lemS'1.1.2}, \eqref{eq:lemS'1.0.1} can be bounded by
    \begin{equation*}
        \E \sum\limits_{i,j\in\mathcal{C}_k}w_V^i(t_k)\cdot w_V^j(t_k) \le 
        \frac{N-p}{N-1}\E\sum\limits_{i\in\mathcal{C}_k}|w_V^i(t)|^2 + \frac{N-p}{N-1}Cp \tau \e^{-C_3 \frac{p}{N}t} + Cp^2 \tau^2 \e^{-C_2\frac{p}{N}t}.
    \end{equation*}
    For \eqref{eq:lemS'1.0.2}, by the Cauchy-Schwarz inequality, one obtains
    \begin{align*}
        \E &\sum\limits_{i,j\in\mathcal{C}_k} w_V^i(t)\cdot (w_V^j(t) - w_V^j(t_k))\\
        \le & \left( \E\sum\limits_{i,j\in\mathcal{C}_k}|w_V^i(t)|^2\right)^\frac{1}{2} \left(\E \sum\limits_{i,j\in\mathcal{C}_k}|w_V^i(t) - w_V^i(t_k)|^2\right)^{{\frac{1}{2}}}.
    \end{align*}
    By the mean value theorem, 
    \begin{align}
        \notag\E &\sum\limits_{i,j\in\mathcal{C}_k}|w_V^i(t) - w_V^i(t_k)|^2\\
        \notag\le C & p\tau^2 \E\sum\limits_{i\in\mathcal{C}_k} \bigg| \chi_{k,i}\tilde{Z}(u) + \frac{1}{N-1}\sum\limits_j \psi(|\tilde{X}^j(u)-\tilde{X}^i(u)|)(\tilde{V}^j(u)-\tilde{V}^i(u)) \\&- \frac{1}{N-1}\sum\limits_j\psi(\hat{X}^{ji}(u)-\hat{X}^{ii}(u)|)(\hat{V}^{ji}(u) - \hat{V}^{ii}(u))\bigg|^2\\
         \label{eq:lemS' 3}\le & Cp^2\tau^2 \e^{-\frac{p}{N}C_3 t},
    \end{align}
    where the last inequality \eqref{eq:lemS' 3} is derived by the combination of Lemma \ref{lem:lem2}, Lemma \ref{lem:ui=t} and Theorem \ref{thm:flocking}.
    Then, it holds 
    \begin{equation*}
         \E \sum\limits_{i,j\in\mathcal{C}_k} w_V^i(t)\cdot (w_V^j(t) - w_V^j(t_k))\le Cp \left(\tau^2 \e^{-C_3\frac{p}{N}t}\right)^{\frac{1}{2}}\left( \E \sum\limits_{i,j\in\mathcal{C}_k}|w_V^i(t)|^2\right)^{\frac{1}{2}}.
    \end{equation*}
    Using \eqref{eq:lemS' 3} again, one knows \eqref{eq:lemS'1.0.3} can be controlled by
    \begin{equation*}
       \left|\E \sum\limits_{i,j\in\mathcal{C}_k}(w_V^i(t) - w_V^i(t_k)) \cdot (w_V^j(t) - w_V^j(t_k)) \right| \le  Cp^2\tau^2 \e^{-\frac{p}{N}C_3 t}.
    \end{equation*}
    Combining the above estimates and Lemma \ref{lem:ui=t}, one gets
    \begin{align*}
        I_1\le & \frac{1}{(p-1)N}(-2p + \frac{N-p}{N-1})\E\sum\limits_{i\in\mathcal{C}_k}|w_V^i(t)|^2 + C\frac{p^{\frac{3}{2}}}{(p-1)N}(\tau^2 \e^{-C_3 \frac{p}{N}t})^{\frac{1}{2}}\left(\E\sum\limits_{i\in\mathcal{C}_k}|w_V^i(t)|^2\right)^{\frac{1}{2}} \\
        &+ C\frac{p^{\frac{3}{2}}}{(p-1)N}\left(\tau^2 \e^{-C_2 \frac{p}{N}t}\left(1 - \frac{p}{N}\right)\right)^{\frac{1}{2}}\left(\E\sum\limits_{i\in\mathcal{C}_k}|w_V^i(t)|^2\right)^{\frac{1}{2}} \\&+ C \frac{1}{(p-1)N}\frac{N-p}{N-1}p\tau \e^{-C_3\frac{p}{N}t}+C\frac{p^2}{(p-1)N}\tau^2 \e^{-\frac{p}{N}C_3 t}.
    \end{align*}
    
    Now we turn to the other term $I_2$. Note that
    \begin{align*}
            I_2 =& \frac{1}{(p-1)N}\E \sum\limits_{i,j\in\mathcal{C}_k} \Big[\psi(|\tilde{X}^j - \tilde{X}^i|)- \psi(|\hat{X}^{ji}-\hat{X}^{ii}|)\Big](\hat{V}^{ji}-\hat{V}^{ii})\cdot w_V^i\\
            \le&C \frac{1}{(p-1)N}\left(\E\sum\limits_{i\in\mathcal{C}_k}|w_V^i(t)|^{2+\epsilon}\right)^{\frac{1}{2+\epsilon}}\left(\E \sum\limits_{i,j\in\mathcal{C}_k}| \tilde{X}^j - \hat{X}^{ji} - \tilde{X}^i + \hat{X}^{ii}|^{2+\epsilon} \right)^{\frac{1}{2+\epsilon}}\\
            &\cdot \left(\E \sum\limits_{i,j\in\mathcal{C}_k} |\hat{V}^{ji} - \hat{V}^{ii}|^{\frac{2+\epsilon}{\epsilon}} \right)^{\frac{\epsilon}{2+\epsilon}},
    \end{align*}
    for any $0< \epsilon<2$. Due to the boundedness of $\tilde{V}$ and $\hat{V}$, and the exponential decay of $\E\sum\limits_{i\in\mathcal{C}_k}|w_V^i|^2$, one can prove the integrability of $\E\sum\limits_{i\in\mathcal{C}_k}|w_X^i|^4$. By dominated convergence theorem, take $\epsilon\to 0$ to obtain
            \begin{align*}
            I_2 \le & \frac{C\sqrt{p}}{(p-1)N}\left(\E\sum\limits_{i\in\mathcal{C}_k}|w_V^i(t)|^2\right)^{\frac{1}{2}}\left(\E \sum\limits_{i,j\in\mathcal{C}_k}| \tilde{X}^j - \hat{X}^{ji} - \tilde{X}^i + \hat{X}^{ii}|^2 \right)^{\frac{1}{2}}\e^{-C_2\frac{p}{N}t},
    \end{align*}
    where we used a variant of Lemma \ref{lem:ui=t} to get
    \begin{equation*}
        \lim\limits_{\epsilon\to 0} \left(\E \sum\limits_{i,j\in\mathcal{C}_k} |\hat{V}^{ji} - \hat{V}^{ii}|^{\frac{2+\epsilon}{\epsilon}} \right)^{\frac{\epsilon}{2+\epsilon}} \lesssim \lim\limits_{\epsilon\to 0} \left(p^2 \exp (-\frac{p}{N}\frac{2+\epsilon}{\epsilon}C_2 t)\right)^{\frac{\epsilon}{2+\epsilon}} \lesssim \e^{-C_2\frac{p}{N}t}.
    \end{equation*}
    Consider
    \begin{align*}
        \E &\sum\limits_{i,j\in\mathcal{C}_k}| \tilde{X}^j - \hat{X}^{ji} - \tilde{X}^i + \hat{X}^{ii}|^2\\
        \le & \E\sum\limits_{i,j\in\mathcal{C}_k}|w_X^j-w_X^i + \hat{X}^{jj}-\hat{X}^{ji}|^2\\
        \le & 4p\E \sum\limits_{i\in\mathcal{C}_k}|w_X^i|^2 + 2\E \sum\limits_{i,j\in\mathcal{C}_k}\left|\int_0^t  \hat{V}^{jj}(s) - \hat{V}^{ji}(s)ds\right|^2.
    \end{align*}
    By Lemma \ref{lem:ui=t}, it holds that
    \begin{equation*}
        \E \sum\limits_{i,j\in\mathcal{C}_k}\left|\int_0^t  \hat{V}^{jj}(s) - \hat{V}^{ji}(s)ds\right|^2 \le C p^2 \tau^2 \left(1-\frac{p}{N}\right).
    \end{equation*}
    Then,
    \begin{equation}\label{eq:S'I2}
    \begin{aligned}
        I_2 \le& \frac{C}{N}\left(\E\sum\limits_{i\in\mathcal{C}_k}|w_V^i(t)|^2\right)^{\frac{1}{2}}
        \left(\E\sum\limits_{i\in\mathcal{C}_k}|w_X^i(t)|^2\right)^{\frac{1}{2}}\e^{-C_2\frac{p}{N}t}\\
        &+ \frac{Cp^{\frac{3}{2}}}{(p-1)N}\left(\E\sum\limits_{i\in\mathcal{C}_k}|w_V^i(t)|^2\right)^{\frac{1}{2}}\left(\tau^2 \e^{-C_3\frac{p}{N}t}\right)^{\frac{1}{2}}.
    \end{aligned}
    \end{equation}

    Therefore, combining the estimates of $I_1$ and $I_2$, it holds
    \begin{align*}
        S\le & I_1+I_2\\
        \le & \frac{1}{(p-1)N}(-2p + \frac{N-p}{N-1})\E\sum\limits_{i\in\mathcal{C}_k}|w_V^i(t)|^2 \\&+ C\frac{p^{\frac{3}{2}}}{(p-1)N}(\tau^2 \e^{-C_3 \frac{p}{N}t})^{\frac{1}{2}}\left(\E\sum\limits_{i\in\mathcal{C}_k}|w_V^i(t)|^2\right)^{\frac{1}{2}} \\
        &+ \frac{C}{N}\left(\E\sum\limits_{i\in\mathcal{C}_k}|w_X^i(t)|^2\right)^{\frac{1}{2}} \left(\E\sum\limits_{i\in\mathcal{C}_k}|w_V^i(t)|^2\right)^{\frac{1}{2}}\e^{-\frac{C_3}{2}\frac{p}{N}t}\\
        &+ C\frac{p}{N}\tau \left(\frac{1}{p-1} - \frac{1}{N-1}\right)\e^{-C_3 \frac{p}{N}t} + C\frac{p}{N}\tau^2 \e^{-C_2 \frac{p}{N}t}.
    \end{align*}
\end{proof}

Next, building on \eqref{eq:WW1}, we refined the estimate of $R(t)$.

\begin{lemma}\label{lem:R'}
Under the assumption of Theorem \ref{thm:main}, one has that
\begin{multline*}
    R(t) \le C\frac{2}{N}\tau \left(\E\sum\limits_{i\in\mathcal{C}_k}|w_V^i(t)|^2 \right)^{\frac{1}{2}}\left(p \e^{-\frac{p}{N}C_3 t}\right)^{\frac{1}{2}}\\ + C \frac{p}{N}\tau \left(\frac{1}{p-1} - \frac{1}{N-1}\right) \e^{-\frac{C_3}{2} \frac{p}{N}t} + C \frac{p}{N}\tau^2 \e^{-\frac{C_3}{2}\frac{p}{N}t},
\end{multline*}
    where the positive constant $C$ depends on $\psi,$ $D(X^{in})$ and $D(V^{in})$.
\end{lemma}
\begin{proof}
In this proof, we use \eqref{eq:WW1} to improve the result. Recall the increment of $w_V^i$ during a short time interval
    \begin{align}
        \notag&\E\sum\limits_{i\in\mathcal{C}_k}|w_V^i(t)-w_V^i(t_k)|^2\\
        \notag&=\E\sum\limits_{i\in\mathcal{C}_k}\bigg|\int_{t_k}^t \frac{1}{p-1} \sum\limits_{j\in\mathcal{C}_k}\psi(|\tilde{X}^j(s) - \tilde{X}^i(s)|)(\tilde{V}^j(s) - \tilde{V}^i(s))
        \\
        \notag&\quad-\frac{1}{N-1} \sum\limits_{j}\psi(|\hat{X}^{ji}(s) - \hat{X}^{ii}(s)|)(\hat{V}^{ji}(s) - \hat{V}^{ii}(s))ds\bigg|^2\\
        \notag&=\E\sum\limits_{i\in\mathcal{C}_k}\tau^2 \bigg| \frac{1}{p-1} \sum\limits_{j\in\mathcal{C}_k}\psi(|\tilde{X}^j(u) - \tilde{X}^i(u)|)(\tilde{V}^j(u) - \tilde{V}^i(u))
        \\
        \notag&\quad-\frac{1}{N-1} \sum\limits_{j}\psi(|\hat{X}^{ji}(u) - \hat{X}^{ii}(u)|)(\hat{V}^{ji}(u) - \hat{V}^{ii}(u))\bigg|^2\\
        \notag& = \tau^2\E\sum\limits_{i\in\mathcal{C}_k}\Bigg|\frac{1}{p-1}\sum\limits_{j\in\mathcal{C}_k}\psi(|\tilde{X}^j(u) - \tilde{X}^i(u)|)(\tilde{V}^j(u) - \tilde{V}^i(u)) \\&\,- \frac{1}{p-1}\sum\limits_{j\in\mathcal{C}_k}\psi(|\hat{X}^{ji}(u) - \hat{X}^{ii}(u)|)(\hat{V}^{ji}(u) - \hat{V}^{ii}(u))
        +\chi_{k,i}(\hat{Z}^{\cdot,i}(u))\Bigg|^2,\label{eq:lemR' 1}
    \end{align}
    where $u\in[t_k,t_{k+1})$. Recalling by Lemma \ref{lem:lem2}, one has that
    \begin{equation*}
        \E\sum\limits_{i\in\mathcal{C}_k}\left|\chi_{k,i} (\hat{Z}^{\cdot,i}(t_k))\right|^2\le Cp \left(\frac{1}{p-1}-\frac{1}{N-1}\right)\exp\left(-\frac{p}{N}2C_2 t_k\right),
    \end{equation*}
    and
    \begin{equation*}
        \E\sum\limits_{i\in\mathcal{C}_k}\left| \chi_{k,i} (\hat{Z}^{\cdot,i}(t)) - \chi_{k,i} (\hat{Z}^{\cdot,i}(t_k))\right|^2 \le Cp \tau^2\exp\left(-\frac{p}{N}2C_2 t\right).
    \end{equation*}
    Hence,
    \begin{equation*}
        \tau^2 \E\sum\limits_{i\in\mathcal{C}_k} |\chi_{k,i}(\hat{Z}^{\cdot,i}(u))|^2 \le Cp\tau^2 (\frac{1}{p-1}-\frac{1}{N-1})\e^{-\frac{p}{N}2C_2 t_k} +  Cp\tau^4 \e^{-\frac{p}{N}2C_2 t_k}.
    \end{equation*}
    For the first two terms of \eqref{eq:lemR' 1},  using the same scheme in Lemma \ref{lem:S}, one has
    \begin{align}
       \notag\E&\sum\limits_{i\in\mathcal{C}_k}\Bigg|\frac{1}{p-1}\sum\limits_{j\in\mathcal{C}_k}\psi(|\tilde{X}^j(u) - \tilde{X}^i(u)|)(\tilde{V}^j(u) - \tilde{V}^i(u))\\
       \label{eq:notsureR}& \quad- \frac{1}{p-1}\sum\limits_{j\in\mathcal{C}_k}\psi(|\hat{X}^{ji}(u) - \hat{X}^{ii}(u)|)(\hat{V}^{ji}(u) - \hat{V}^{ii}(u))\Bigg|^2\\
        \notag\lesssim& \frac{1}{p-1}\E\sum\limits_{i,j\in\mathcal{C}_k}\left|\left(\psi(|\tilde{X}^j(u) - \tilde{X}^i(u)|) - \psi(|\hat{X}^{ji}(u) - \hat{X}^{ii}(u)|)\right)(\tilde{V}^j(u) - \tilde{V}^i(u))\right|^2\\
        \notag & 
        \quad+ \frac{1}{p-1}\E\sum\limits_{j\in\mathcal{C}_k}\left|\psi(|\hat{X}^{ji}(u) - \hat{X}^{ii}(u)|)(\tilde{V}^j(u) - \tilde{V}^i(u) - \hat{V}^{ji}(u) + \hat{V}^{ii}(u))\right|^2.
        \end{align}
        By the boundedness of $\tilde{V}$ in Proposition \ref{ineq: coarseD},
        \begin{align*}
        \E&\sum\limits_{i,j\in\mathcal{C}_k}\left|\left(\psi(|\tilde{X}^j(u) - \tilde{X}^i(u)|) - \psi(|\hat{X}^{ji}(u) - \hat{X}^{ii}(u)|)\right)(\tilde{V}^j(u) - \tilde{V}^i(u)\right|^2\\
        \lesssim & \E\sum\limits_{i,j\in\mathcal{C}_k}\left| \tilde{X}^j(u)- \tilde{X}^i(u) - \hat{X}^{ji}(u) + \hat{X}^{ii}(u)\right|^2\\
        \lesssim &p \E\sum\limits_{i\in\mathcal{C}_k}|w_X^i(u)|^2  + p^2\tau^2,
        \end{align*}
        where the last inequality is derived similarly as \eqref{eq:lem:S'} in Lemma \ref{lem:S'}. Then \eqref{eq:notsureR} is estimated as
        \begin{align*}
        \E&\sum\limits_{i\in\mathcal{C}_k}\Bigg|\frac{1}{p-1}\sum\limits_{j\in\mathcal{C}_k}\psi(|\tilde{X}^j(u) - \tilde{X}^i(u)|)(\tilde{V}^j(u) - \tilde{V}^i(u))\\
        & \quad- \frac{1}{p-1}\sum\limits_{j\in\mathcal{C}_k}\psi(|\hat{X}^{ji}(u) - \hat{X}^{ii}(u)|)(\hat{V}^{ji}(u) - \hat{V}^{ii}(u))\Bigg|^2\\
        \lesssim
        &\E\sum\limits_{i\in\mathcal{C}_k}|w_X^i(u)|^2  + p\tau^2
        + \frac{1}{p-1}\left(p\E\sum\limits_{i\in\mathcal{C}_k}|w_V^i(u)|^2 + p^2\tau^2\e^{-C_2\frac{p}{N}t}\right)\\
        \lesssim& \E\sum\limits_{i\in\mathcal{C}_k}|w_V^i(t)|^2 
        +\E\sum\limits_{i\in\mathcal{C}_k}|w_X^i(t)|^2
        + p \tau^2 + p\tau^2 \e^{-C_2\frac{p}{N}t},
    \end{align*}
    where the last inequality is derived by \eqref{eq:WW1}. Hence, \eqref{eq:lemR' 1} becomes
    \begin{multline}\label{eq:lemR'2}
        \E\sum\limits_{i\in\mathcal{C}_k}|w_V^i(t)-w_V^i(t_k)|^2 \le C\tau^2 \E\sum\limits_{i\in\mathcal{C}_k}|w_V^i(t)|^2
        +C\tau^2 \E\sum\limits_{i\in\mathcal{C}_k}|w_X^i(t)|^2\\ + Cp\tau^4 + C\tau^2 p \left(\frac{1}{p-1}-\frac{1}{N-1}\right)\e^{-C_3\frac{p}{N}t}.
    \end{multline}
    Combining \eqref{eq:lemR'2} with \eqref{eq:lemR 1.1} and \eqref{eq:lemR 1.2} in Lemma \ref{lem:R}, and Lemma \ref{lem:lem2}, one has that
    \begin{align*}
        R(t) \le&C\frac{2}{N}\tau \left(\E\sum\limits_{i\in\mathcal{C}_k}|w_V^i(t)|^2 \right)^{\frac{1}{2}}\left(p \e^{-\frac{p}{N}C_3 t}\right)^{\frac{1}{2}} \\
        &+ C \frac{1}{N}\left(\E\sum\limits_{i\in\mathcal{C}_k}|w_V^i(t)|^2 \right)^{\frac{1}{2}}\left(p\tau^2 \left(\frac{1}{p-1} - \frac{1}{N-1}\right) \e^{-C_3 \frac{p}{N}t}\right)^{\frac{1}{2}}\\
        &+ C \frac{1}{N}\left(\E\sum\limits_{i\in\mathcal{C}_k}|w_X^i(t)|^2 \right)^{\frac{1}{2}}\left(p\tau^2 \left(\frac{1}{p-1} - \frac{1}{N-1}\right) \e^{-C_3 \frac{p}{N}t}\right)^{\frac{1}{2}}\\
        &+C \frac{p}{N}\tau\left(\frac{1}{p-1} - \frac{1}{N-1}\right) \e^{-C_3 \frac{p}{N}t} + C \frac{p}{N}\tau^2 \e^{-\frac{C_3}{2}\frac{p}{N}t}\\
        \le&C\frac{2}{N}\tau \left(\E\sum\limits_{i\in\mathcal{C}_k}|w_V^i(t)|^2 \right)^{\frac{1}{2}}\left(p \e^{-\frac{p}{N}C_3 t}\right)^{\frac{1}{2}} \\
        &+ C \frac{p}{N}\tau \left(\frac{1}{p-1} - \frac{1}{N-1}\right) \e^{-\frac{C_3}{2} \frac{p}{N}t} + C \frac{p}{N}\tau^2 \e^{-\frac{C_3}{2}\frac{p}{N}t},
    \end{align*}
    where the last inequality is derived by 
    \begin{align*}
        &\frac{1}{N}\left(\E\sum\limits_{i\in\mathcal{C}_k}|w_X^i(t)|^2 \right)^{\frac{1}{2}}\left(p\tau^2 \left(\frac{1}{p-1} - \frac{1}{N-1}\right) \e^{-C_3 \frac{p}{N}t}\right)^{\frac{1}{2}} \\
        &\lesssim \frac{p}{N}\tau \E\sum\limits_{i\in\mathcal{C}_k}|w_X^i(t)|^2 + \frac{p}{N}\tau \left(\frac{1}{p-1} - \frac{1}{N-1}\right) \e^{-\frac{C_3}{2} \frac{p}{N}t}\\
       & \lesssim \frac{p}{N}\tau^2  \e^{-\frac{C_3}{2} \frac{p}{N}t} +
        \frac{p}{N}\tau \left(\frac{1}{p-1} - \frac{1}{N-1}\right) \e^{-\frac{C_3}{2} \frac{p}{N}t},
    \end{align*}
    since \eqref{eq:WW1} holds.
\end{proof}

\subsection{Refined estimate 2.}\label{subsubsec:refined2}

Here, we further improve the exponential decay rates for the functionals $S(t)$ and $R(t)$.

\begin{lemma}\label{lem:S''}
Under the assumption of Theorem \ref{thm:main}, one has
\begin{align}
    \notag S(t) \le & \frac{1}{(p-1)N}(-2p + \frac{N-p}{N-1})\E\sum\limits_{i\in\mathcal{C}_k}|w_V^i(t)|^2 \\
        \notag &+\frac{C\sqrt{p}}{N}\left(\left(\frac{1}{p-1}-\frac{1}{N-1}\right)\tau\e^{-C_3\frac{p}{N}t} + \tau^2 \e^{-C_3 \frac{p}{N}t} \right)^{\frac{1}{2}} \left(\E\sum\limits_{i\in\mathcal{C}_k}|w_V^i(t)|^2\right)^{\frac{1}{2}}\\
        \notag &+ C\frac{p}{N}\tau \left(\frac{1}{p-1} - \frac{1}{N-1}\right)\e^{-C_3 \frac{p}{N}t}  + C\frac{p}{N}\tau^2 \e^{-C_2 \frac{p}{N}t}.
\end{align}
where $C$ is a constant depending on $\psi,$ $D(X^{in})$ and $D(V^{in})$.
\end{lemma}
\begin{proof}
We only modify the estimate of $I_2$ in Lemma \ref{lem:S'}. Recall \eqref{eq:S'I2}, by \eqref{eq:WW2}, it holds that
    \begin{align*}
        I_2 \le& \frac{C}{N}\left(\E\sum\limits_{i\in\mathcal{C}_k}|w_V^i(t)|^2\right)^{\frac{1}{2}}
        \left(\E\sum\limits_{i\in\mathcal{C}_k}|w_X^i(t)|^2\right)^{\frac{1}{2}}\e^{-C_2\frac{p}{N}t}\\
        &+ \frac{Cp^{\frac{3}{2}}}{(p-1)N}\left(\E\sum\limits_{i\in\mathcal{C}_k}|w_V^i(t)|^2\right)^{\frac{1}{2}}\left(\tau^2 \e^{-C_3\frac{p}{N}t}\right)^{\frac{1}{2}}\\
        \le& \frac{C\sqrt{p}}{N}\left(\left(\frac{1}{p-1}-\frac{1}{N-1}\right)\tau\e^{-C_3\frac{p}{N}t}\right)^{\frac{1}{2}}
        \left(\E\sum\limits_{i\in\mathcal{C}_k}|w_V^i(t)|^2\right)^{\frac{1}{2}}\\
        &+ \frac{Cp^{\frac{3}{2}}}{(p-1)N}\left(\E\sum\limits_{i\in\mathcal{C}_k}|w_V^i(t)|^2\right)^{\frac{1}{2}}\left(\tau^2 \e^{-C_3\frac{p}{N}t}\right)^{\frac{1}{2}}.
    \end{align*}

    Therefore, combining the estimates of $I_1$ in Lemma \ref{lem:S'}, it holds
    \begin{align*}
        S\le & I_1+I_2\\
        \le & \frac{1}{(p-1)N}(-2p + \frac{N-p}{N-1})\E\sum\limits_{i\in\mathcal{C}_k}|w_V^i(t)|^2 \\
        &+\frac{C\sqrt{p}}{N}\left(\left(\frac{1}{p-1}-\frac{1}{N-1}\right)\tau\e^{-C_3\frac{p}{N}t}\right)^{\frac{1}{2}}
        \left(\E\sum\limits_{i\in\mathcal{C}_k}|w_V^i(t)|^2\right)^{\frac{1}{2}}\\
        &+ C\frac{C\sqrt{p}}{N}(\tau^2 \e^{-C_3 \frac{p}{N}t})^{\frac{1}{2}}\left(\E\sum\limits_{i\in\mathcal{C}_k}|w_V^i(t)|^2\right)^{\frac{1}{2}} \\
        &+ C\frac{p}{N}\tau \left(\frac{1}{p-1} - \frac{1}{N-1}\right)\e^{-C_3 \frac{p}{N}t}  + C\frac{p}{N}\tau^2 \e^{-C_2 \frac{p}{N}t}.
    \end{align*}
\end{proof}

Based on \eqref{eq:WW2}, we modify the estimate of $R(t)$.

\begin{lemma}\label{lem:R''}
Under the assumption of Theorem \ref{thm:main}, one has that
\begin{align*}
    R(t) \le& C\frac{\sqrt{p}}{N}\tau \left(\E\sum\limits_{i\in\mathcal{C}_k}|w_V^i(t)|^2 \right)^{\frac{1}{2}}\left(\tau^2 \e^{-\frac{p}{N}C_3 t}\right)^{\frac{1}{2}}\\ &+ C \frac{p}{N}\tau \left(\frac{1}{p-1} - \frac{1}{N-1}\right) \e^{-C_3 \frac{p}{N}t} + C \frac{p}{N}\tau^2 \e^{-\frac{C_3}{2}\frac{p}{N}t},
\end{align*}
    where the positive constant $C$ depends on $\psi,$ $D(X^{in})$ and $D(V^{in})$.
\end{lemma}
\begin{proof}
In this proof, we use \eqref{eq:WW2} to improve the result. Recall the decomposition of $R$ in Lemma \ref{lem:R} combined with Lemma \ref{lem:lem2}:
    \begin{align}
        \notag R(t) =& \frac{2}{N}\E\sum\limits_{i\in\mathcal{C}_k}[w_V^i \cdot \chi_{k,i}(\hat{Z}^{\cdot,i})]\\
        \notag =&  \frac{2}{N}\E[\sum\limits_{i\in\mathcal{C}_k}(w_V^i(t)-w_V^i(t_k))\cdot\chi_{k,i}(\hat{Z}^{\cdot,i}(t_k))]\\
        \notag &+  \frac{2}{N}\E [\sum\limits_{i\in\mathcal{C}_k}w_V^i(t)\cdot(\chi_{k,i}(\hat{Z}^{\cdot,i}(t))-\chi_{k,i}(\hat{Z}^{\cdot,i}(t_k)))]\\
        \notag \le& C\frac{\sqrt{p}}{N} \left(\E\sum\limits_{i\in\mathcal{C}_k}|w_V^i(t)|^2 \right)^{\frac{1}{2}}\left(\tau^2 \e^{-C_2\frac{p}{N}t}\right)^{\frac{1}{2}} \\
        \notag &+ \frac{2}{N}\E[\sum\limits_{i\in\mathcal{C}_k}(w_V^i(t)-w_V^i(t_k))\cdot\chi_{k,i}(\hat{Z}^{\cdot,i}(t_k))].
    \end{align}
We consider
\begin{align}
    \notag \frac{2}{N}&\E[\sum\limits_{i\in\mathcal{C}_k}(w_V^i(t)-w_V^i(t_k))\cdot\chi_{k,i}(\hat{Z}^{\cdot,i}(t_k))]\\
    \notag =&\frac{2}{N}\E\bigg[\sum\limits_{i\in\mathcal{C}_k}\int_{t_k}^t \bigg(\frac{1}{p-1} \sum\limits_{j\in\mathcal{C}_k}\psi(|\tilde{X}^j(s)-\tilde{X}^i(s)|)(\tilde{V}^j(s) - \tilde{V}^i(s)) \\\notag &\,- \frac{1}{p-1}\sum\limits_{j\in\mathcal{C}_k}\psi(|\hat{X}^{ji}(s) - \hat{X}^{ii}(s)|)(\hat{V}^{ji}(s) - \hat{V}^{ii}(s))+\chi_{k,i}(\hat{Z}^{\cdot,i}(s))\bigg)ds \cdot \chi_{k,i}(\hat{Z}^{\cdot,i}(t_k))\bigg]\\
    \label{eq:R''1}\le& \frac{2}{(p-1)N}\tau \E\bigg[\sum\limits_{i,j\in\mathcal{C}_k} \Big|\psi(|\tilde{X}^j(u)-\tilde{X}^i(u)|)- \psi(|\hat{X}^{ji}(u) - \hat{X}^{ii}(u)|)\Big|\Big|\hat{V}^{ji}(u)-\hat{V}^{ii}(u))\Big|\bigg]\\
    \notag&+ \frac{2}{(p-1)N}\tau \E\bigg[\sum\limits_{i,j\in\mathcal{C}_k} \bigg|\psi(|\tilde{X}^j(u)-\tilde{X}^i(u)|)\Big(\tilde{V}^j(u)-\tilde{V}^i(u)-\hat{V}^{ji}(u)+\hat{V}^{ii}(u)\Big)\bigg|\\
    \label{eq:R''2}g&\cdot \bigg|\chi_{k,i}(\hat{Z}^{\cdot,i}(t_k))\bigg|\bigg]\\
    \label{eq:R''3}&+ \frac{2}{(p-1)N}\tau \E\bigg[\sum\limits_{i,j\in\mathcal{C}_k} \chi_{k,i}(\hat{Z}^{\cdot,i}(u)) \cdot \chi_{k,i}(\hat{Z}^{\cdot,i}(t_k))\bigg].
\end{align}
For \eqref{eq:R''1}, using similar scheme as in Lemma \ref{lem:S'} and the Lipschitz property of $\psi$, one obtains
\begin{align*}
    \E&\bigg[\sum\limits_{i,j\in\mathcal{C}_k} \Big|\psi(|\tilde{X}^j(u)-\tilde{X}^i(u)|)- \psi(|\hat{X}^{ji}(u) - \hat{X}^{ii}(u)|)\Big|\Big|\hat{V}^{ji}(u)-\hat{V}^{ii}(u))\Big|\bigg]\frac{2}{(p-1)N} \tau\\
    \le& C \frac{\tau}{Np} \left(\E\sum\limits_{i,j\in\mathcal{C}_k} \Big|\tilde{X}^j(u) - \tilde{X}^i(u)-\hat{X}^{ji}(u)+\hat{X}^{ii}(u)\Big|^{2+\epsilon}\right)^{\frac{1}{2+\epsilon}}\\
    &\cdot \left(\E\sum\limits_{i,j\in\mathcal{C}_k} \Big|\hat{V}^{ji}(u)-\hat{V}^{ii}(u)\Big|^{\frac{2+\epsilon}{\epsilon}}\right)^{\frac{\epsilon}{2+\epsilon}}\\
    &\cdot \left(\E\sum\limits_{i,j\in\mathcal{C}_k} \Big|\chi_{k,i}(\hat{Z}^{\cdot,i}(t_k))\Big|^{2+\epsilon}\right)^{\frac{1}{2+\epsilon}}\\
    \le& C \frac{\tau\sqrt{\tau}p}{N} \left(\frac{1}{p-1}-\frac{1}{N-1}\right) \e^{-C_2 \frac{p}{N}t}\e^{-\frac{C_3}{2} \frac{p}{N}t}.
\end{align*}
For \eqref{eq:R''2}, it holds that
\begin{align*}
    \E&\bigg[\sum\limits_{i,j\in\mathcal{C}_k} \bigg|\psi(|\tilde{X}^j(u)-\tilde{X}^i(u)|)\Big(\tilde{V}^j(u)-\tilde{V}^i(u)-\hat{V}^{ji}(u)+\hat{V}^{ii}(u)\Big)\bigg|\\
    &\cdot \bigg|\chi_{k,i}(\hat{Z}^{\cdot,i}(t_k))\bigg|\bigg]\frac{2}{(p-1)N}\tau\\
    \le & C\frac{\tau}{Np}\left(\E\sum\limits_{i,j\in\mathcal{C}_k}|w_V^j(u)-w_V^i(u)|^2 \right)^{\frac{1}{2}}\left(\E\sum\limits_{i,j\in\mathcal{C}_k} |\chi_{k,i}(\hat{Z}^{\cdot,i}(t_k))|^2\right)^{\frac{1}{2}}\\
    &+ C\frac{\tau}{Np}\left( \E \sum\limits_{i,j\in\mathcal{C}_k}\Big| \hat{V}^{jj}(u)-\hat{V}^{ji}(u)\Big|^2\right)^{\frac{1}{2}}\left(\E\sum\limits_{i,j\in\mathcal{C}_k} |\chi_{k,i}(\hat{Z}^{\cdot,i}(t_k))|^2\right)^{\frac{1}{2}}\\
    \le &C \frac{\tau}{N} \left(\E\sum\limits_{i\in\mathcal{C}_k}|w_V^i(u)|^2 \right)^{\frac{1}{2}}\left(p\left(\frac{1}{p-1}-\frac{1}{N-1}\right) \e^{-C_3\frac{p}{N}t}\right)^{\frac{1}{2}}\\
    &+ C\frac{p}{N}\tau^2 \left(1-\frac{p}{N}\right)^{\frac{1}{2}}\e^{-\frac{C_2}{2}\frac{p}{N}t} \left(\frac{1}{p-1}-\frac{1}{N-1}\right)^{\frac{1}{2}}\e^{-\frac{C_3}{2}\frac{p}{N}t}.
\end{align*}
In addition, by Lemma \ref{lem:lem2}, it holds
\begin{align*}
    &\frac{2}{(p-1)N}\tau \E\bigg[\sum\limits_{i,j\in\mathcal{C}_k} \chi_{k,i}(\hat{Z}^{\cdot,i}(u)) \cdot \chi_{k,i}(\hat{Z}^{\cdot,i}(t_k))\bigg]\\
    &\le C \frac{p}{N}\tau \left(\frac{1}{p-1}-\frac{1}{N-1}\right) \e^{-C_3\frac{p}{N}t} + C \frac{p}{N}\tau^2 \left(\frac{1}{p-1}-\frac{1}{N-1}\right)^{\frac{1}{2}}\e^{-C_3\frac{p}{N}t}.
\end{align*}
Combining all the estimates above and \eqref{eq:WW2}, 
\begin{align*}
    \frac{2}{N}&\E[\sum\limits_{i\in\mathcal{C}_k}(w_V^i(t)-w_V^i(t_k))\cdot\chi_{k,i}(\hat{Z}^{\cdot,i}(t_k))]\\
    \le & C \frac{p}{N}\tau \left(\frac{1}{p-1}-\frac{1}{N-1}\right) \e^{-C_3}\frac{p}{N}t + C \frac{p}{N}\tau^2 \e^{-\frac{C_3}{2}t} + C \frac{\tau}{N}\left(\E\sum\limits_{i\in\mathcal{C}_k}|w_V^i(t)|^2\right)^{\frac{1}{2}}(p\e^{-C_3\frac{p}{N}t})^{\frac{1}{2}}.
\end{align*}
Hence we conclude the proof.
\end{proof}
    
\section{Numerical Simulations}\label{sec:numerical}
In this section, we present numerical simulations on the RBM-r-approximation on a one-dimension test example.

To show the performance of RBM-r and compare it with the original system clearly, we introduce the following equivalent Algorithm \ref{algo: the RBMr'}. Without loss of generality, we set $\tau$ that divides $T$.
\begin{algorithm}\label{algo: the RBMr'}
\caption{ The RBM-r for \eqref{eq: cs}}
\For{$k=1$ to $\frac{T}{\tau}$}
{\For{$\ell=1$ to $\frac{N}{p}$}
    {Pick a batch $\mathcal{C}_k^{(\ell)}$ of size $p$ randomly.
    Update $(\tilde{X}^i,\tilde{V}^i)$ in $\mathcal{C}_k^{(\ell)}$ by solving 
    \begin{equation*}
        \left\{\begin{aligned}
            \partial_t \tilde{X}^i(t) =& \tilde{V}^i(t),\, i=1,\cdots, N,\\
            \partial_t \tilde{V}^i(t) =& \frac{\kappa}{p-1}\sum\limits_{j\in\mathcal{C}^{(\ell)}_k}
            \psi(|\tilde{X}^j(t)-\tilde{X}^i(t)|)(\tilde{V}^j(t)-\tilde{V}^i(t)),
        \end{aligned}\right.
    \end{equation*}
    for $t\in [t_k, t_{k+1}).$ 
    }}
\end{algorithm}

In the test example, we set
\begin{equation*}
    \kappa =1,\, \psi(r)=\frac{1}{(1+|r|^2)^\frac{1}{4}}.
\end{equation*}
The numerical simulations in this section were conducted with the following parameters, unless otherwise specified:
\begin{equation*}
    N=2^6,\,\tau =0.1,\,p=2.
\end{equation*}
To integrate the RBM-r system, we employed the forward Euler method for efficient computation.
Although there is a positive lower bound assumption on $\psi$ in \eqref{ass:psi}, the numerical simulations in this section yield results consistent with Theorem \ref{thm:main}, as the relative positions in the RBM-approximated trajectories do not increase rapidly.

\paragraph{Zig-zag trajectories.}
First, in Figure \ref{fig:traj}, we present the trajectories of the original system, the RBM-1 approximation, and the RBM-r approximation. It is evident that while both the RBM-1 and RBM-r approximation systems converge to zero following zig-zag paths, the RBM-r shows greater dispersion due to the allowance for replacements.

\begin{figure}
    \centering
    \includegraphics[width=1.0\linewidth]{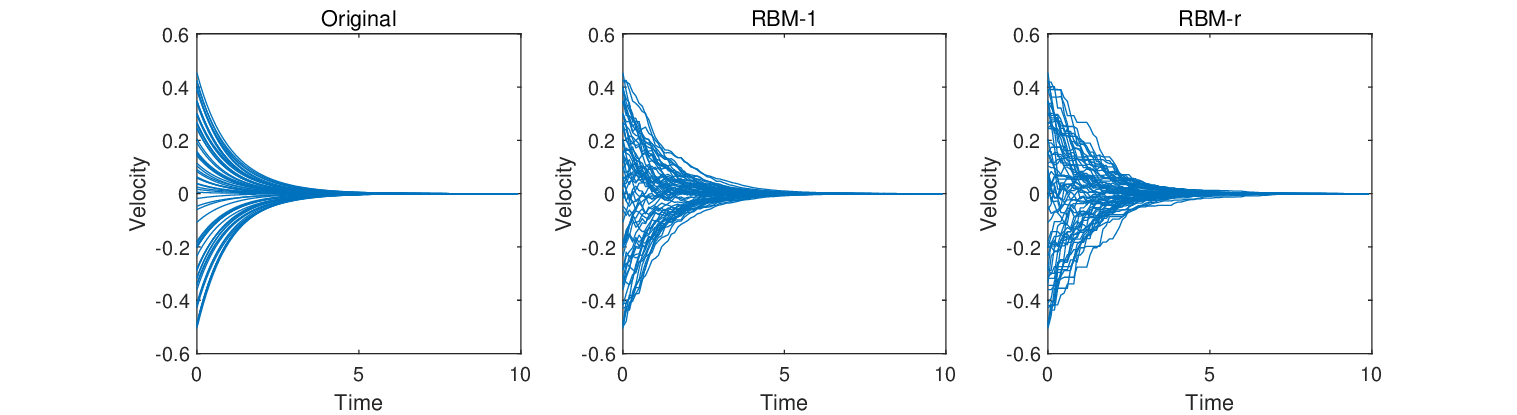}
    \caption{A simulation on trajectories along time on the original system (left), the RBM-1 system (middle) and the RBM-r system (right) with $p=2$.}
    \label{fig:traj}
\end{figure}

\paragraph{Stochastic flocking}
Next, we show the asymptotic flocking of the RBM systems. In Figure \ref{fig:flock0.1} and Figure \ref{fig:flock0.01}, we consider the scaled sum of squared difference (SSD):
\begin{equation*}
    \text{SSD of V} := \frac{1}{N^2}\sum\limits_{i,j}|V^i-V^j|^2, \quad \text{SSD of X} := \frac{1}{N^2}\sum\limits_{i,j}|X^i-X^j|^2,
\end{equation*}
from 100 simulations. In Figure \ref{fig:flock0.1}, we set $\tau = 0.1$, while in Figure \ref{fig:flock0.01} $\tau = 0.01$. We denote the RBM-r with a dashed line, the RBM-1 with a dotted line, and the original system with a solid line. As shown in Theorem \ref{thm:flocking} and Theorem \ref{thm:flockingrbm1}, the flocking rates of the two random batch methods and the original system are similar when the time step is much less than one.

\begin{figure}[htbp]
    \centering
    \subfigure[]{
    \begin{minipage}[t]{0.46\textwidth}
        \centering
        \includegraphics[scale=0.46]{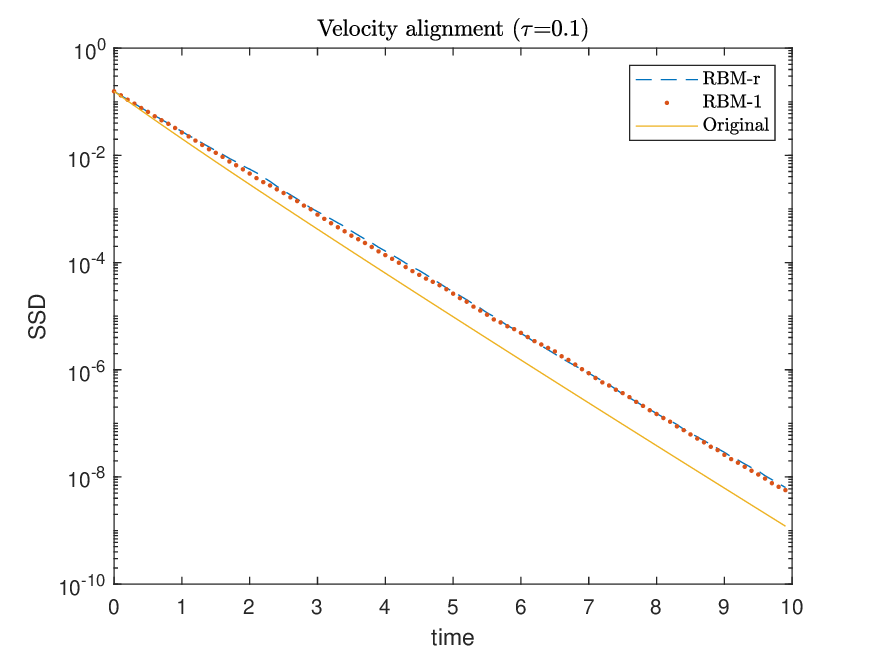}
    \label{fig:flockv0.1}
    \end{minipage}}
    \subfigure[]{
    \begin{minipage}[t]{0.46\textwidth}
        \centering
        \includegraphics[scale=0.46]{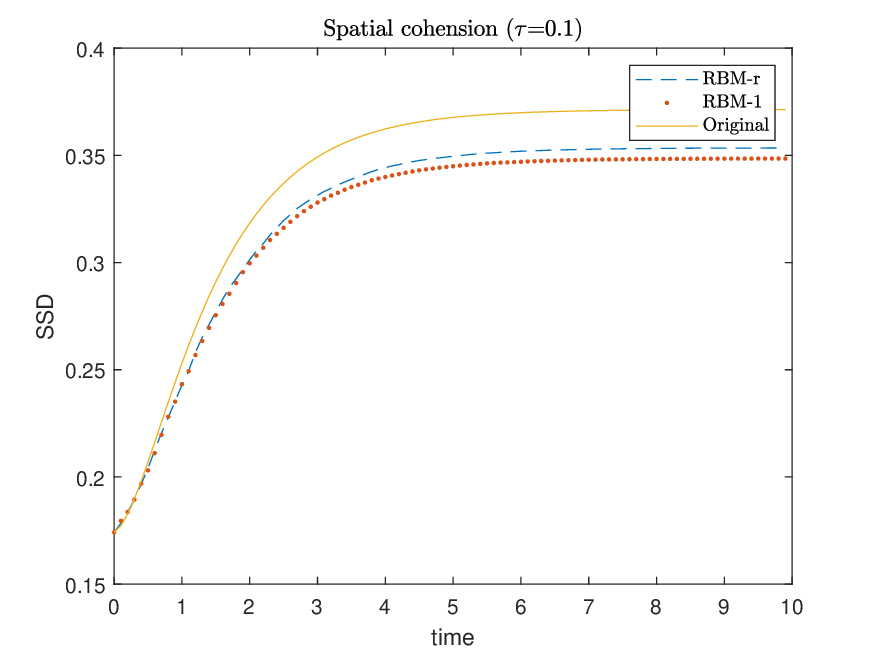}
        \label{fig:flockx0.1} 
        \end{minipage}}
        \caption{(a): The SSD of V (velocities) from 100 simulations at $\tau =0.1$. (b): The SSD of X (positions) from 100 simulations at $\tau =0.1$.}
        \label{fig:flock0.1}
\end{figure}
\begin{figure}[htbp]
    \centering
    \subfigure[]{
    \begin{minipage}[t]{0.46\textwidth}
        \centering
        \includegraphics[scale=0.46]{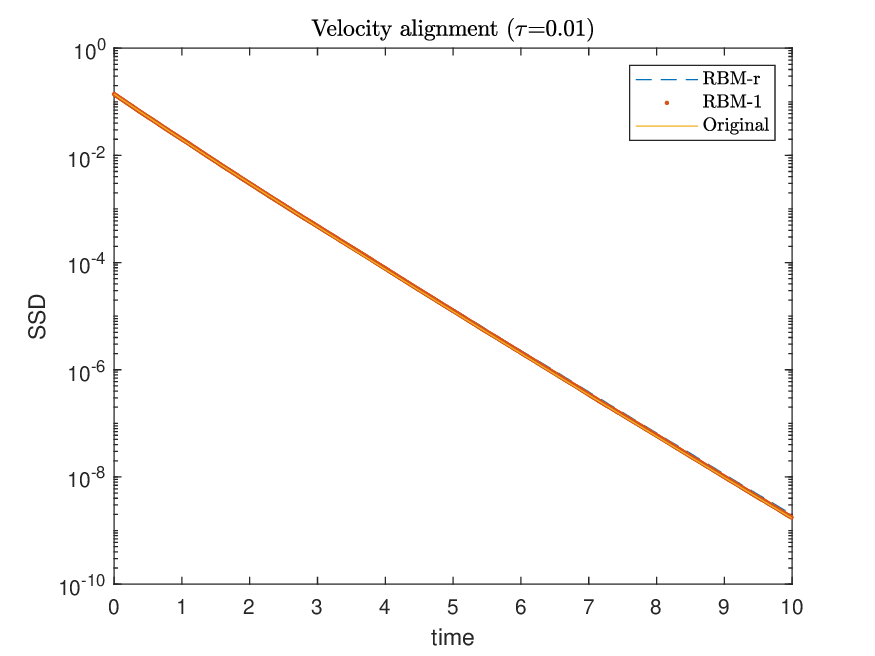}
    \label{fig:flockv0.01}
    \end{minipage}}
    \subfigure[]{
    \begin{minipage}[t]{0.46\textwidth}
        \centering
        \includegraphics[scale=0.46]{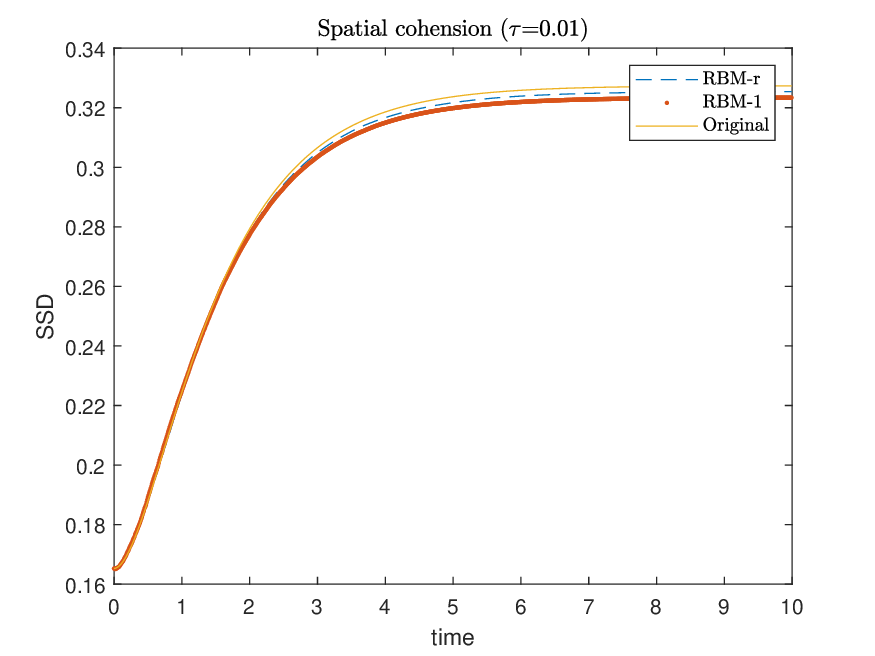}
        \label{fig:flockx0.01} 
        \end{minipage}}
        \caption{(a): The SSD of V (velocities) from 100 simulations at $\tau =0.01$. (b): The SSD of X (positions) from 100 simulations at $\tau =0.01$.}
        \label{fig:flock0.01}
\end{figure}

\paragraph{Dependence of error on the batch size.}
We now focus on numerical simulations related to the dependence on $p$. For simplicity, we set the time step of the Euler method to match the time step used for random batch selections, with $\tau =0.1$. We consider various batch sizes: $p = 2, 4, 8, 16, 32$.

In Figure \ref{fig:vp}, we present the $\ell^2$-errors derived from 1000 random simulations for each value of $p$. The $\ell^2$-error is calculated using the formula:
\begin{equation*}
    \ell^2\text{-error}(t) = \sqrt{\frac{1}{N}\sum\limits_{i=1}^N |\tilde{V}^i(t)-V^i(t)|^2},
\end{equation*}
where $\tilde{V}$ represents each simulated solution. The shaded areas corresponding to each $p$ illustrate the evolution of the $\ell^2$-errors over time.

Figure \ref{fig:vps} displays the scaled error proportional to $\sqrt{1- \frac{p}{N} + \frac{1}{p-1}-\frac{1}{N-1}}$. The scaling factor comes from the second term of \eqref{eq:thmerr}. Notably, the scaled errors for different $p$ values exhibit similar median trends over time. This suggests that the error estimate in \eqref{eq:thmerr} accurately reflects the expected order with respect to $p$.

\begin{figure}[htbp]
    \centering
    \subfigure[]{
    \begin{minipage}[t]{0.46\textwidth}
        \centering
        \includegraphics[scale=0.46]{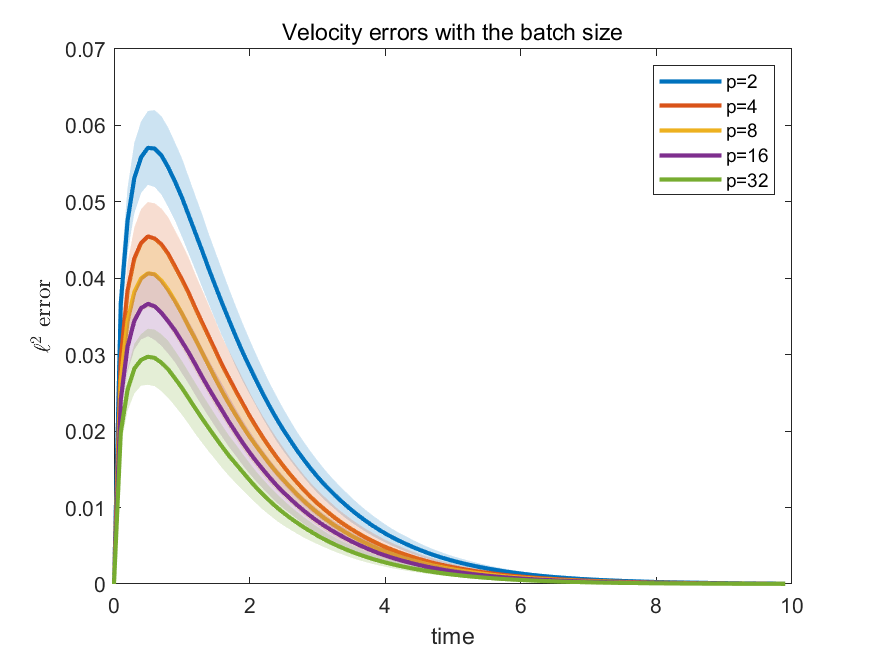}
    \label{fig:vp}
    \end{minipage}}
    \subfigure[]{
    \begin{minipage}[t]{0.46\textwidth}
        \centering
        \includegraphics[scale=0.46]{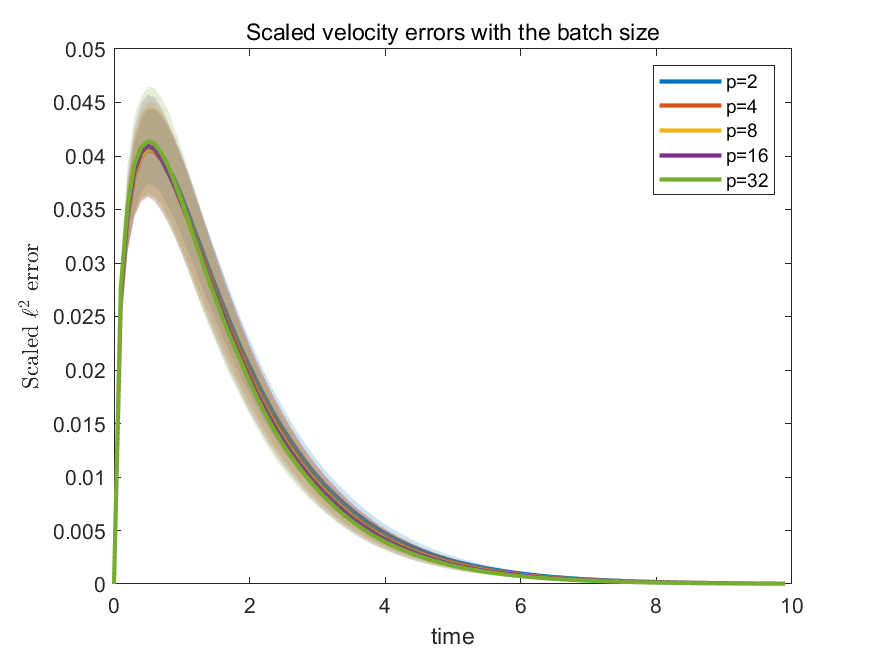}
        \label{fig:vps} 
        \end{minipage}}
        \caption{(RBM-r)(a): The $\ell^2$-errors of velocities from 1000 simulations, computed with different $p$. (b): Scaled error by the term $\sqrt{1- \frac{p}{N} + \frac{1}{p-1}-\frac{1}{N-1}}$. The scaled errors from different $p$ show similar values along time.}
        \label{fig:vp2}
\end{figure}
 We also show the parallel simulation of the RBM-1 in Figure \ref{fig:vp2rbm1}.
\begin{figure}[htbp]
    \centering
    \subfigure[]{
    \begin{minipage}[t]{0.46\textwidth}
        \centering
        \includegraphics[scale=0.46]{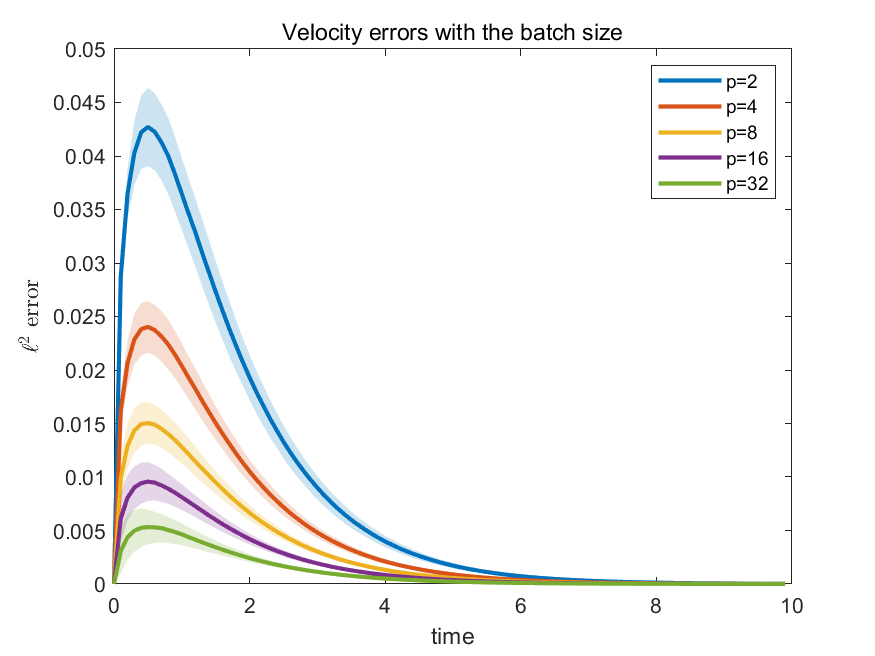}
    \label{fig:vprbm1}
    \end{minipage}}
    \subfigure[]{
    \begin{minipage}[t]{0.46\textwidth}
        \centering
        \includegraphics[scale=0.46]{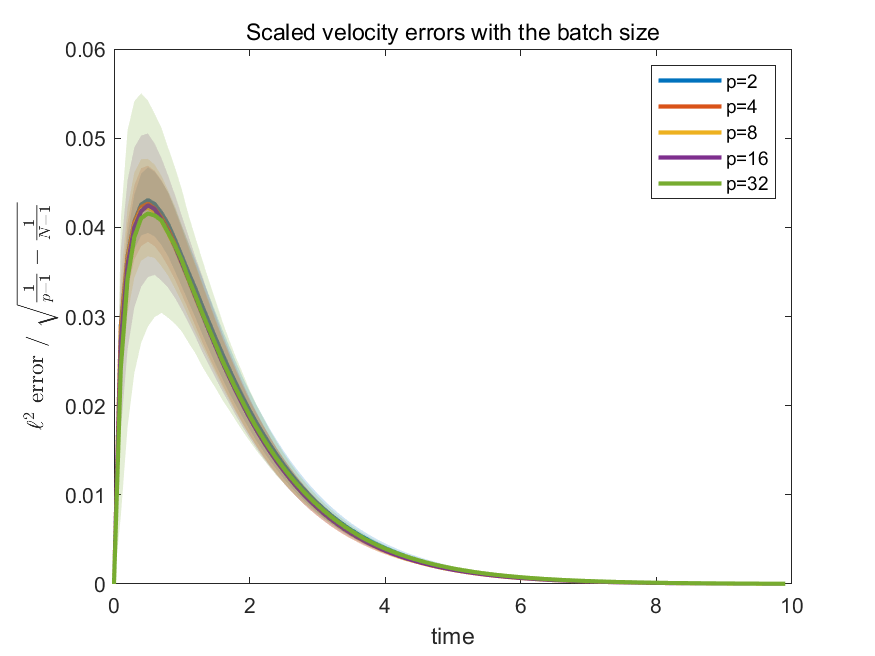}
        \label{fig:vpsrbm1} 
        \end{minipage}}
        \caption{(RBM-1)(a): The $\ell^2$-errors of velocities from 1000 simulations, computed with different $p$. (b): Scaled error by the term $\sqrt{ \ \frac{1}{p-1}-\frac{1}{N-1}}$. The scaled errors from different $p$ show similar values along time.}
        \label{fig:vp2rbm1}
\end{figure}



\paragraph{Dependence of error on the time step.}
Next, we fix $p=2$ but instead test various $\tau.$ To ensure a fair comparison among different $\tau$ values under the same conditions, we set the time step of the Euler method to  $\Delta t = 0.0125$, and we test $\tau = 0.1, 0.05, 0.025, 0.0125$. All other parameters and the graphing methods remain the same as in the previous case for $p$.

In Figure \ref{fig:dt2}, we can observe that the rate at which the error decreases with respect to $\tau$ is approximately on the order of $\sqrt{\tau}$, consistent with our expectations from the error estimate in \eqref{eq:thmerr}. 
Figure \ref{fig:dt2rbm1} shows the corresponding simulation of the RBM-1.

\begin{figure}[htbp]
    \centering
    \subfigure[]{
    \begin{minipage}[t]{0.46\textwidth}
        \centering
        \includegraphics[scale=0.46]{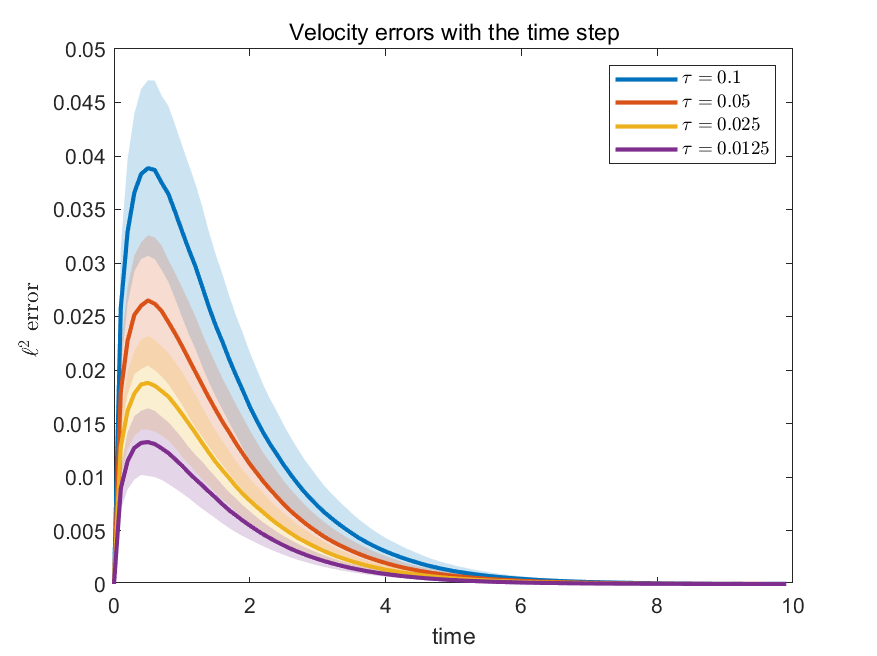}
    \label{fig:dtp}
    \end{minipage}}
    \subfigure[]{
    \begin{minipage}[t]{0.46\textwidth}
        \centering
        \includegraphics[scale=0.46]{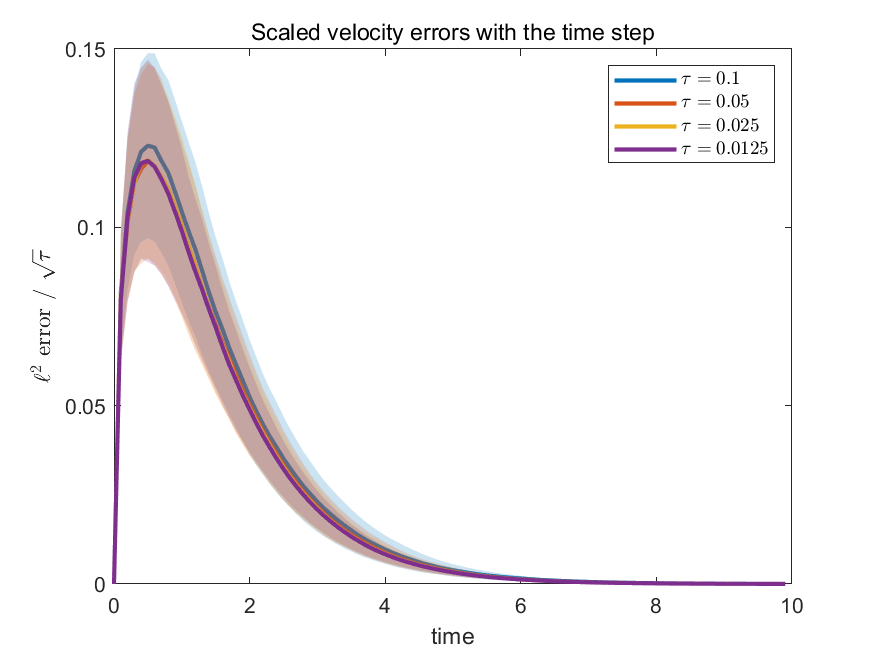}
        \label{fig:dtps} 
        \end{minipage}}
        \caption{(RBM-r)(a): The $\ell^2$-errors of velocities from 1000 simulations, computed with different $\tau$. (b): Scaled error by the term $\sqrt{\tau}$. The scaled errors from different $\tau$ show similar values along time.}
        \label{fig:dt2}
\end{figure}
\begin{figure}[htbp]
    \centering
    \subfigure[]{
    \begin{minipage}[t]{0.46\textwidth}
        \centering
        \includegraphics[scale=0.45]{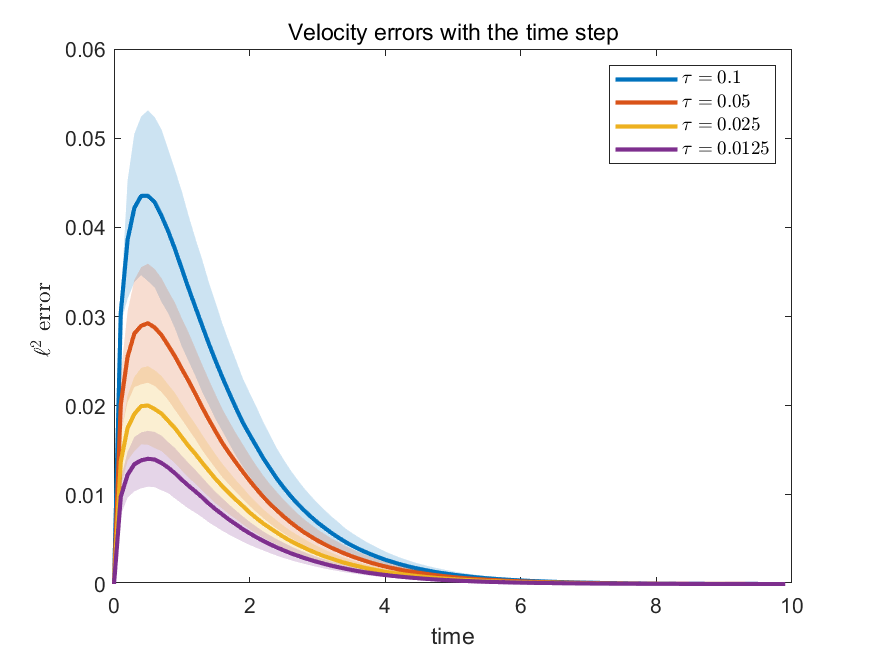}
    \label{fig:dtprbm1}
    \end{minipage}}
    \subfigure[]{
    \begin{minipage}[t]{0.46\textwidth}
        \centering
        \includegraphics[scale=0.45]{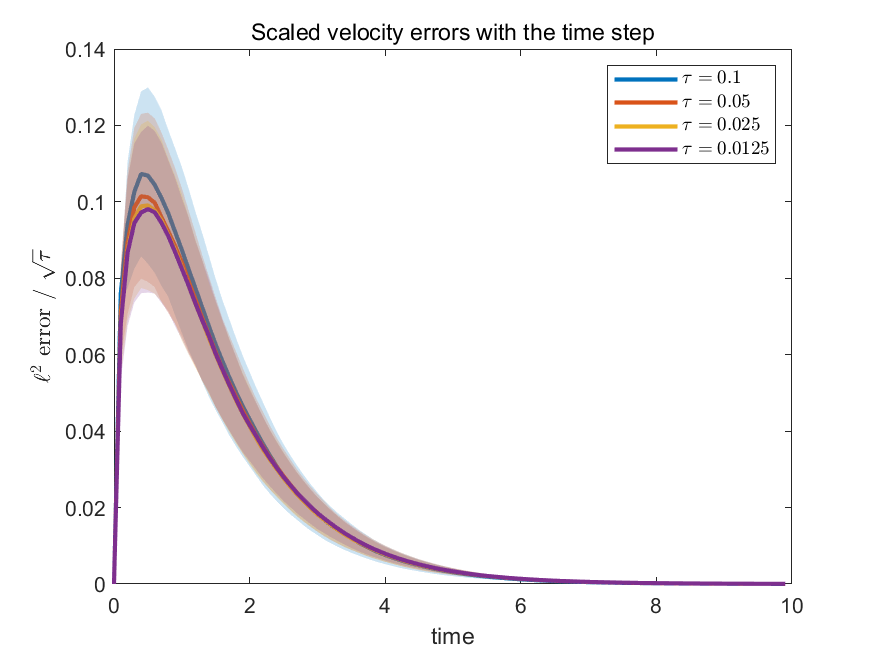}
        \label{fig:dtpsrbm1} 
        \end{minipage}}
        \caption{(RBM-1)(a): The $\ell^2$-errors of velocities from 1000 simulations, computed with different $\tau$. (b): Scaled error by the term $\sqrt{\tau}$. The scaled errors from different $\tau$ show similar values along time.}
        \label{fig:dt2rbm1}
\end{figure}

\paragraph{Conservation of the first moment.}
As a type of kinetic Monte Carlo method, the conservation of the first moment is a significant feature when compared to the Direct Monte Carlo method discussed in \cite{carrillo2019particle}. In that paper, the authors introduced a similar stochastic method with the RBM, referred to as MCgPC, which can be used to approximate stochastic mean-field models of swarming. By ignoring the random interaction kernel in MCgPC, we can rewrite it as Algorithm \ref{algo: dmc}, which represents a direct Monte Carlo method (MC).

The primary distinction between the MC and the RBM-r (or RBM-1) lies in whether they preserve the first moment (also known as momentum when the mass is considered as identity), as demonstrated in Proposition \ref{lem:M}. In Figure \ref{fig:MC}, we compare the first moments of one simulation, calculated by both the MC and the RBM-r with $N=2^6,$ $p=2$, $\tau=0.1$. Although both methods reach a balance after some time, only the RBM-r successfully preserves the first moment.

\begin{algorithm}\label{algo: dmc}
\caption{ The MC for \eqref{eq: cs}}
\For{$k=1$ to $\frac{T}{\tau}$}
{\For{$i=1$ to $N$}
    {Sample $p-1$ particles $j_1,\cdots,j_{p-1}$ uniformly without repetition among all particles.
    Update $(X^i,V^i)$ by solving 
    \begin{equation*}
        \left\{\begin{aligned}
            \partial_t X^i(t) =& V^i(t),\\
            \partial_t V^i(t) =& \frac{\kappa}{p-1}\sum\limits_{\ell =1}^{p-1}
            \psi(|X^j(t)-X^i(t)|)(V^j(t)-V^i(t)),
        \end{aligned}\right.
    \end{equation*}
    for $t\in [t_k, t_{k+1}).$ 
    }}
\end{algorithm}
\paragraph{Comparison with RBM-1.} 

In Figure \ref{fig:rbm1}, we present the $\ell^2$-errors derived from 100 simulations using the RBM-r approximation (Algorithm \ref{algo: the RBMr'}) and the RBM-1 approximation (Algorithm \ref{algo: the RBM1}). Solid lines represent the RBM-r approximation, while dashed lines with circles denote the RBM-1 approximation. We set the same time step $\tau=0.1$ and batch size $p=2.$ Various colors indicate the simulations with $N=2^4,2^6,2^8,2^{10}.$ Lines of the same color correspond to simulations with the same number of particles $N$.

Figure \ref{fig:rbm1} illustrates that the performance of the RBM-r is not superior to that of the RBM-1 due to the allowance for replacements, which aligns with the error estimates in our main theorem and the trajectories shown in Figure \ref{fig:traj}.

\begin{figure}
  \begin{minipage}[t]{0.45\linewidth}
    \centering
    \includegraphics[scale=0.46]{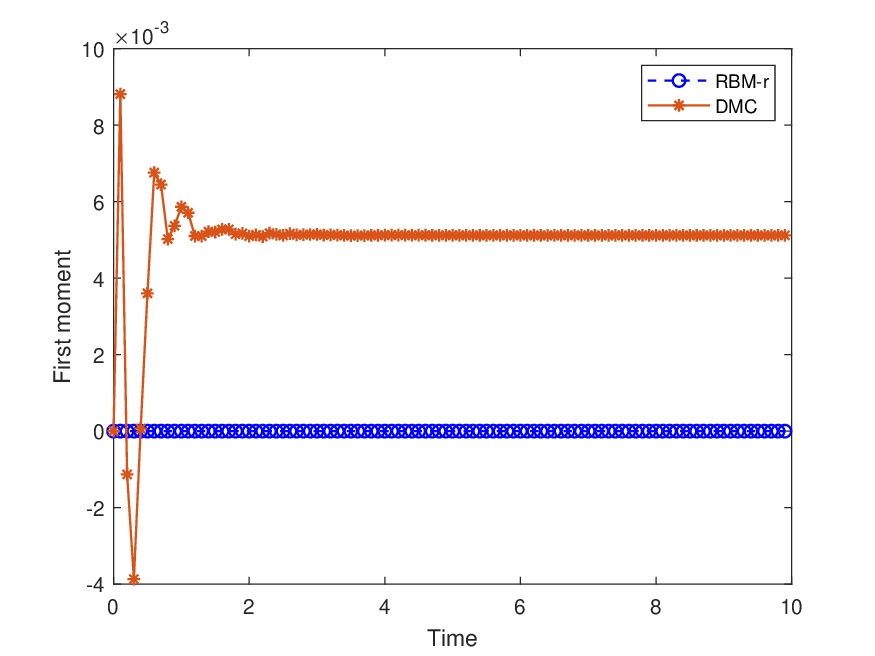}
    \caption{The first moment simulated by the MC and the RBM-r respectively. The dashed lines with circles represent the RBM-r approximation, while the solid lines with stars denote the MC approximation.}
    \label{fig:MC}
  \end{minipage}%
  \hspace{3mm}
  \begin{minipage}[t]{0.45\linewidth}
    \centering
    \includegraphics[scale=0.46]{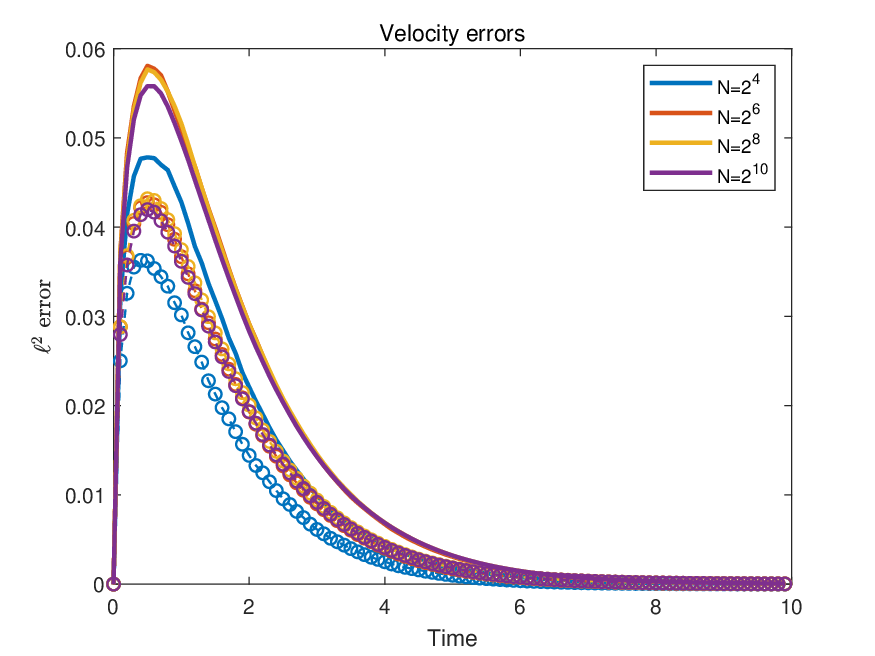}
    \caption{The $\ell^2$-errors are simulated for the RBM-r and RBM-1 methods, respectively. Here solid lines represent the RBM-r approximation, while the dashed lines with circles denote the RBM-1 approximation.}
    \label{fig:rbm1}
    \hspace{3mm}
  \end{minipage}
\end{figure}

\section{Conclusion}\label{sec:conclusion}

In this paper,  we analyzed the RBM approximations for the deterministic Cucker-Smale model. We provide an improved stochastic flocking analysis and the uniform-in-time error estimate independent of $N$ for the random batch method with and without replacement applied to the Cucker-Smale model. Our theoretical error estimates are further validated through numerical simulations. For future work, it would be interesting to consider the error estimate of the general consensus model. 

\section*{Acknowledgement}
We thank Prof. Lei Li and Prof. Yiwen Lin for their helps on the manuscript. 
This work is supported by the NSFC grants (No. 12031013, 12350710181), and the Shanghai Municipal Science and Technology Key Project (No. 22JC1402300).

\begin{appendix}
    
    \section{Proof of Proposition \ref{ineq: coarseD}.}\label{app:A}
    The proof is a variant of the Lemma 2.4 the for the RBM-1) in \cite{ha2021uniform}.
    \begin{proof}[Proof of Proposition \ref{ineq: coarseD}]
        For the first assertion, we claim that the relative velocities are non-increasing in time. Let $t\in [t_{m-1}, t_m)$ be given. Then one can choose time-dependent indices $k$ and $\ell$ such that 
        \begin{equation*}
            |\tilde{V}^k (t) - \tilde{V}^\ell (t)| = \mathcal{D}_{\tilde{V}}(t).
        \end{equation*}
        \paragraph{Case 1.} If both $k,\ell \in \mathcal{C}_m$, then one has 
        \begin{equation}\label{eq:A1}
        \begin{aligned}
            \frac{d}{dt} | \tilde{V}^k(t) -\tilde{V}^\ell(t)|^2 = &
            2(\tilde{V}^k(t) - \tilde{V}^\ell (t)) \cdot\frac{d}{dt}(\tilde{V}^k(t) - \tilde{V}^\ell (t))\\
            = & \frac{2}{p-1} \sum\limits_{j\in \mathcal{C}_m}\psi (|\tilde{X}^j - \tilde{X}^k|)(\tilde{V}^j - \tilde{V}^k)\cdot(\tilde{V}^k - \tilde{V}^\ell)\\
            &- \frac{2}{p-1}\sum\limits_{j\in\mathcal{C}_m}\psi (|\tilde{X}^j - \tilde{X}^\ell|)(\tilde{V}^j - \tilde{V}^\ell)\cdot(\tilde{V}^k - \tilde{V}^\ell).
        \end{aligned}
        \end{equation}
        In order to show that the right-hand side of \eqref{eq:A1} is not positive, we use the maximality of $| \tilde{V}^k(t) -\tilde{V}^\ell(t)|$ at time $t.$ Since
        \begin{equation*}
            | \tilde{V}^k(t) -\tilde{V}^\ell(t)| \ge | \tilde{V}^j(t) -\tilde{V}^\ell(t)|, \quad j = 1,\cdots,N,
        \end{equation*}
        one has
        \begin{multline}\label{eq:A2}
            (\tilde{V}^j(t)-\tilde{V}^k (t))\cdot(\tilde{V}^k(t)-\tilde{V}^\ell (t)) =
            - ((\tilde{V}^k-\tilde{V}^\ell) - (\tilde{V}^j-\tilde{V}^\ell))\cdot (\tilde{V}^k - \tilde{V}^\ell)\\
            \le - | \tilde{V}^k(t) -\tilde{V}^\ell(t)|^2 + | \tilde{V}^j(t) -\tilde{V}^\ell(t)|| \tilde{V}^k(t) -\tilde{V}^\ell(t)|\le 0, \quad j= 1,\cdots,N.
        \end{multline}
        Similarly, one has 
        \begin{equation*}
            (\tilde{V}^j(t)-\tilde{V}^\ell (t))\cdot (\tilde{V}^k(t)-\tilde{V}^\ell (t)) \ge 0, \quad j=1,\cdots,N.
        \end{equation*}

        \paragraph{Case 2.} If $k\in\mathcal{C}_m$ but $\ell\notin\mathcal{C}_m,$ then
        \begin{equation*}
            \begin{aligned}
                \frac{d}{dt} | \tilde{V}^k(t) -\tilde{V}^\ell(t)|^2 = &
                2(\tilde{V}^k(t) - \tilde{V}^\ell (t)) \cdot\frac{d}{dt}(\tilde{V}^k(t) - \tilde{V}^\ell (t))\\
                = & - \frac{2}{p-1}\sum\limits_{j\in\mathcal{C}_m}\psi (|\tilde{X}^j - \tilde{X}^\ell|)(\tilde{V}^j - \tilde{V}^\ell)\cdot(\tilde{V}^k - \tilde{V}^\ell)\\
                \le&0,
            \end{aligned}
        \end{equation*}
        by \eqref{eq:A2}.

        \paragraph{Case 3.} If $k,\ell\notin\mathcal{C}_m$, then
        \begin{equation*}
            \frac{d}{dt} | \tilde{V}^k(t) -\tilde{V}^\ell(t)|^2 = 0.
        \end{equation*}
        Generalizing the above three cases, we finish the proof.

    \end{proof}

    \section{Proof of Lemma \ref{lem:gronwalltype}.}\label{app:gronwall}

    \begin{proof}
    We basically repeat the same arguments appearing in Lemma 3.9 of \cite{ha2018local} using a bootstrapping argument. First, we will show that the uniform bound of $\mathcal{V}$, and then we use this uniform bound to derive the exponential decay of $\mathcal{V}$ in two steps.

\paragraph{Step A (Uniform boundedness of $\mathcal{V}$. )}
    In this step, we derive $\mathcal{V}(t) \le A_{\infty}(\gamma)\|f\|_{L^1}.$ For this, we set a maximal function $\mathcal{M}_\nu$ :
$$
\mathcal{M}_{\mathcal{V}}(t):=\max _{u \in[0, t]} \mathcal{V}(u), \quad t>0
$$

Next, we will show that
\begin{equation*}
    \mathcal{M}_{\mathcal{V}}(t) \le A_\infty(\gamma)\|f\|_{L^1}, \quad t \geq 0,
\end{equation*}

where $A_{\infty}(\gamma)=e^{\gamma \int_0^{\infty}} s e^{-\alpha s} d s$.
It follows from the first differential inequality that

$$
\sqrt{\mathcal{X}(t)} \le \sqrt{\mathcal{X}(0)}+\int_0^t \sqrt{\mathcal{V}(s)} ds.
$$

We substitute this into the second differential inequality to get

$$
\begin{aligned}
\frac{d \mathcal{V}}{d t} & \le-\alpha \mathcal{V}+\gamma e^{-\beta t} \mathcal{X}+f \le-\alpha \mathcal{V}+\gamma e^{-\beta t}\int_0^t \mathcal{V}(s) ds+f \\
& \le \gamma e^{-\beta t}\int_0^t \mathcal{V}(s) ds+f.
\end{aligned}
$$

Then, we integrate the above inequality to obtain

$$
\begin{aligned}
\mathcal{V}(u) & \le \mathcal{V}(0)+\int_0^u f(s) d s+\gamma \int_0^u e^{-\beta s}\int_0^s \mathcal{V}(u) d u\, d s \\
& \le \|f\|_{L^1}+\gamma \int_0^u e^{-\beta s}\int_0^s \mathcal{V}(u) d u\, d s , \quad u \le t.
\end{aligned}
$$

This implies
\begin{equation*}
    \mathcal{M}_{\mathcal{V}}(t) \le\|f\|_{L^1}+\gamma \int_0^t s e^{-\beta s} \mathcal{M}_{\mathcal{V}}(s) d s.
\end{equation*}
We set
\begin{equation*}
Z(t):=\|f\|_{L^1}+\gamma \int_0^t s e^{-\beta s} \mathcal{M}_{\mathcal{V}}(s) d s.
\end{equation*}
Then, it's clear that
$$
\mathcal{M}_{\mathcal{V}}(t) \le Z(t).
$$
We differentiate $Z(t)$ obtain

$$
\dot{Z}(t)=\gamma t e^{-\beta t} \mathcal{M}_{\mathcal{V}}(t) \le \gamma t e^{-\beta t} Z(t).
$$
This yields
\begin{equation*}
    Z(t) \le Z(0) e^{\gamma \int_0^t s e^{-\beta s} d s}\le A_{\infty}(\gamma)\,\|f\|_{L^1},
\end{equation*}
where $A_{\infty}(\gamma):=e^{\gamma \int_0^{\infty} s e^{-\beta s} d s}$.
Then, it holds that
$$
\mathcal{V}(t) \le \mathcal{M}_{\mathcal{V}}(t) \le Z(t) \le A_{\infty}(\gamma)\,\|f\|_{L^1}.
$$

\paragraph{Step B (Decay estimate of $\mathcal{V}(t)$. )}
We recall the original dynamics of $\mathcal{V}$ and note that
$$
\max _{0 \le t<\infty} t e^{-ct}=\frac{1}{c e}, \quad \max _{0 \le t<\infty} t^2 e^{-ct}=\frac{4}{e^2}
$$
for any $c>0$. Then we obtain
\begin{align*}
    \mathcal{V}(t) & \le \gamma\e^{-\alpha t} \int_0^t \e^{(\alpha-\beta)s}s A_\infty(\gamma)\|f\|_{L^1} ds + \int_0^t \e^{-\alpha(t-s)}f(s)\\
    &\le C \e^{-\frac{\alpha \wedge\beta}{2}t} \,\|f\|_{L^1} + \frac{1}{\alpha}f(\frac{t}{2}).
\end{align*}

\paragraph{Step C (Boundedness of $\mathcal{X}(t)$. )}
By the inequality
\begin{equation*}
    \frac{d}{dt}\sqrt{\mathcal{X}} \le \sqrt{\mathcal{V}},
\end{equation*}
it holds that
\begin{equation*}
    \mathcal{X}(t) \le \int_0^t \mathcal{V}(s)ds \le C\|f\|_{L^1}.
\end{equation*}
\end{proof}

\end{appendix}

\bibliographystyle{plain}
\bibliography{main}

\begin{thebibliography}{10}

\bibitem{albi2019vehicular}
Giacomo Albi, Nicola Bellomo, Luisa Fermo, Seung-Yeal Ha, Jeongho Kim, Lorenzo Pareschi, David Poyato, and Juan Soler.
\newblock Vehicular traffic, crowds, and swarms: From kinetic theory and multiscale methods to applications and research perspectives.
\newblock {\em Mathematical Models and Methods in Applied Sciences}, 29(10):1901--2005, 2019.

\bibitem{albi2013binary}
Giacomo Albi and Lorenzo Pareschi.
\newblock Binary interaction algorithms for the simulation of flocking and swarming dynamics.
\newblock {\em Multiscale Modeling \& Simulation}, 11(1):1--29, 2013.

\bibitem{bellomo2017quest}
Nicola Bellomo and Seung-Yeal Ha.
\newblock A quest toward a mathematical theory of the dynamics of swarms.
\newblock {\em Mathematical Models and Methods in Applied Sciences}, 27(04):745--770, 2017.

\bibitem{bellomo2012mathematical}
Nicola Bellomo and J~Soler.
\newblock On the mathematical theory of the dynamics of swarms viewed as complex systems.
\newblock {\em Mathematical Models and Methods in Applied Sciences}, 22(supp01):1140006, 2012.

\bibitem{bortz1975new}
Alfred~B Bortz, Malvin~H Kalos, and Joel~L Lebowitz.
\newblock {A new algorithm for Monte Carlo simulation of Ising spin systems}.
\newblock {\em Journal of Computational physics}, 17(1):10--18, 1975.

\bibitem{buck1966biology}
John Buck and Elisabeth Buck.
\newblock Biology of synchronous flashing of fireflies, 1966.

\bibitem{cai2024convergence}
Zhenhao Cai, Jian-Guo Liu, and Yuliang Wang.
\newblock {Convergence of Random Batch Method with replacement for interacting particle systems}.
\newblock {\em arXiv preprint arXiv:2407.19315}, 2024.

\bibitem{carrillo2019particle}
Jos{\'e}~A Carrillo, Lorenzo Pareschi, Mattia Zanella, et~al.
\newblock {Particle based gPC methods for mean-field models of swarming with uncertainty}.
\newblock {\em Comuunications in Computational Physics}, 25(2):508--531, 2019.

\bibitem{cucker2007emergent}
Felipe Cucker and Steve Smale.
\newblock Emergent behavior in flocks.
\newblock {\em IEEE Transactions on automatic control}, 52(5):852--862, 2007.

\bibitem{degond2008continuum}
Pierre Degond and S{\'e}bastien Motsch.
\newblock Continuum limit of self-driven particles with orientation interaction.
\newblock {\em Mathematical Models and Methods in Applied Sciences}, 18(supp01):1193--1215, 2008.

\bibitem{dong2020stochastic}
Jiu-Gang Dong, Seung-Yeal Ha, Jinwook Jung, and Doheon Kim.
\newblock On the stochastic flocking of the cucker--smale flock with randomly switching topologies.
\newblock {\em SIAM Journal on Control and Optimization}, 58(4):2332--2353, 2020.

\bibitem{dorfler2014synchronization}
Florian D{\"o}rfler and Francesco Bullo.
\newblock Synchronization in complex networks of phase oscillators: A survey.
\newblock {\em Automatica}, 50(6):1539--1564, 2014.

\bibitem{gao2024rbmd}
Weihang Gao, Teng Zhao, Yongfa Guo, Jiuyang Liang, Huan Liu, Maoying Luo, Zedong Luo, Wei Qin, Yichao Wang, Qi~Zhou, et~al.
\newblock Rbmd: A molecular dynamics package enabling to simulate 10 million all-atom particles in a single graphics processing unit.
\newblock {\em arXiv preprint arXiv:2407.09315}, 2024.

\bibitem{ha2018local}
Seung-Yeal Ha and Shi Jin.
\newblock {Local sensitivity analysis for the Cucker-Smale model with random inputs}.
\newblock {\em Kinetic and Related Models}, 11(4):859--889, 2018.

\bibitem{ha2021convergence}
Seung-Yeal Ha, Shi Jin, Doheon Kim, and Dongnam Ko.
\newblock Convergence toward equilibrium of the first-order consensus model with random batch interactions.
\newblock {\em Journal of Differential Equations}, 302:585--616, 2021.

\bibitem{ha2021uniform}
Seung-Yeal Ha, Shi Jin, Doheon Kim, and Dongnam Ko.
\newblock {Uniform-in-time error estimate of the random batch method for the Cucker--Smale model}.
\newblock {\em Mathematical Models and Methods in Applied Sciences}, 31(06):1099--1135, 2021.

\bibitem{ha2016collective}
Seung-Yeal Ha, Dongnam Ko, Jinyeong Park, and Xiongtao Zhang.
\newblock Collective synchronization of classical and quantum oscillators.
\newblock {\em EMS Surveys in Mathematical Sciences}, 3(2):209--267, 2016.

\bibitem{ha2009simple}
Seung-Yeal Ha and Jian-Guo Liu.
\newblock {A simple proof of the Cucker-Smale flocking dynamics and mean-field limit}.
\newblock {\em Communications in Mathematical Sciences}, 7(2):297--325, 2009.

\bibitem{Ha2008particle}
Seung-Yeal Ha and Eitan Tadmor.
\newblock From particle to kinetic and hydrodynamic descriptions of flocking.
\newblock {\em Kinetic and Related Models}, 1(3):415--435, 2008.

\bibitem{jin2021randomreview}
Shi Jin and Lei Li.
\newblock Random batch methods for classical and quantum interacting particle systems and statistical samplings.
\newblock In {\em Active Particles, Volume 3: Advances in Theory, Models, and Applications}, pages 153--200. Springer, 2021.

\bibitem{jin2020random}
Shi Jin, Lei Li, and Jian-Guo Liu.
\newblock {Random batch methods (RBM) for interacting particle systems}.
\newblock {\em Journal of Computational Physics}, 400:108877, 2020.

\bibitem{jin2022random}
Shi Jin, Lei Li, and Yiqun Sun.
\newblock {On the Random Batch Method for second order interacting particle systems}.
\newblock {\em Multiscale Modeling \& Simulation}, 20(2):741--768, 2022.

\bibitem{jin2020randomQMC}
Shi Jin and Xiantao Li.
\newblock Random batch algorithms for quantum monte carlo simulations.
\newblock {\em Communications in Computational Physics}, 28(5):1907--1936, 2020.

\bibitem{justh2002simple}
Eric~W Justh and Perinkulam~S Krishnaprasad.
\newblock {A simple control law for UAV formation flying}.
\newblock Technical report, Technical Report 2002-38, Institute for Systems Research, 2002.

\bibitem{ko2021uniform}
Dongnam Ko, Seung-Yeal Ha, Shi Jin, and Doheon Kim.
\newblock Uniform error estimates for the random batch method to the first-order consensus models with antisymmetric interaction kernels.
\newblock {\em Studies in Applied Mathematics}, 146(4):983--1022, 2021.

\bibitem{kuramoto1975self}
Yoshiki Kuramoto.
\newblock Self-entrainment of a population of coupled non-linear oscillators.
\newblock In {\em International Symposium on Mathematical Problems in Theoretical Physics: January 23--29, 1975, Kyoto University, Kyoto/Japan}, pages 420--422. Springer, 1975.

\bibitem{li2020random}
Lei Li, Zhenli Xu, and Yue Zhao.
\newblock A random-batch monte carlo method for many-body systems with singular kernels.
\newblock {\em SIAM Journal on Scientific Computing}, 42(3):A1486--A1509, 2020.

\bibitem{liang2021random}
Jiuyang Liang, Zhenli Xu, and Yue Zhao.
\newblock Random-batch list algorithm for short-range molecular dynamics simulations.
\newblock {\em The Journal of Chemical Physics}, 155(4), 2021.

\bibitem{motsch2011new}
Sebastien Motsch and Eitan Tadmor.
\newblock A new model for self-organized dynamics and its flocking behavior.
\newblock {\em Journal of Statistical Physics}, 144:923--947, 2011.

\bibitem{peskin1975mathematical}
Charles~S Peskin.
\newblock Mathematical aspects of heart physiology.
\newblock {\em Courant Inst. Math}, 1975.

\bibitem{toner1998flocks}
John Toner and Yuhai Tu.
\newblock Flocks, herds, and schools: A quantitative theory of flocking.
\newblock {\em Physical review E}, 58(4):4828, 1998.

\bibitem{topaz2004swarming}
Chad~M Topaz and Andrea~L Bertozzi.
\newblock Swarming patterns in a two-dimensional kinematic model for biological groups.
\newblock {\em SIAM Journal on Applied Mathematics}, 65(1):152--174, 2004.

\bibitem{vicsek1995novel}
Tam{\'a}s Vicsek, Andr{\'a}s Czir{\'o}k, Eshel Ben-Jacob, Inon Cohen, and Ofer Shochet.
\newblock Novel type of phase transition in a system of self-driven particles.
\newblock {\em Physical review letters}, 75(6):1226, 1995.

\bibitem{voter2007introduction}
Arthur~F Voter.
\newblock {Introduction to the kinetic Monte Carlo method}.
\newblock In {\em Radiation effects in solids}, pages 1--23. Springer, 2007.

\end{thebibliography}

\end{document}